\documentclass[a4paper,11pt]{amsart}
\usepackage{amsmath, amsfonts, amssymb, amsthm, amscd, mathrsfs}
\usepackage[colorlinks=true, linkcolor=blue, citecolor=blue]{hyperref}
\usepackage{graphicx}
\usepackage{enumerate,enumitem}
\usepackage{url}
\usepackage{color}
\usepackage[utf8]{inputenc}
\usepackage{tikz}
\usetikzlibrary{arrows}
\usetikzlibrary{patterns} 
\usepackage[a4paper,scale={0.72,0.74},marginratio={1:1},footskip=7mm,headsep=10mm]{geometry}

\numberwithin{equation}{section} 

\newtheorem{theorem}{Theorem}[section]
\newtheorem{lemma}[theorem]{Lemma}
\newtheorem{proposition}[theorem]{Proposition}

\newtheorem{remark}[theorem]{Remark}

\newcommand{\bbE}{{\ensuremath{\mathbb E}} }

\newcommand{\bbN}{{\ensuremath{\mathbb N}} }

\newcommand{\bbP}{{\ensuremath{\mathbb P}} }

\newcommand{\bbR}{{\ensuremath{\mathbb R}} }

\newcommand{\bbZ}{{\ensuremath{\mathbb Z}} }

\newcommand{\cA}{{\ensuremath{\mathcal A}} }
\newcommand{\cB}{{\ensuremath{\mathcal B}} }

\newcommand{\cE}{{\ensuremath{\mathcal E}} }
\newcommand{\cF}{{\ensuremath{\mathcal F}} }

\newcommand{\cH}{{\ensuremath{\mathcal H}} }

\newcommand{\cN}{{\ensuremath{\mathcal N}} }

\newcommand{\cP}{{\ensuremath{\mathcal P}} }

\newcommand{\cR}{{\ensuremath{\mathcal R}} }
\newcommand{\cS}{{\ensuremath{\mathcal S}} }
\newcommand{\cT}{{\ensuremath{\mathcal T}} }

\newcommand{\ga}{\alpha}
\newcommand{\gb}{\beta}
\newcommand{\gga}{\gamma}

\newcommand{\gd}{\delta}
\newcommand{\gD}{\Delta}
\newcommand{\gep}{\varepsilon} 
\newcommand{\gz}{\zeta}
\newcommand{\gt}{\theta}
\newcommand{\gk}{\kappa}
\newcommand{\gl}{\lambda}
\newcommand{\gL}{\Lambda}

\newcommand{\gs}{\sigma}

\newcommand{\go}{\omega}
\newcommand{\gO}{\Omega}
\newcommand{\ub}{\underline}
\newcommand{\petit}{{\sf{p}}}

\renewcommand{\tilde}{\widetilde}          % wider `tilde'
\DeclareMathSymbol{\leqslant}{\mathalpha}{AMSa}{"36} % nicer `smaller or equal'
\DeclareMathSymbol{\geqslant}{\mathalpha}{AMSa}{"3E} % nicer `larger or equal'
\DeclareMathSymbol{\eset}{\mathalpha}{AMSb}{"3F}     % nicer `emptyset'
\newcommand{\dd}{\text{\rm d}}             % a straight d for differentials
\newcommand{\sumtwo}[2]{\sum_{\substack{#1 \\ #2}}} % sum with 2 lines
\newcommand{\R}{\mathbb{R}}

\newcommand{\Z}{\mathbb{Z}}
\newcommand{\N}{\mathbb{N}}

\newcommand{\PEfont}{\mathrm}
\DeclareMathOperator{\var}{\ensuremath{\PEfont Var}}

\newcommand\bP{\ensuremath{\mathrm{P}}}
\newcommand\bE{\ensuremath{\mathrm{E}}}

\DeclareMathOperator\argmax{argmax}
\DeclareMathOperator\argmin{argmin}

\newcommand{\ind}{{\sf 1}}

\newcommand{\card}{\mathrm{card}}

\newcommand{\8}{\infty}

\newenvironment{myenumerate}{
\renewcommand{\theenumi}{\arabic{enumi}}
\renewcommand{\labelenumi}{{\rm(\theenumi)}}
\begin{list}{\labelenumi}
{
\setlength{\itemsep}{0.4em}
\setlength{\topsep}{0.5em}
\setlength\leftmargin{2.45em}
\setlength\labelwidth{2.05em}
\setlength{\labelsep}{0.4em}
\usecounter{enumi}
}
}
{\end{list}
}

\renewenvironment{enumerate}{
\begin{myenumerate}}
{\end{myenumerate}}

\newcommand{\beq}{\begin{equation}}
\newcommand{\eeq}{\end{equation}}
\newcommand{\ba}{\begin{aligned}}
\newcommand{\ea}{\end{aligned}}

\newcommand{\bin}{\mathrm{Bin}}

\newcommand{\lopt}{\ell_0}
\newcommand{\kopt}{{k_0}}
\newcommand{\ins}{{\rm in}}
\newcommand{\out}{{\rm out}}
\newcommand{\loc}{{\rm loc}}
\newcommand{\obs}{{\rm obs}}
\newcommand{\good}{{\rm Good}}
\newcommand{\tone}{\mathsf{T}_1}
\newcommand{\ttwo}{\mathsf{T}_2}

\newcommand{\pin}{{\rm pin}}
\newcommand{\Inw}{{\rm Inw}}
\newcommand{\Out}{{\rm Out}}
\newcommand{\chop}{\mathsf{m}}

\newcommand{\kBone}{B^{(1)}}
\newcommand{\kBtwo}{B^{(2)}}
\newcommand{\kBthree}{B^{(3)}}
\newcommand{\kBfour}{B^{(4)}}
\newcommand{\kBfive}{B^{(5)}}
\newcommand{\mZ}{\mathsf Z}
\newcommand{\mK}{\mathsf K}
\newcommand{\mtau}{\overline{\tau}}
\newcommand{\bgs}{\boldsymbol{\gs}}

\begin{document}
\date{\today}
\title[]{Localization of a one-dimensional simple random walk among power-law renewal obstacles}
\author[J. Poisat]{Julien Poisat}
\address[J. Poisat]{Universit\'e Paris-Dauphine, CNRS, UMR [7534], CEREMADE, PSL Research University, 75016 Paris, France}
\email{poisat@ceremade.dauphine.fr}

\author[F. Simenhaus]{Fran\c cois Simenhaus}
\address[F. Simenhaus]{Universit\'e Paris-Dauphine, CNRS, UMR [7534], CEREMADE, 75016 Paris, France}
\email{simenhaus@ceremade.dauphine.fr}

\thanks{\textit{Acknowledgements}. JP acknowledges the support of the ANR-17-CE40-0032 grant SWiWS and the ANR-22-CE40-0014 grant LOCAL}


\begin{abstract} 
We consider a one-dimensional simple random walk killed by quenched soft obstacles. The position of the obstacles is drawn according to a renewal process with a power-law increment distribution. In a previous work, we computed the large-time asymptotics of the quenched survival probability. In the present work we continue our study by describing the behaviour of the random walk conditioned to survive. We prove that with large probability, the walk quickly reaches a unique time-dependent optimal gap that is free from obstacles and gets localized there. We actually establish a dichotomy. If the renewal tail exponent is smaller than one then the walk hits the optimal gap and spends all of its remaining time inside, up to finitely many visits to the bottom of the gap.  If the renewal tail exponent is larger than one then the random walk spends most of its time inside of the optimal gap but also performs short outward excursions, for which we provide matching upper and lower bounds on their length and cardinality. Our key tools include a Markov renewal interpretation of the survival probability as well as various comparison arguments for obstacle environments. Our results may also be rephrased in terms of localization properties for a directed polymer among multiple repulsive interfaces.
\end{abstract}

%%%%%%%%%%%%%%%%%%%%%%%%%%%%%%%%%%%%%%%%%%%%%%
%% Please use \tableofcontents for articles %%
%% with 50 pages and more                   %%
%%%%%%%%%%%%%%%%%%%%%%%%%%%%%%%%%%%%%%%%%%%%%%
\maketitle
\setcounter{tocdepth}{3}
\tableofcontents

%%%%%%%%%%%%%%%%%%%%%%%%%%%%%%%%%%%%%%%%%%%%%%
%%%% Main text entry area:

\section*{Introduction}

When put under a certain type of constraint, a random walk may be localized in an atypically small region of space, in contrast to the unconstrained diffusive behaviour. This phenomenon can be observed in various models and under different forms. Let us mention for instance the collapse transition of a polymer in a poor solvent~\cite{CNPT18}, the pinning of a polymer on a defect line~\cite{Gia07,Gia11}, as well as in the parabolic Anderson model~\cite{Ko16}, localization of a polymer in a heterogeneous medium~\cite{Comets-book}, confinement of random walks among obstacles~\cite{Sz98}. Such models, which are often motivated by Biology, Chemistry or Physics, offer challenging mathematical problems and have been an active field of research.\\

In this paper we consider a discrete-time one-dimensional simple random walk on $\Z$ among quenched random obstacles drawn from a renewal process. Equivalently, the picture is that of a $(1+1)$-directed polymer among randomly located repulsive defect lines. Assuming that the increment (or gap) distribution of the renewal process has a power-law decay with tail exponent $\gga$ (see \eqref{eq:deftau}), we proved in a previous work~\cite[Theorem 2.2]{PoiSim19} that the logarithm of the quenched probability to survive up to time $n$, rescaled by $n^{\gga/(\gga+2)}$, converges in distribution to 
\beq
\label{eq:varform}
\inf_{(x,y)\in \Pi}\Big(\gl x + \frac{\pi^2}{2y^2} \Big),
\eeq
where $\Pi$ is a Poisson point process on $[0,+\8) \times (0,+\8)$ with explicit intensity measure and $\gl$ is a constant (the Lyapunov exponent) depending on the gap distribution and on the killing strength of the obstacles. As we pointed out in~\cite{PoiSim19}, this is a hint that the random walk conditioned to survive for a long time $n$ should eventually localize near an optimal gap whose position and length correspond (after suitable renormalization) to the minimizer of~\eqref{eq:varform}. In the present paper, we refine our analysis to prove such localization results. We actually establish that the random walk quickly reaches a unique optimal gap (in the sense that the hitting time is negligible in front of $n$) and roughly stays there for the time left. To be more precise, our study reveals a dichotomy:
\begin{itemize}
\item When $\gga>1$, the random walk is allowed to perform short excursions outside of the optimal gap, see Theorem~\ref{thm:ggasup1}. 
\item When $\gga<1$ the random walk reaches the optimal gap and may touch its boundary for a $O(1)$ amount of time after which it remains {\it inside} the optimal gap, see Theorem~\ref{thm:ggainf1}.
\end{itemize}
\par Since the obstacles are drawn according to a renewal process, our work provides an exemple of localization in a correlated disordered environment. Influence of spatial correlations has been investigated in other contexts such as localization of directed polymers with space-time noise~\cite{Lacoin11}, Brownian motion in correlated Poisson potential~\cite{Lacoin2012-I, Lacoin2012-II, Rang2020} as well as Anderson localization, both by mathematicians~\cite{MR3736656,MR3689969,Lyu2020} and physicists~\cite{AubryAndre1980, Croy2011, tessieri2015}. In our model however, the extreme value statistics of the gaps are more relevant than the correlation structure itself. Those are related to the heavy tails of the gap distribution. Influence of heavy tails in localization phenomena has been considered in the context of directed polymers~\cite{AufLou11, BergerTorri19,MR3572330,MR3491341,Viveros21} as well as in the parabolic Anderson model~\cite{BisKon01,MR3052403,MR3551159, MR3835474, MR2883854,MR2474543}. Heavy tails also play a decisive role in trap models, see e.g.~\cite{BAC07, CroyMuir16}.\\

\par Random motion among random obstacles has been intensively studied in the past decades. We refer to Sznitman's monograph~\cite{Sz98} for a thorough analysis of Brownian motion among Poissonian obstacles in dimension $d\geq 1$. More recently, Ding, Fukushima, Sun and Xu~\cite{ding2019distribution, DIngXu2019, DingXu2019LOC} gave detailed results in the case of Bernoulli hard obstacles in dimension $d\geq 2$. We will give a more complete description of the results for Bernoulli (soft or hard) obstacles, as well as a comparison with our own work and techniques, in Comment~\ref{it:ber} (Section~\ref{sec:comments}).\\

\par Note that our model can also be interpreted as a directed $(1+1)$ polymer interacting with multiple repulsive interfaces. In this context, the probability to survive turns out to be the partition function while the distribution of the random walk conditioned to survive coincides with the polymer measure (see \eqref{eq:funcpart}, \eqref{eq:polymeasure} and the paragraph below these two equations). This statistical mechanics point of view was adopted in our previous work \cite{PoiSim19} and previously by Caravenna and Pétrélis~\cite{CP09b, CP09} (see Comment~\ref{it:cp} in Section~\ref{sec:comments}).\\

%%%%%%%%%%%%%%%%%%%%%%%%%%%%%%%%%%%%%%%%%%%
%%%%%%%%%%%%%%%%%%%%%%%%%%%%%%%%%%%%%%%%%%%
\section{Model and results}

Throughout the paper, the sets of positive and non-negative integers are respectively denoted by $\bbN$ and $\bbN_0=\bbN \cup \{0\}$.

\subsection{Definition of the model}
{\it \noindent Simple random walk and hitting times.} Let $S=(S_n)_{n\ge 0}$ be a simple random walk on $\bbZ$ starting from the origin and defined on a probability space $(\gO,\cA,\bP)$. We use $\bP_x$ when the walk rather starts at $x\in \bbZ$. We denote by $H_x$ (resp. $H_A$) the first hitting time of $x$ (resp. the set $A$) starting from time $1$ while $\tilde H_x$ (resp. $\tilde H_A$) is the first hitting time of $x$ (resp. the set $A$) starting from time $0$. For all subsets $I$ of $\bbN_0$, we denote $S_I=\{S_n,\ n\in I\}$.\\

{\it \noindent Obstacle set and killed random walk.} 
An environment is a subset of $\bbZ$ whose elements are called {\it obstacles}. The intervals between consecutive obstacles shall be referred to as {\it gaps}. Let $\gb>0$, the value of which is fixed through the paper. Given an environment $\tau$, we define the {\it random walk killed by the obstacles} as follows. Each time the walk $S$ visits an obstacle, it has a probability $1-e^{-\gb}$ to be killed, independently of the previous visits to $\tau$. Moreover, the walk is killed with probability one when visiting $\bbZ^- := \bbZ\setminus \bbN$. This assumption is only for technical convenience and does not hide anything deep. We denote by $\bgs$ the time of death. In other terms, $\bgs = \bgs(S,\cN,\tau)=\gt_\cN \wedge H_{\bbZ^-}$, where $\gt = (\gt_n)_{n\ge 0}$ are the consecutive visits of the walk to $\tau$ ($\gt_0=0$ when there is an obstacle at the starting point) and $\cN$ is an \bbN-valued Geometric random variable with parameter $1-e^{-\gb}$ that is also defined on $(\gO,\cA,\bP)$ and independent of the random walk $(S_n)_{n\ge 0}$. We will also sometimes consider $\gs := \gt_\cN$ instead of $\bgs$.
\\

{\it \noindent Assumption on the obstacles.} Let $\gga>0$. In this article we consider a random environment $\tau=(\tau_n)_{n\ge 0}$ that is a discrete $\bbN_0$-valued renewal process defined on a probability space $(\gO_{\obs}, \cA_{\obs}, \bbP)$, which starts from $\tau_0=0$ and such that the increments $(\tau_i - \tau_{i-1})_{i\geq 1}$ are i.i.d., $\bbN$- valued, and satisfy
\beq
\label{eq:deftau}
\bbP(\tau_i - \tau_{i-1} = n) \stackrel{n\to\infty}{\sim} c_\tau  n^{-(1+\gga)}, \qquad i \in \bbN,
\eeq
for some positive constant $c_\tau$. The point $\tau_i$ indicates the position of the $i$-th obstacle along the integer line. The length of the gap $(\tau_{i-1}, \tau_{i})$ is denoted by $T_i:= \tau_{i} - \tau_{i-1}$ for $i\ge 1$. As $\bgs=\bgs(S,\cN,\tau)$ it may be viewed as a random variable defined on the probability space  $(\gO \times \gO_{\obs},\cA\otimes \cA_{\obs},\bP\otimes \bbP)$ or, when $\tau$ is fixed, as a random variable defined on $(\gO,\cA,\bP)$. As we study \textit{quenched} properties of the walker, we will in the following mainly consider this last point of view.\\

{\it \noindent Notation.}
Let us collect here some notation that we use throughout the paper and which we keep as close as possible to that of \cite{PoiSim19}.
\begin{itemize}
\item {\bf (Dependence on parameters)} We may add a superscript to the symbol $\bP$ when we work with a different value of~$\gb$. For convenience, the dependence of $\gs$ and $\bgs$ on the environment will be written explicitly only when it is not the one defined in~\eqref{eq:deftau}.
\item {\bf (Expectation)} For all events $A$ (measurable with respect to the random walk) we sometimes use the standard notation
$\bE(X, A) = \bE(X \ind_A)$.
\item {\bf (Moment Generating Functions)} If $\{u(n)\}_{n\ge 1}$ is a sequence of non-negative real numbers then we denote by
\beq
\label{eq:mgf}
\hat u(f) = \sum_{n\ge 1} \exp(fn) u(n)
\eeq
 its $[0,\infty]$-valued moment generating function ($f\ge 0$).
\item {\bf (Records)} The sequence of {\it gap records}, denoted by $(T^*_k)_{k\ge 1}$ and defined by
\beq
\label{eq:records}
T^*_k = T_{i(k)},
\eeq
where 
\beq
\label{eq:indexes}
i(1)=1 \quad \text{ and } \quad  i(k+1) = \inf\{i > i(k) \colon T_i > T_{i(k)}\},
\eeq
plays an important role in our analysis. We also define for $k\in \bbN$,
\beq
\tau^*_k := \tau_{i(k)-1},
\eeq
that is the obstacle at the bottom of the gap $T^*_k$ and 
\beq
\label{eq:defhstar}
H^*_k :=H_{\tau^*_k}
\eeq
is the hitting time of $\tau^*_k$.

\item {\bf (Statistical mechanics terminology)} We shall write for all events $A$ (measurable with respect to the random walk) and $n\in\bbN$,
\beq
\label{eq:funcpart}
Z_n(A) := \bP(A, \bgs >n),
\qquad Z_n := Z_n(\gO),
\eeq
so that 
\beq
\label{eq:polymeasure}
\bP(\, A\, |\bgs >n) = Z_n(A)/Z_n.
\eeq
The quantity $Z_n$ is referred to as the partition function and the measure $Z_n(\cdot)/Z_n$ as the polymer measure. This terminology comes from statistical mechanics: one can indeed consider the path $(k,S_k)_{0\le k\le n}$ as a $(1+1)$ directed polymer among horizontal interfaces located at heights $(\tau_i)_{i\geq 0}$. Each time this polymer touches one of these interfaces it increases its energy by one, so that the Hamiltonian associated to this model writes $\cH_n(S)=\sum_{k=1}^n1_{\{S_k\in \tau\}}$ and the polymer (probability) measure has a density with respect to the simple random walk law that is proportional to $\exp\{-\beta \cH_n(S)\}$ where $\beta$ plays the role of the inverse temperature. One can easily check that the partition function of this model, that is the renormalisation constant, coincides with the probability $\bP(\bgs>n)$ to survive up to time $n$, while the polymer measure writes $Z_n(\cdot)/Z_n$. In the following we adopt the most convenient terminology, according to context.
\end{itemize}
%%%%%
\subsection{Results}  In order to formulate our results we need to introduce additional notation.\\
{\it \noindent Confinement estimates.} For all $t\in\bbN$, $g(t)$ denotes the asymptotic rate of decay of the probability that the walk stays confined in a slab of width $t$. More precisely, we define for $n\in \bbN$,
\beq
\label{eq:defRuinproba}
q_t(n) = \bP_0(\inf\{k>0\colon S_k \in \{-t,0,t\}\}= n),
\eeq
and 
\beq
\label{eq:def:gt}
g(t) := - \log \cos\Big(\frac{\pi}{t}\Big) =\frac{\pi^2}{2t^2}\Big[1 + \frac{\pi^2}{6t^2} + o\Big(\frac 1 {t^2}\Big) \Big],
\eeq
where the $o$ holds as $t\to\infty$.
The following result will be used many times throughout the paper:
\begin{lemma}
\label{lem:estim-qtn}
There exist $\cT_0>0$ and $c_1,c_2,c_3,c_4,c_5>0$ such that for all $t>\cT_0$,
\beq
\label{eq:estim-qtn1}
\frac{c_3}{t^3 \wedge n^{3/2}} e^{-g(t)n} \le q_t(n) \le \frac{c_4}{t^3 \wedge n^{3/2}} e^{-g(t)n},
\eeq
with the upper bound valid for all $n\in\bbN$ and the lower bound valid (i) for all $n\in 2\bbN$ and (ii) for all $n\in2\bbN-1$ with the extra conditions that $n\ge c_5t^2$ and $t\in 2\bbN-1$. Moreover, for all $n\in \bbN$ and $t>\cT_0$,
\beq
\label{eq:estim-qtn2}
\frac{c_1}{t \wedge n^{1/2}} e^{-g(t)n} \le \sum_{i>n} q_t(i) \le \frac{c_2}{t \wedge n^{1/2}} e^{-g(t)n}.
\eeq
\end{lemma}
A crucial point here is the uniformity of the constants. The proof of Lemma~\ref{lem:estim-qtn} and other useful results concerning the generating functions of ruin probabilities are collected in Appendix~\ref{app:ruin}.\\

{\it \noindent Localization interval.} 
We recall from~\cite{PoiSim19} the notation $ N = n^{\frac{\gga}{\gga+2}}$ that appears to be the relevant scale for studying localization, and also the function
\beq
\label{def:Gell}
G_n^\gb(\ell) =  \frac{\gl(\gb,\ell-1)(\ell-1)}{N} + g(T_{\ell}) \frac{n}{N},  \qquad \ell\in \bbN,
\eeq
where
\beq
\label{def:gl}
\gl(\gb,\ell) := -\frac{1}{\ell}\log \bP_0(H_{\tau_\ell} < \bgs), \qquad \ell \in \bbN,
\eeq
with the convention $\gl(\gb,0)=0$.
We also remind the reader that from \cite[Proposition 2.1]{PoiSim19}, the function $\gl$ is defined for all $\gb> 0$ by
\beq
\label{eq:defgl}
\lambda(\beta)=\lim_{\ell \to +\8}\gl(\gb,\ell),
\eeq
where the convergence holds $\bbP-a.s.$ and in $L_1(\bbP)$.
Let 
\beq
\label{eq:defell0}
\lopt = \lopt(n,\tau,\gb) = \argmin_{\ell \ge 1} G_n^\gb(\ell),
\eeq
which is uniquely defined when $n$ is large, on an event of large $\bbP$-probability (see~\eqref{eq:newcoinditions} and Proposition~\ref{pr:goodenv} below). On the asymptotically unlikely event of more than one minimizer we define $\lopt$ as the smallest one. Note that $\lopt$ necessarily corresponds to one of the gap records, i.e. there exists $\kopt = \kopt(n,\tau,\gb)$ such that $T_{\lopt} = T^*_{\kopt}$. The following interval (which depends on $n$, $\gb$ and the obstacles):
\beq
I_\loc^{(n)} = (\tau_{\kopt}^*, \tau_{\kopt}^*+ T^*_{\kopt}) = (\tau_{\lopt-1}, \tau_{\lopt}).
\eeq
will be referred to as {\it optimal gap} or {\it localization interval}. We shall often omit the dependence on $n$ in order to lighten notation.

%%%%%%%%%%%%%%%%%%%%%%%%%%%%%%%%%%%%%%%%%%%

\par Let us now introduce some random variables that will help us describe the trajectory of the random walk once it has reached the optimal gap.
We first define
\beq
\label{eq:defthetabar}
\bar \gt = \{\bar\gt_i\}_{i\ge 0} = \{k \ge 0 \colon S_k \in \{\tau_{\lopt-1}, \tau_{\lopt}\}\},
\eeq
that is the increasing sequence of times at which the walk visits the boundary of the optimal gap. Note that $\bar \gt$ is both a (delayed) renewal process and a subset of $\gt$. Note also that $\bar \gt_0$ is the hitting time of the optimal gap and thus coincides with $H_{\tau_{\lopt-1}}$ under $\bP$.
Let
\beq
\label{eq:defcN}
\cN(n) = \sup\{i\ge 0 \colon \bar \gt_i \le n\} \qquad (\sup \emptyset = 0),
\eeq
be the number of excursions ``in'' or ``out'' of the localization interval and, for all $a,b\in\{0,1\}$,
\beq
\cN_\ins^{a,b}(n) = \card\{0\le k \le \cN(n) \colon S_{\bar \gt_k} = \tau_{\lopt-1+a},  S_{\bar \gt_{k+1}} = \tau_{\lopt-1+b}, S_{(\bar \gt_k, \bar \gt_{k+1})} \subseteq I_\loc\}
\eeq
be the number of excursions inside the localization interval (from one side of the interval to the same one when $a=b$ and crossing the interval when $a\neq b$), and
\beq
\cN_\out^{a}(n) = \card\{0\le k \le \cN(n) \colon S_{\bar \gt_k} = S_{\bar \gt_{k+1}} = \tau_{\lopt-1+a}, S_{(\bar \gt_k, \bar \gt_{k+1})} \cap I_\loc = \emptyset\}
\eeq
be the number of excursions outside of the localization interval. Finally, we denote by
\beq
\cT_\out(n) = \card\{\bar \gt_0 \le k \le n \colon S_k \notin I_\loc\}
\eeq
the time spent by the random walk outside of the localization interval between the first contact with $I_\loc$ and the final time $n$.\\

\par We now have all the notation in hands to state our results, which we split in two theorems respectively dealing with the case of integrable gaps ($\gamma >1$) and non-integrable ones ($\gamma <1$). This is justified by the fact that the localization of the random walk in $I_\loc^{(n)}$ is stronger when $\gamma <1$. 
\begin{theorem}[Case $\gga >1$] 
\label{thm:ggasup1}
For all $\kappa,\petit\in (0,1)$ there exists $C>0$ such that, for $n$ large enough and for all $a,b\in\{0,1\}$
\beq
\label{eq:ggasup1}
\bP\Big(H_{\tau_{\lopt-1}} \le \kappa n, \frac{n}{CT_{\lopt}^3} \le \cN^{a,b}_{\ins}(n), \cN^a_{\out}(n),\cT_{\out}(n)  \le C \frac{n}{T_{\lopt}^3}\Big| \bgs > n\Big) \ge 1-\petit,
\eeq
with $\bbP$-probability larger than $1-\petit$.
\end{theorem}
\begin{theorem}[Case $\gga <1$]
\label{thm:ggainf1}
For all $\kappa,\petit\in(0,1)$, there exists $C>0$ such that, for $n$ large enough,
\beq
\label{eq:ggainf1}
\bP\Big(H_{\tau_{\lopt-1}} \le \kappa n,
\forall k \in [H_{\tau_{\lopt-1}}+C,n], S_k \in I_{\loc}^{(n)} \Big| \bgs > n
\Big) \ge 1-\petit,
\eeq
with $\bbP$-probability larger than $1-\petit$.
\end{theorem}

\subsection{Comments}
\label{sec:comments}
Let us give a few comments about our results.
\begin{enumerate}
\item {\bf (Lyapunov exponent)} Our definition of the Lyapunov exponent in~\eqref{def:gl} and~\eqref{eq:defgl} does not coincide with~\cite[Section 5.2]{Sz98} where the index of the obstacle is replaced by the distance to the origin. This is actually irrelevant when $\gga>1$, by the Law of Large Numbers, but not when $\gga<1$.
\item {\bf (Alternative definition of the optimal gap)} In our results the localization interval $I_\loc^{(n)}$ is defined via the minimizer $\lopt$ of the function $G$. Actually it could also be defined using the simpler function $\tilde G$ and its minimizer $\tilde \lopt$ defined only later in~\eqref{eq:defelltilde0} and~\eqref{eq:tildeG}. These two functions are indeed close to each other for the value of $\ell$ of interest, so that their respective minimizers coincide (see the definition of $\kBone_n$ in \eqref{eq:newcoinditions} and the proof of Lemma \ref{lem:goodB0} that states that the latter event occurs with large probability when $n$ goes to infinity).
\item \label{it:cp} {\bf (Heuristics for the dichotomy)} Caravenna and Pétrélis~\cite{CP09b} considered the case of deterministic gaps scaling like $t_n \approx n^{\ga}$. They proved that the value $\ga = 1/3$ is the frontier between two different regimes. If $1/3 < \ga < 1/2$ then the random walk conditioned to survive spends all of its time inside of a single gap, up to finitely many visits to the obstacle sitting at the starting point. If $0<\ga<1/3$ then excursions in-between consecutive visits to the obstacle set have length $t_n^3 \approx n^{3\ga}$ and the total number of excursions is $n^{1-3\ga}$, in order of magnitude. This gives a heurisitic explanation of the critical value $\gga=1$, since the relevant gaps are of order $n^{1/(\gga+2)}$ in our model (see also Remark~\ref{rk:comparegapn} below). In order not to lengthen too much the paper, we restrain from treating the case $\gga=1$, for which one should probably replace the constant prefactor in~\eqref{eq:deftau} by a slowly varying function.
\item {\bf (Hitting time of the optimal gap)} The fact that the random walk reaches the optimal gap in a time $o(n)$ was already observed in the case of Bernoulli obstacles, see~\cite[Corollary 5.7]{Sz98}. Our theorems do not provide a lower bound for the hitting time of the optimal gap and we actually do not believe that our upper bound is sharp. This is suggested by the fact that, for $x\in \N$, the random walk under $\bP(\cdot|\bgs>H_x)$ is a Markov chain with explicit transitions (see \cite[Proposition~3.6]{PoiSim19} for a precise formulation). Using standard tools from the theory of random walks in random environment, one can thus establish that as $x\to\infty$, 
\beq
\bE(H_x | \bgs  > H_x) \sim \text{(cst)}
\left\{
\begin{array}{ll}
x^2& (\gga<1)\\
x^{{2}/{\gga}}& (1<\gga<2)\\
x&(2< \gga).
\end{array}
\right.
\eeq
Letting $x=\tau_{\lopt-1}$, this leads to an estimate for the expected hitting time of the optimal gap that is much smaller than $n$.
We postpone to future work a more precise study of $H_{\tau_{\lopt-1}}$ under $\bP(\cdot| \bgs >n)$.
\item \label{it:ber} {\bf (Bernoulli obstacles)} The case of Bernoulli obstacles has been extensively studied in the literature. In dimension $d=1$ this corresponds to obstacles being distributed in an i.i.d. fashion, or equivalently, to the gaps being independent and geometrically distributed,
\beq
\label{eq:bernoulli}
\bbP(T_1 = n) = (1-p)^{n-1}p, \qquad n\in\bbN, \qquad p\in(0,1). 
\eeq
In all dimensions, there exists a constant $c>0$ such that the quenched survival probability satisfies the following large-time asymptotics:
\beq
\bP(\gs >n) = \exp\Big(-[c+o(1)]\frac{n}{(\log n)^{2/d}}\Big),
\eeq
where the convergence holds $\bbP$-almost-surely, see~\cite[Theorem 5.1]{Sz98} and~\cite{Antal95, Fuk2009,Sznitman93-ptrf}. The fact that the mode of convergence differs from~\cite[Theorem 2.2]{PoiSim19} comes from the difference in extreme value asymptotics. To the best of our knowledge, the strongest result regarding path localization in dimension one is \cite[Corollary~5.7]{Sz98}, which states (in the context of Brownian motion in Poisson obstacles) that the walk conditioned to survive hits one of several possible pockets of localization in a time $o(n)$ and remains there. These pockets are at distance of order $n(\log n)^{-3}$ from the origin and have a diameter of order $(\log n)^{2+o(1)}$, as $n\to\infty$. There remains the question of whether the walk targets a unique pocket of localization (referred to as one city or one island property in the parabolic Anderson literature). In a recent series of papers, Ding, Fukushima, Sun and Xu~\cite{ding2019distribution,DIngXu2019,DingXu2019LOC} have answered this question in the positive (among other finer results) in the case of hard obstacles ($\gb = +\infty$) and dimension $d\ge 2$. Interestingly, the same authors expect the walk to make excursions up to length (cst)$\log n$ outside of the localization ball, see \cite[Remark~1.4]{ding2019distribution}. This echoes the existence of short excursions outside of the localization interval that we establish in our model, when the gap exponent is larger than one.\\

The extreme value statistics in the Bernoulli case differs from ours. Indeed, under assumption~\eqref{eq:bernoulli}, one may check that the point measure associated to the collection $(i/n, |\log (1-p)| T_i - \log n)$ for $i\in \bbN$ converges in distribution to a Poisson point process on $(0,\infty)\times \bbR$ with intensity measure $\dd x \otimes e^{-y}\dd y$. Therefore, the difference between two competing gaps is of a much smaller order than in our case. This explains why it is difficult to isolate one among all the many gaps which contribute to the first order asymptotic of the survival probability. See the discussion in~\cite[Section 6.1]{Sz98}, in particular (1.15). In contrast, the existence of a unique localizing gap in our model may be already conjectured from the first-order asymptotics of the survival probability, which is expressed as the solution of a variational problem. Also, the decomposition technique (according to the furthest visited record gap) that we use to derive upper bounds and Proposition~\ref{pr:weakloc} (a preliminary weak form of localization) seems too weak to handle Bernoulli obstacles, since the ratio between a given record gap and all the (many) competing gaps will be dangerously close to one in that case. The enlargement of obstacles technique (see \cite{Sz98,Sznit01} and references therein) was designed to tackle such situations by providing an elaborate coarse-graining of the obstacle environment, which is then split into safe and dangerous zones (the so-called {\it clearings} and {\it forests}). Our method is more elementary but sufficient to handle polynomial gap distribution. Let us stress however that we do need more refined renewal techniques to go from Proposition~\ref{pr:weakloc} to the more detailed localization results of Theorems~\ref{thm:ggasup1} and~\ref{thm:ggainf1}.
\item {\bf (Parabolic Anderson model)} The random walk in random obstacles belongs to the more general framework of the Parabolic Anderson Model (PAM), that is a random walk sampled according to a probability measure proportional to
\beq
\exp\Big(\sum_{1\le i\le n} \go(S_i)\Big),
\eeq
where $\{\omega(x)\}_{x\in \bbZ^d}$ is a field of random variables called potential, see König~\cite{Ko16} for an account on this topic and the appendix therein for a discussion of open problems. In the language of the PAM, Theorems~\ref{thm:ggasup1} and ~\ref{thm:ggainf1} correspond to a one city (or one island) theorem. 
Such results have been obtained for the PAM with heavy-tailed i.i.d. potentials~\cite{KLMS09}. In our case, the potential $\omega(x) = -\gb \ind_{\{x\in \tau\}}$ is bounded and correlated. In contrast with the i.i.d. setup, the width of the heavy-tailed gaps then plays the decisive role, rather than the height of the potential. For other results about localisation in a similar framework, see also~\cite{BKdS18, konig2023weakly}.
\end{enumerate}

\subsection{Organisation of the paper} The proof of Theorems~\ref{thm:ggasup1} and~\ref{thm:ggainf1} proceeds in two steps: we show that \eqref{eq:ggasup1} and \eqref{eq:ggainf1} hold on a certain subset $\good_n\subseteq \gO_\obs$ (defined in Section~\ref{sec:goodenv}) and then prove that this subset has a large $\bbP$-probability (Appendix~\ref{sec:proof-good-env}). The proof of the first part is itself decomposed in two parts:
\begin{enumerate}
\item Control of the hitting time of the optimal gap (Section~\ref{sec:hitting});
\item Description of the trajectory of the random walk started at the lower boundary of the optimal gap (Sections~\ref{sec:locinf1} and~\ref{sec:locsup1}). In both sections the main tool is the comparison between the random walk conditioned to survive with a certain Markov renewal process that we define and study in Section~\ref{sec:rmrp}.
\end{enumerate}
We finally assemble these different parts in Section~\ref{sec:proofmain}.

\subsection{Reminder on the case of equally spaced obstacles}
Caravenna and Pétrélis~\cite{CP09b} treated the case of equally spaced obstacles, which we refer to as the {\it homogeneous} case, in the sense that increments of $\tau$ are all equal. We summarize their results here as we use them repeatedly in the following (recall the difference between $\gs$ and $\bgs$).

\begin{proposition}[Homogeneous case, see Eq. (2.1)-(2.3) in \cite{CP09b}]
\label{pr:homo}
Assume $\tau = t\bbZ$, where $t\in \bbN$. There exists a positive constant $\phi(\gb,t)$ such that 
\beq 
\label{eq:315}
\phi(\gb,t) = - \lim_{n\to\infty} \frac{1}{n} \log \bP(\gs > n),
\eeq
with 
\beq
\label{eq:phi}
\phi(\gb,t) = \frac{\pi^2}{2t^2} \Big[1 - \frac{4}{e^{\gb} - 1}\frac{1}{t} + o\Big(\frac{1}{t} \Big) \Big].
\eeq
Moreover, it is the only solution of the equation:
\beq
 \bE[\exp(\phi \inf\{n\in \bbN \colon S_n \in \tau\})] = \hat q_t(\phi) =\exp(\gb),\qquad \gb\ge 0.
\eeq
\end{proposition}
Note that the first order term in the expansion of $\phi(\gb,\cdot)$ does not depend on $\gb$ and coincides with that of~\eqref{eq:def:gt}. We may drop the subscript $\gb$ when there is no risk of confusion. The function $\phi(\gb,\cdot)$ provides an upper bound on the survival probabilities even when the obstacles are not equally spaced:
\begin{proposition}\label{prop:roughUB} (See Proposition 3.4 in~\cite{PoiSim19})
Let $(\tau_i)_{i\in \bbN_0}$ be any increasing sequence of integers with $\tau_0=0$ and let $t_i = \tau_i - \tau_{i-1}$ for all $i\in\bbN$.
There exists a constant $C>0$ (that does not depend on $\tau$) such that for all $0\le k < r <  \ell$ and $n\ge 1$, one has
\beq
\bP_{\tau_r}(\gs \wedge H_{\tau_k} \wedge H_{\tau_\ell} > n) \le Cn^2(\ell-k) \exp(-\phi(\gb,\max\{t_i \colon k<i\le \ell\})n).
\eeq
\end{proposition}
This statement may look stronger than in~\cite[Proposition 3.4]{PoiSim19} as the constant $C$ does not depend on $\tau$ anymore. An inspection of the proof of \cite[Proposition 3.4]{PoiSim19} reveals that it is actually the case.
We refer to~\cite[Section 3.2]{PoiSim19} for more details on these confinement estimates.

%%%%%%%%%%%%%%%%%%%%%%%%%%%%%%%%%%%%%%%%%%%
%%%%%%%%%%%%%%%%%%%%%%%%%%%%%%%%%%%%%%%%%%%
\section{Good obstacles environment}
\label{sec:goodenv}
In this section we define the event $\good_n$ of \textit{good environments} (in the sense that on this event \eqref{eq:ggasup1} and \eqref{eq:ggainf1} are satisfied) and state that this event happens with large probability. The event $\good_n$ is defined as the intersection of the event $\Omega_n$ introduced in \cite{PoiSim19} (recalled right below) and additional events specific to this paper. 
%%%%%%%
\par \textit{Reminder from previous work.}
We first recall the definition of $\Omega_n$ in \cite{PoiSim19} that depends on parameters $\delta,\gep_0,\gep,\eta>0$. In order to be exhaustive we list the various notation and objects involved in the definition of $\Omega_n$ and point out where they have been introduced in~\cite{PoiSim19}:
\begin{enumerate}
\item The constant $C_1$ is a positive constant so that
\beq
\label{eq:encadr_gphi}
1/(C_1 t^2) \leq g(t), \phi(\gb,t) \le C_1 / t^2, \qquad t\in \bbN. 
\eeq
Its existence is guaranteed by \eqref{eq:def:gt} and \eqref{eq:phi}.
\item 
The exponent $\kappa$ is defined by
\beq
\label{def:kappa}
\gk =
\begin{cases}
\frac{\gamma}{4} & \text{if } \gamma \le 1\\
\frac{1}{2\gamma} - \frac14 & \text{if } 1 <\gamma < 2\\
\frac{1}{2\gamma} & \text{if } \gamma \ge 2.
\end{cases}
\eeq
\item The constant $\cT_0$ comes from Lemma~\ref{lem:estim-qtn}. 
\item Recall~\eqref{eq:records} and \eqref{eq:indexes}. The sets of records are defined by
\beq
\label{eq:defrecordset}
\ba
R(a,b) &= \{k \in \bbN \colon a\le i(k) \le b\},\quad \cR(a,b) = i(R(a,b)),\quad a,b\in\bbN,\ a<b,\\
R_\gep(n) &= R(\gep N, \gep^{-1}N),\qquad \cR_\gep(n) = \cR(\gep N, \gep^{-1}N), \qquad n\in \bbN, \qquad \gep>0.
\ea
\eeq

\item \label{it:f} The functions $\alpha$, $\cT_0$ and $f: z\in (0,\pi)\mapsto z/\sin(z)$ are defined in \cite[Lemma 3.7]{PoiSim19} while $h\colon (A,L,\ga)\in(0,\infty)^3 \mapsto A^{2L}/[2\ga e^\gb(e^\gb-1)]$ comes from \cite[Lemma 3.8]{PoiSim19}.
\item We define for $k\geq 1$
\beq
\label{eq:def_fk}
f_{k}:=2 f\left(\pi \frac{T^*_{k-1}}{T^*_k}\Big[1+\frac{C}{(T^*_k)^2}\Big]\right),
\eeq
where the constant $C$ is the same as the one in \cite[Lemma 3.7]{PoiSim19}.
\item Given $\alpha>0$ and $k\in\bbN$, we define the set of {\it bad} edges as
\beq
\label{eq:defB}
\cB_{k,\alpha} = \{j < i(k),\; T_{j} > \alpha T^*_{k}\},
\eeq
and its cardinal
\beq
L_{k,\alpha} = |\cB_{k,\alpha}|.
\eeq
\item The functions $\gl(\cdot)$ and $\gl(\cdot,\cdot)$ are defined in \eqref{def:gl} and \eqref{eq:defgl}.

\item For $n\geq 1$ we define the random point measure
\beq
\label{eq:defpin}
\Pi_n =\sum_{i=1}^{\infty}\delta_{\left(\frac{i-1}{n}, \frac{T_i}{n^{1/\gamma}}\right)}.
\eeq

\end{enumerate}

\par We are now ready to recall the definition of $\Omega_n$:
\beq
\Omega_n(\gd,\gep_0,\gep,\eta)=\bigcap_{i=1}^{11} A_n^{(i)}(\gd,\gep_0,\gep,\eta),
\eeq
with
\beq
\label{eq:ge}
\ba
A^{(1)}_n&=
\begin{cases}
{\{\tau_{N^{1+\gk}}^2 < n^{1-\frac{\gamma\wedge (2-\gamma)}{4(\gamma+2)}}\}} & \text{if } \gga < 2\\
{\{\tau_{N^{1+\gk}}^2 < n^{1+\frac{2\gga - 1}{2(\gamma+2)}}\}} & \text{if } \gga \ge 2,\\
 \end{cases}\\
A_n^{(2)}(\gep_0) &:= {\{ T_k \le \gep_0^{\frac{1}{2\gamma}} N^{\frac{1}{\gamma}},\quad \forall k\le \gep_0N\}}\\
A_n^{(3)}(\gep_0) &:= \{\tau_{N/\gep_0} < n\}\\ 
A_n^{(4)}(\gd) &:=\{\exists \ell \in \{N,\ldots, 2N\}\colon T_\ell \ge \cT_0 \vee \gd N^{\frac{1}{\gamma}}\}\\
A_n^{(5)}(\gep_0,\gep) &:=\{\forall k\in R_{\gep_0}(n),\ T^*_k > \cT_0(\gep) \vee \gep_0^{\frac{3}{2\gamma}}N^{\frac{1}{\gamma}}\}\\
A_n^{(6)}(\gep_0,\gep) &:=\{\forall k\in R_{\gep_0}(n),\ f_k^{L_{k,\alpha(\gep)}} \le \exp(n^{\frac{\gamma}{2(\gamma+2)}})\}\\
A_n^{(7)}(\gep_0,\gep) &:=\{\forall k\in R_{\gep_0}(n),\ T^*_k > h(f_k, L_{k,\ga(\gep)}, \ga(\gep))\}\\
A_n^{(8)}(\gep_0) &:= \{|R(1,N/\gep_0)| \le [\log(N/\gep_0)]^2\}\\
A_n^{(9)}(\gd) &:=\{ |\gl(2N,\gb) - \gl(\gb)| \le \tfrac{C_1}{2\gd^{2}}\}\\
A_n^{(10)}(\gep_0,\gep,\eta) &:= \{ |\gl(\ell-1,b) - \gl(b)| \le \tfrac{\gep_0\eta}{2},\ \forall \ell\ge \gep_0N,\ b\in \{\beta, \beta-\gep\}\}\\
A^{(11)}_n(\gep_0)&:=\{ \Pi_N( [0,\gep_0^{-\gga/2}]\times [\gep_0^{-1},+\8[) = 0 \}.
\ea
\eeq

{\begin{remark}
We point out a mistake that is easy to correct in the definition of $A^{(11)}$ in \cite{PoiSim19}. With the new definition above it is easy to check, adapting \cite[(6.41)]{PoiSim19}, that 
\beq
\lim_{\gep_0\to 0} \liminf_{n \to \8} \bbP(A_n^{(11)}(\gep_0))=1.
\eeq
Furthermore, the event $A^{(11)}_n(\gep_0)$ is used in the proof of \cite[Theorem 2.2]{PoiSim19} only, more precisely in \cite[(6.58)]{PoiSim19}. With the correct definition of $A^{(11)}_n(\gep_0)$ above, this part of the proof is now valid assuming that $\gep_0$ has been chosen small enough so that $\gep_0^{-\gga/2}\geq \frac{u+2\eta}{\lambda(\gb-\gep)}$.
\end{remark}
}
%%%%%%%%%%%%%%%%%%%%%%%%%%
\par \textit{New conditions.} We turn to the five new conditions that we add to the previous ones to define a \textit{good environment}.
Recall \eqref{eq:defell0} and define analogously
\beq
\label{eq:defelltilde0}
\tilde\ell_0 = \tilde\ell_0(n,\tau,\gb) = \argmin_{\ell \ge 1} \tilde G_n^\gb(\ell)
\eeq
where 
\beq
\label{eq:tildeG}
\tilde G_n^\gb(\ell) = \gl(\gb) \frac{\ell-1}{N} + \frac{\pi^2}{2 T_{\ell}^2} \frac{n}{N}, \qquad N = n^{\frac{\gga}{\gga+2}}.
\eeq
We define for $\gep_0\in(0,1)$, $\gep\in(0,\gb/2)$, $\eta>0$, $\rho\in(0,\frac12)$, $J \in \bbN$ and $\mathsf C \in \bbN$ the events
\beq
\label{eq:newcoinditions}
\ba
\kBone_n(\gep_0)&=\left\{
\begin{array}{c}
\lopt(\gb) \textrm{ and } \tilde{\lopt}(\gb) \textrm{ are uniquely defined,} \\
\gep_0 N \le \lopt(\gb) = \tilde{\lopt}(\gb) \le \gep_0^{-1} N,\\
\gep_0 N^{\frac 1\gga}\le T_{\tilde{\lopt}(\gb)} \le \gep_0^{-1} N^{\frac 1\gga}  \\
\end{array}
\right\}\\
\kBtwo_n(\eta)&=\Big\{\min_{\ell\neq \lopt}  G_n^{\beta}(\ell)- \min_{\ell}  G_n^{\beta}(\ell)\geq 2\eta\Big\}\\
\kBthree_n(\rho) &= \{\rho T_{\lopt} < \max_{i\neq \lopt, i< i(\kopt+1)} T_i < (1-\rho) T_{\lopt}\},\\
\kBfour_n(J) &= 
\Big\{\forall \ell \ge J,\ \max_{1\le i\le \ell} T_{\lopt + i} \le \ell^{\frac{4+\gamma}{4\gamma}}\Big\}
\cap
\Big\{\forall J\le \ell \le \lopt,\ \max_{1\le i\le \ell} T_{\lopt - i} \le \ell^{\frac{4+\gamma}{4\gamma}}\Big\}\\
\kBfive_n(\gep_0, \gep, \mathsf{C})&=\Big\{
\card\{k\in [\gep_0 N, \gep_0^{-1}N] \colon T_k \ge \frac{\alpha(\gep)\gep_0}{4} N^{\frac 1 \gga}\} \le \mathsf{C}
\Big\}.
\ea
\eeq
The first event assures that the optimal gap is well defined and gives a first control on its position and its width. 
\begin{remark}
\label{rk:comparegapn}
Note also that the last inequality in $\kBone_n$ implies that $n\ge T_{\lopt}^3$, provided $\gga >1$ (resp. $n\le T_{\lopt}^3$, provided $\gga <1$) if $n$ is large enough. This dichotomy will become important in the following.
\end{remark}
The second event guarantees that the strategy of localization in the optimal gap is the most favorable one. 
The third event allows us to control the ratio between the optimal gap and the second largest gap in the interval $[0,\tau^*_{\kopt+1}]$, that is the interval between the origin and the $(\kopt+1)$-th record gap.
%%%%%%%%%%%%
 The fourth event provides a control on the lengths of the gaps below and above the optimal gap, relatively to their distance to the latter.
 Note that the exponent $\frac{4+\gamma}{4\gamma}$ in the definition of $\kBfour_n$ is somewhat arbitrary and that it could be safely replaced by $c/\gga$ for any $1<c<1+\gga/2$. We choose $c = 1 + \gga/4$ in order to avoid a new parameter. Finally, the fifth event allows us to control the number of gaps competing with the optimal one in the proof of Proposition \ref{pr:weakloc}.\\
%%%%%%%%
\par \textit{Definition of a good environment.}
We may now define

\beq
\good_n(\gd, \gep_0, \gep, \eta, \rho, J, \mathsf{C}) := \gO_n(\gd, \gep_0, \gep, \eta)\cap \left(\bigcap_{1\le i\le 5} B^{(i)}_n(\gep_0, \gep, \eta, \rho, J, \mathsf{C})\right).
\eeq

In the following we will prove that \eqref{eq:ggasup1} and \eqref{eq:ggainf1} hold on $\good_n$ so that the proofs of Theorems~\ref{thm:ggasup1} and~\ref{thm:ggainf1} will be complete once we prove
\begin{proposition}
\label{pr:goodenv}
\par For all $\petit\in(0,1)$, for all $\gd, \gep_0, \eta, \rho \in (0,\frac12)$ small enough and $J \geq 1$ large enough, for all $\gep\in(0,\gb/2)$, there exists $\mathsf{C}=\mathsf{C}(\gep_0,\gep)$ such that
\beq
\liminf_{n\to\infty}  \bbP\left(\good_n(\gd, \gep_0, \gep, \eta, \rho, J, \mathsf{C})\right)\ge 1-\petit.
\eeq
\end{proposition}

We defer the proof of this proposition to Appendix~\ref{sec:proof-good-env}.
%%%%%%%%%%%%%%%%%%%%%%%%%%%%%%%%%%%%%%%%%%%
%%%%%%%%%%%%%%%%%%%%%%%%%%%%%%%%%%%%%%%%%%%
\section{Hitting time of the optimal gap}
\label{sec:hitting}
In this section we only focus on the first event on the left-hand sides of \eqref{eq:ggasup1} and \eqref{eq:ggainf1}, that is the one dealing with the hitting time of the localization interval. What happens after this hitting time will be studied in the next sections. We first provide in Proposition \ref{pr:weakloc} a rough control on the behaviour of the random walk conditioned to survive, namely that with large probability, it reaches the gap with label $i(k_0)$ (corresponding to the minimizer of $G_n$) but it does not reach the gap $i(k_0+1)$ (corresponding to the next record gap). We then refine this result in Proposition~\ref{pr:hitting} to get the desired upper bound on the hitting time of the optimal gap. Recall the notation from \eqref{eq:defhstar}.
\begin{proposition}
\label{pr:weakloc}
Let $\petit \in(0,1),\delta>0$. There exists $\gep_0>0$ so that for all $\eta\in(0,1)$ and for all $\gep$ small enough (depending on $\gep_0$ and $\eta$), if $n$ is large enough and $\tau\in \good_n(\gd,\gep_0,\gep,\eta)$,
\beq
\bP(H^*_{\kopt}\leq n <H^*_{\kopt+1} |\bgs > n)\ge1-\petit.
\eeq
\end{proposition}
\begin{proof}[Proof of Proposition~\ref{pr:weakloc}]
Recall the definitions of $C_1$ in \eqref{eq:encadr_gphi} and $\gl$ in \eqref{eq:defgl} and fix $\gep_0$ small enough so that \beq
\label{eq:gep0}
\ba
\frac{\beta}{\gep_0}& >2(C_1 \delta^{-2} + \gl(\gb)),\\
 \gep_0^{-1/\gamma}& >4C_1(C_1 \delta^{-2} + \gl(\gb)).
\ea
\eeq
We have to prove that 
\beq
\frac{\sum_{k\neq \kopt}Z_n^{(k)}}{Z_n}\leq \petit,
\eeq
where 
\beq
\label{eq:defZpark}
Z_n^{(k)}=Z_n(H^*_k\leq n < H^*_{k+1}).
\eeq
Recall \eqref{eq:defrecordset}.
We first get rid of all the records that do not lie in $R_{\gep_0}(n)$. In Step $1$ of \cite[Proposition 5.1]{PoiSim19} we proved that on $\Omega_n(\gd,\gep_0)$,
\beq
\ba
&Z_n\geq \frac{c_1}{2\sqrt{n}}\exp\{-N(2C_1\gd^{-2}+2\gl(\beta)+o(1))\},\\
&\sum_{k\in R(N/\gep_0,\8)}  Z_n^{(k)}\leq e^{-\beta N/\gep_0},\\
&\sum_{k\in R(0,\gep_0 N)}  Z_n^{(k)}\leq Z_n(H_{\tau_{\gep_0 N}}>n) \le \exp \Big({-\frac{ \gep_0^{-1/\gamma}N}{2C_1}}\Big).
\ea
\eeq
Together with the choice of parameters in \eqref{eq:gep0}, this implies that if $n$ is large enough
\beq
\frac{\sum_{k\notin R_{\gep_0}(n)}Z_n^{(k)}}{Z_n} \leq \petit.
\eeq
We turn to the records in $R_{\gep_0}(n)$ that are not $\kopt$.
We consider $\eta$ and $\gep$ positive, the latter being small enough so that the conclusion of \cite[Lemma 3.7]{PoiSim19} is valid.
We recall the lower bound on $Z_n$ obtained in \cite[Proposition 4.1]{PoiSim19}: on $\Omega_n(\gd)$,
\beq
\label{eq:binf}
Z_n\geq \exp{\{-N \min_{1<\ell\leq N^{1+\kappa}} G^{\gb}_n(\ell)+o(N)\}}.
\eeq
In Eq.~$(5.20)$ in Step $2$ of \cite[Proposition~5.1]{PoiSim19} we derived an upper bound on $Z_n^{(k)}$ for $k \in R_{\gep_0}(n)$ on $\Omega_n(\gd,\gep_0,\gep)$. If we restrict to $k\neq \kopt$ this upper bound becomes
\beq
Z_n^{(k)} \leq 2C\ n^5\exp\Big(-N \min_{\ell \in \mathcal{R}_{\gep_0}(n)\atop \ell\neq \lopt}  G_n^{\beta-\epsilon} (\ell) +o(N)\Big).
\eeq
Similarly as in Step $3$ and summing only over $k\in R_{\gep_0}(n), k\neq \kopt$, we obtain
\beq
\label{eq:bsup}
\sum_{k\in R_{\gep_0}(n)\atop k\neq \kopt}Z_n^{(k)} \leq 2\gep_0^{-1} C\ n^6\exp\Big(-N \min_{\ell \in \mathcal{R}_{\gep_0}(n)\atop \ell\neq \lopt}  G_n^{\beta-\gep}(\ell)  +o(N)\Big)
\eeq
on $\Omega_n(\gd,\gep_0,\gep)$.

On $\kBone_n(\gep_0)$, $\ell_0(\beta)$ lies in $\mathcal{R}_{\gep_0}(n)$ and we may deduce from \eqref{eq:binf} and \eqref{eq:bsup} that
\beq
\label{eq:avantdernier}
\frac{\sum_{k\in R_{\gep_0}(n)\atop k\neq \kopt}Z_n^{(k)}}{Z_n}\leq 
2 \gep_0^{-1} C\ n^6 \exp{\Big\{-N \Big(\min_{\ell\neq \lopt}  G_n^{\beta-\gep}(\ell) -\min G^{\gb}_n(\ell)+o(1)\Big)\Big\}}.
\eeq
We suppose $\gep>0$ to be small enough so that $|\gl(\beta)-\gl(\gb-\gep)|\leq \frac{\gep_0 \eta}{2}$. This implies that on $A^{(10)}(\gep_0,\gep,\eta) \cap \kBone_n(\gep_0) \cap \kBtwo_n(\eta) $,
\beq
\ba
\min_{\ell\neq \lopt}  G_n^{\beta-\gep}(\ell) -\min G^{\gb}_n(\ell)&= 
 \min_{\ell\neq \lopt}  G_n^{\beta}(\ell) -\min G^{\gb}_n(\ell)  +\min_{\ell\neq \lopt}  G_n^{\beta-\gep}(\ell) - \min_{\ell\neq \lopt}  G_n^{\beta}(\ell)\\
&\geq 2\eta -\max_{\gep_0 N \leq \ell \leq N/\gep_0}|\gl(\ell-1,\gb-\gep)-\gl(\ell-1,\gb)|\frac{\ell-1}{N}\\
&\geq 2\eta-3\frac{\gep_0 \eta}{2}\frac{\gep_0^{-1} N}{N}\geq \frac{\eta}{2}.
\ea
\eeq
Together with \eqref{eq:avantdernier}, this concludes the proof.
\end{proof}

With this weak localization result in hands, we are now able to control the hitting time of the optimal gap.
\begin{proposition}
\label{pr:hitting}
Let $\petit, \kappa\in(0,1)$ and $\gd\in(0,1)$. There exists $\gep_0>0$ so that for $\eta\in(0,1)$, $\rho\in(0,\frac12)$, $\mathsf{C}\ge1$ and for all $\gep>0$ small enough (depending on $\gep_0$ and $\eta$), $n$ large enough and $\tau\in \good_n(\gd,\gep_0,\gep,\eta,\rho, \mathsf{C})$,
\beq
\bP(H_{\tau_{\lopt-1}} >\kappa n|\bgs > n)\le \petit.
\eeq
\end{proposition}

\begin{proof}[Proof of Proposition~\ref{pr:hitting}]
Note that $H^*_{\kopt}=H_{\tau_{\lopt-1}}$. We first observe that for all $0\leq M \leq n$,
\beq
\label{eq:decompohit}
\bP^{\beta}(H_{\tau_{\lopt-1}} > M|\bgs > n)\leq 
\bP^{\beta}(M < H^*_{\kopt}\leq n < H^*_{\kopt+1}  |\bgs   > n)+\bP^{\beta}(n \notin [H^*_{\kopt},H^*_{\kopt+1}) |\bgs> n).
\eeq
Using Proposition \ref{pr:weakloc}, there exists $\gep_0$ such that for $\eta>0$ and $\gep$ small enough, on $ \good_n(\gd,\gep_0,\gep,\eta)$ and for $n$ large enough, 
\beq
\bP^{\beta}(n \notin [H^*_{\kopt},H^*_{\kopt+1}) |\bgs> n)\leq \petit.
\eeq
To manage the first term in the r.h.s. of \eqref{eq:decompohit}, we 
first provide an upper bound on $\bP^{\beta}(M < H^*_{k_0}\leq n < H^*_{k_0+1}  ,\bgs  > n)$.
 By applying the Markov property at $H^*_{k_0}$ and using the Chebyshev inequality, we obtain for all $\phi\geq 0$,
\beq
\label{eq:inter}
\ba
\bP^{\beta}(M &< H^*_{k_0}\leq n < H^*_{k_0+1}  ,\bgs   > n)\\
& \leq \sum_{m=M}^{n} \bP^{\beta}( H^*_{k_0}=m,  n < H^*_{k_0+1}  ,\bgs   > n)\\
& \leq \sum_{m=M}^{n} e^{-\phi m} \bE^{\beta}(e^{\phi H^*_{k_0}}, H^*_{k_0} <  \bgs)\bP_{\tau^*_{\kopt}}^\beta \left( \bgs \wedge H^*_{k_0+1}  >n-m \right).
\ea
\eeq

%%%%%%%%%
\par From Proposition \ref{prop:roughUB} we obtain that on $\kBone_n(\gep_0)$,
\beq
\label{eq:control1}
\bP_{\tau^*_{\kopt}}^\beta \left( \bgs \wedge H^*_{k_0+1}  >n-m \right)\leq Cn^3 e^{-\phi(\beta, T^*_\kopt)(n-m)}.
\eeq
%%%%%%%%%

\par We now turn to the upper bound on $ \bE^{\beta}(e^{\phi H^*_{k_0}}, H^*_{k_0} <  \bgs)$. To this end, let us first recall the notation from~\cite{PoiSim19} : 
\begin{itemize}
\item $({\sf X}_k)_{k\in\bbN_0}$ is the sequence of the obstacles indexes successively visited by the random walk, that is the process defined by $\tau_{{\sf X}_k} = S_{\gt_k}$;
\item $\gz_k^* = \inf\{n\in\bbN\colon {\sf X}_n = i(k)\}$.
\end{itemize}
The idea is to condition the random walk $S$ on $\sf X$ and use the technical Laplace transform estimates in~\cite[Lemma 3.7]{PoiSim19}
with the choice $t=(1-\rho/2)T^*_{k_0}$ and $\phi:=\frac{\pi^2}{2((1-\rho/2)T^*_{k_0})^2}$. We rewrite $\phi:=(1+c(\rho))\frac{\pi^2}{2(T^*_\kopt)^2}$ with the suitable choice of $c(\rho)>0$. At the cost of choosing $\eta>0$ and then $\gep >0$ even smaller (so that the assumptions of Proposition~\ref{pr:weakloc} and~\cite[Lemma 3.7]{PoiSim19} are satisfied) we may assume that
\beq
\label{eq:choicepar}
\lambda'(\gb^-)\frac{\gep}{\gep_0}+{\eta} <\kappa c(\rho) \gep_0^{2}\frac{\pi^2}{2},
\eeq
where $\lambda'(\gb^-)$ is the left derivative of the convex function $\lambda$ at $\gb$, for reasons that will become clear at the end of the proof. 
This choice of $t$ and $\phi$ guarantees that the assumption of~\cite[Lemma 3.7]{PoiSim19} are satisfied on $\kBone_n(\gep_0) \cap A_n^{(5)}(\gep_0,\gep)$ and provide $\alpha(\gep) >0$. This leads us to introduce for $\rho\in (0,1/2)$ (recall the definition of $f$ in Item ~\ref{it:f} of Section \ref{sec:goodenv}),
\beq
\ba
f_{k_0}(\rho)&=2 f\left(\pi \frac{T^*_{k_0-1}}{(1-\rho/2)T^*_{k_0}}\Big[1+\frac{C}{(T^*_{k_0})^2}\Big]\right)\\
\cB_{k_0,\alpha(\gep)}(\rho)& = \{j < i(k_0) \colon  T_{j} > (1-\rho/2) \alpha(\gep) T^*_{k_0}\}\\
L_{k_0,\alpha(\gep)}(\rho)& = |\cB_{k_0,\alpha(\gep)}(\rho)|.
\ea
\eeq
In this way, we obtain (see~\cite[Step 2 in the proof of Proposition~5.1, Eq.~(5.12)-(5.15)]{PoiSim19} for details)
\beq
\bE^{\beta}(e^{\phi H^*_{k_0}}, H^*_{k_0} <  \bgs)\leq \bE^{\gb-\gep}\left(  \big(f_{\kopt}(\rho)\big)^{\sharp\{ i\le \gz_{\kopt}^* \colon |\tau_{{\sf X}_{i-1}}-\tau_{{\sf X}_{i}}| >(1-\rho/2) \alpha(\gep) T^*_{k_0}\}}, H^*_{k_0} <  \bgs\right).
\eeq
We are now left with providing an upper bound for the moment generating function of $\sharp\{i\le \gz_{\kopt}^* \colon |\tau_{{\sf X}_{i-1}}-\tau_{{\sf X}_{i}}| >(1-\rho/2) \alpha(\gep) T^*_{k_0}\}$ under $\bP^\gb(\cdot | H^*_{k_0} <  \bgs)$. This is precisely the role of~\cite[Lemma $3.8$]{PoiSim19}, which gives an upper bound depending on $L_{k_0,\alpha(\gep)}(\rho)$ but requires some assumption on $T^*_\kopt$. From the definition of $h(A,L,\alpha)$ in Item~\ref{it:f} it is clear that the function $h$ is increasing with $A$ and $L$ but decreasing with $\alpha$, so that on $\kBthree_n(\rho)\cap \kBfive_n(\gep_0,\gep,\mathsf{C})$,
\beq
h(f_{k_0}(\rho), L_{k_0,\ga(\gep)}(\rho), \ga(\gep)) \leq h(2 f\left(\pi ({1-\rho/3})\right), \mathsf C,\alpha (\gep)).
\eeq
This implies that for $n$ large enough $(1-\rho/2)T^*_{k_0} > h(f_{k_0}(\rho), L_{k_0,\ga(\gep)}(\rho), \ga(\gep))$. One may finally use \cite[Lemma $3.8$]{PoiSim19} (with $T^*_{k_0}$ replaced by $(1-\rho/2)T^*_{k_0}$) to obtain that on $A^{(3)}_n(\gep_0)$,
\beq
\bE^{\beta}(e^{\phi H^*_{k_0}}, H^*_{k_0} <  \bgs) \leq 2  \big(f_{\kopt}(\rho)\big)^{L_{\kopt,\alpha(\gep)}(\rho)} n \bP^{\gb-\gep}\left( H^*_{k_0} <  \bgs \right).
\eeq
On $\kBthree_n(\rho)$, as $f$ is non-decreasing, for $n$ large enough 
\beq
f_{k_0}(\rho)\le2 f\left(\pi ({1-\rho/3})\right).
\eeq
 Note that on $\kBfive_n(\gep_0, \gep,\mathsf{C})$
\beq
\ba
L_{k_0,\alpha(\gep)}(\rho)&\leq |\{\gep_0 N \leq j \leq \gep_0^{-1}N \colon  T_{j} > \alpha(\gep) (1-\rho/2) \gep_0 N^{\frac{1}{\gamma}} \}|\\
&\leq \Pi_N\Big([\gep_0,\gep_0^{-1}]\times [\frac{\alpha(\gep)\gep_0}{4} ,+\8)\Big)\leq \mathsf C. 
\ea
\eeq
Combining the three previous bounds, we finally get:
\beq
\label{eq:control2}
\bE^{\beta}(e^{\phi H^*_{k_0}}, H^*_{k_0} <  \bgs) 
\leq 2  \big[2 f\left(\pi ({1-\rho/3})\right)\big]^{\mathsf C} n \bP^{\gb-\gep}\left( H^*_{k_0} <  \bgs \right).
\eeq
%%%%
\par With \eqref{eq:control1} and \eqref{eq:control2} in hands, one can turn back to \eqref{eq:inter} to get
\beq
\label{eq:inter2}
\ba
&\bP^{\beta}(M < H^*_{k_0}\leq n < H^*_{k_0+1}  ,\bgs   > n)\\
& \leq 2Cn^4\  \big[2 f\left(\pi ({1-\rho/3})\right)\big]^{\mathsf C} \bP^{\beta-\gep}(H^*_{k_0} <  \bgs) \sum_{m=M}^{n} e^{-\phi m} \ e^{-\phi(\beta, T^*_\kopt)(n-m)}\\
& \leq 2Cn^4 \big[2 f\left(\pi ({1-\rho/3})\right)\big]^{\mathsf C} \bP^{\beta-\gep}(H^*_{k_0} <  \bgs) e^{-\phi(\beta, T^*_\kopt)n -c(\rho)\frac{\pi^2 M}{2(T^*_{\kopt})^2}}\sum_{m=M}^{n}  e^{-\left(\frac{\pi^2}{2(T^*_{\kopt})^2}-\phi(\beta, T^*_\kopt)\right)m}.
\ea
\eeq
Using the expansion of the functions $g$ and $\phi(\gb,\cdot)$ in~\eqref{eq:def:gt} and~\eqref{eq:phi} (see~\cite[Eq. (5.17)-(5.18)]{PoiSim19} for details) there exists some constant $C(\beta)$ so that on $\kBone_n(\gep_0) \cap A_n^{(5)}(\gep_0,\gep)$
\beq
 e^{-\left(\phi(\beta, T^*_\kopt)-g(T^*_{\kopt}) \right)n}  \sum_{m=M}^{n} e^{-\left(\frac{\pi^2}{2(T^*_{\kopt})^2}-\phi(\beta, T^*_\kopt)\right)m}\leq ne^{C(\gb)\gep_0^{-9/2\gamma}n^{\frac{\gamma-1}{\gamma+2}}}.
\eeq
Moreover, on $A^{(10)}_n(\gep_0,\gep,\eta)\cap \kBone_n(\gep_0)$,
\beq
\ba
\bP^{\beta-\gep}(H^*_{k_0} <  \bgs )&\leq \exp\Big\{-\Big(\lambda{(\beta-\gep)}-\frac{\gep_0 \eta}{2}\Big)(\lopt-1)\Big\}\\
&\leq \exp\Big\{-\Big(\lambda(\beta)-\gep \lambda'(\beta^-)-\frac{\gep_0 \eta}{2}\Big)(\lopt-1)\Big\},
\ea
\eeq
where we used for the last inequality that $\lambda$ is convex so that $\lambda{(\beta-\gep)}\geq \lambda(\beta)-\gep \lambda'(\beta^-)$.
Gathering all these intermediate estimates we finally obtain from \eqref{eq:inter2} the bound
\beq
\label{eq:ubn}
\ba
\bP^{\beta}(M < &H^*_{k_0}\leq n< H^*_{k_0+1}  ,\bgs > n)\\
&\leq \exp\Big\{-g(T^*_\kopt)n-c(\rho)\frac{\pi^2}{2(T^*_\kopt)^2}M-\left(\lambda(\beta)-\gep \lambda'(\beta-)-\frac{\gep_0 \eta}{2}\right)\lopt+o(N)\Big\}.
\ea
\eeq

\par To complete our control of the first term in \eqref{eq:decompohit} we turn to the lower bound on $\bP^{\beta}(\bgs   > n)$. 
On $\kBone_n(\gep_0) \cap \Omega_n(\gd)$, one can use \cite[Eq.~(4.8) and Lemma 4.2]{PoiSim19} with $\ell=\lopt$ so that 
\beq
Z_n\geq \frac{c_1}{2\sqrt{n}}  \bP^{\gb}(H^*_{k_0} <  \bgs ) e^{- g(T^*_{k_0}) n + o(N)},
\eeq
where the $o(N)$ is a \emph{deterministic} function of $N$.
On $A_n^{(10)}(\gep_0,\gep,\eta)$, we finally obtain
\beq
\label{eq:lbd}
Z_n\geq \frac{c_1}{2\sqrt{n}} e^{-(\lambda(\gb)+\frac{\gep_0 \eta}{2})(\lopt-1)} e^{- g(T^*_{k_0}) n +  o(N)}.
\eeq
Choosing $M=\kappa n$, we get from \eqref{eq:ubn} and \eqref{eq:lbd} 
\beq
\bP^{\beta}(\kappa n < H^*_{k_0}\leq n < H^*_{k_0+1}  |\bgs   > n)\leq  \exp\Big\{-c(\rho)\frac{\pi^2}{2(T^*_\kopt)^2}\kappa n+(\gep \lambda'(\beta-) + {\gep_0 \eta})\lopt+o(N)\Big\}. 
\eeq
On $\kBone_n(\gep_0)$, $\lopt\leq \gep_0^{-1} N$ and $T^*_{\kopt}\leq \gep_0 ^{-1}N^{\frac{1}{\gamma}}$ so that 
\beq
\ba
\bP^{\beta}(\kappa n < H^*_{k_0}\leq n < H^*_{k_0+1}  |\bgs   > n)\leq   \exp\Big\{-N\Big[c(\rho) \kappa \frac{\pi^2 \gep_0^{2}}{2}-\left(\frac{\gep}{\gep_0}\lambda'(\gb^-)+\eta\right)+o(1)\Big]\Big\}
\ea 
\eeq
and the conclusion follows from the choice of $\gep$ and $\eta$ in \eqref{eq:choicepar}.
\end{proof}

%%%%%%%%%%%%%%%%%%%%%%%%%%%%%%%%%%%%%%%%%%%
\section{A related Markov renewal process}
\label{sec:rmrp}

In this section we introduce a Markov renewal process that is related to the partition function of the polymer and, more generally, to the polymer measure itself. This auxiliary process turns out to be much easier to handle, due to its time-homogeneous Markov structure, and will be the main tool to establish localization results in the next two sections. The interpretation of the partition function using a renewal process is a standard and powerful tool in the field of polymer models. We refer the reader to \cite[Chapter 1]{Gia07} for a simple but instructive example in the context of homogeneous pinning on a defect line as well as more refined models (see also~\cite{Gia11}). This interpretation also plays a central role in the study of polymers interacting with multiple interfaces, see~\cite{CP09b,CP09}. As for {\it Markov} renewal processes, they have been previously applied to the pinning model with periodic disorder~\cite[Chapter 3]{Gia07} or with an annealed correlated Gaussian disorder~\cite{Poisat2012, Poisat2013b} and to the strip wetting model~\cite{Sohier_2013}.
\par Let us outline the content of the present section. In Section \ref{sec:markov-ren-int}, we introduce a transformation of the environment and the Markov renewal process with law $\cP$, see \eqref{def:cP}. We point its relation with the partition function in Proposition \ref{pr:ren-int}. In Section \ref{subsec:fee}, we start our study of $\cP$ and provide estimates on the free energy (defined in \eqref{eq:defel}) by an extensive use of comparisons with other environments (see Lemmas \ref{lem:sandw:phit1t2} and \ref{lem:est-scd-order-FE}). Section \ref{subsec:te} is dedicated to the proof of two technical estimates (Lemmas \ref{lem:Zout} and \ref{eq:control.ratio.h}) that will later be decisive for the description of $\cP$. In Section \ref{subsec:tpel} we begin this description by proving that the transition probabilities of the auxiliary Markov chain are non-degenerate (see Lemma \ref{lem:binf.proba.out}) and by providing estimates on the first two moments of the excusions lengths, depending on their type (see Lemma \ref{lem:mom-exc}). We complete our description in Section \ref{subsec:mrfsp} with a control on the mass renewal function (see Lemma \ref{lem:mass-ren-fct}) leading to estimates on the survival probability of the random walk starting from the lower boundary of the localization interval (see Lemma \ref{lem:sharp-Zn-renorm}).
\subsection{A Markov renewal interpretation of the partition function} 
\label{sec:markov-ren-int}
Throughout this section, we consider that $\tau$ belongs to $\good_n$, which implies in particular that $\lopt(n)$ is uniquely defined.
\par \textit{Re-parametrization.} To simplify, we shall write 
\beq
\tone = \tone(n) := T_{\ell_0(n)}, \qquad
\ttwo = \ttwo(n) = \max_{i\neq \lopt, i< i(\kopt+1)} T_i.
\eeq
Note that $\ttwo$ corresponds to the second largest gap in the interval $[0,\tau^*_{\kopt+1}]$ that appears in the event $\kBthree_n(\rho)$, which now reads:
\beq
\kBthree_n(\rho) = \{\rho \tone < \ttwo < (1-\rho) \tone \}.
\eeq

We replace all gaps after $T^*_{\kopt+1}$ by gaps of size $1$: For $k\geq 1$,
\beq
\overline{T}_k=
\begin{cases}
\ba
 T_k \quad &\text{if }k < i(\kopt+1)\\
1 \quad &\text{if }k \geq i(\kopt+1).
\ea
\end{cases}
\eeq
The environment $\mtau$ is then defined by $\mtau_0=0$ and, for $n\geq 1$, 
\beq
\mtau_n-\mtau_{n-1}=\overline{T}_n. 
\eeq
Proposition \ref{pr:weakloc} indicates that with high probability the walk does not reach the $i(\kopt+1)$-th gap so that the behaviour of the walk under $\bP(\ \cdot\ | \bgs(\tau)>n)$ or $\bP(\ \cdot\ | \bgs(\mtau)>n)$ can be compared (this will be made precise in the proof of Theorem~\ref{thm:ggasup1}).
Let us also write 
\beq
\label{eq:deftau+-}
\ba
\tau_k^+ &= \mtau_{\lopt+k} - \mtau_{\lopt}, \qquad &k\in\bbN, \\
\tau_k^- &= \mtau_{\lopt-1} - \mtau_{\lopt-1-k}, \qquad &k\in\{1,\ldots, \lopt-1\},
\ea
\eeq
so that we will sometimes refer to $\mtau$ as $(\tau^-,\tone, \tau^+)$. This notation will be specially handy when we later compare $\mtau$ to other environments where $\tau^-$ and $\tau^+$ are modified.

\par We define, for all $m\geq 1$, $a,b\in\{0,1\}$ and all events $A$,
\beq
\ba
 \mZ_m^a(A) &= \bP_{\tau_{\lopt-1}+a\tone}(\bgs(\mtau) > m,A),\qquad \mZ_m^a =\mZ_m^a (\Omega),\\
\mZ_m^{ab}(A)&= \bP_{\tau_{\lopt-1}+a\tone}( \bgs(\mtau)> m, S_m = \tau_{\lopt-1}+b\tone,A),\qquad \mZ_m^{ab}=\mZ_m^{ab}(\Omega).
\ea
\eeq
To lighten notation we may later use $\mZ_m=\mZ_m^0$ and $\mZ_m(A)=\mZ_m^0(A)$.\\

\par \textit{Decomposition of the partition function and auxiliary Markov renewal process.}
\label{sec:defcP}
We recall that the visits to the boundary of the optimal gap are denoted by $(\bar \gt_i)_{i\ge 0}$ (see \eqref{eq:defthetabar}). We define
\beq
\ba
X_i &= \frac{S_{\bar \gt_i}-\tau_{\lopt-1}}{\tone}\in\{0,1\}, \qquad i\geq 0\\
Y_i &=
\begin{cases}
\ins & \text{if} \quad S_{(\bar \gt_{i-1}, \bar \gt_i)} \subset I_\loc\\
\out &\text{if} \quad S_{(\bar \gt_{i-1}, \bar \gt_i)} \cap I_\loc = \emptyset
\end{cases}
\qquad i\geq 1.
\ea
\eeq

Recall~\eqref{eq:defRuinproba}. In this section, we need to consider ruin probabilities depending on the exit point of the random walk. For convenience, we use the same notation as in~\cite{CP09b} (see (2.5) and (2.6) therein), that is, for $t\in\bbN$:
\beq
\label{eq:defRuinproba2}
\ba
q_t^0(n) &= \bP_0(\inf\{k>0\colon S_k \in \{-t,0,t\}\}= n,\ S_n = 0),\\
q_t^1(n) &= \bP_0(\inf\{k>0\colon S_k \in \{-t,0,t\}\}= n,\ S_n = t),\\
\ea
\eeq
so that for all $n\in \bbN$,
\beq
 q_t(n) = q_t^0(n) + 2 q_t^1(n).
\eeq
We refer to Appendix~\ref{sec:mgf-qt} for useful results on the moment generating functions of these sequences.

\par For $a,b \in \{0,1\}$ and $\ell\in\bbN$, define 
\beq
\label{eq:lienZq}
\ba
\mK_{\ins}(a, b, \ell)&:=
\mZ^a_{\ell}(\bar \gt_1 = \ell, Y_1 = \ins, X_1 = b)
=
\begin{cases}
\frac{e^{-\beta}}{2}q_{\tone}^0(\ell) \quad \textrm{if } a=b\\
{e^{-\beta}}q_{\tone}^1(\ell) \quad \textrm{if } a\neq b,
\end{cases}\\
\mK_{\out}(a, \ell)&:=\mZ^a_{\ell}(\bar \gt_1 = \ell, Y_1 = \out),
\ea
\eeq
and
\beq
\label{def Zexc2}
\mK(a, b, y; \ell) := 
\begin{cases}
\mK_{\ins}(a, b, \ell) & (y = \ins)\\
\mK_{\out}(a, \ell) & (y = \out \textrm{ and } a=b)\\
0 & \textrm{ otherwise}.
\end{cases}
\eeq
We define for all events $A$,
\beq
\mZ_m^\pin(A) := \mZ_m(A, m\in \bar \gt),
\eeq
and the pinned partition function $\mZ_m^\pin := \mZ_m^\pin(\Omega)=\mZ_m^{00}+\mZ_m^{01}$, which we decompose according to the contact points with the interfaces:
\beq
\label{eq:decomposeZn}
\mZ_m^\pin = 
\sum_{k=1}^m
\sum_{0<u_1 < \ldots < u_k=m} 
\sumtwo{x_1, \ldots, x_k \in \{0,1\}}%
{y_1, \ldots, y_k \in \{\ins, \out\}}
\prod_{i=1}^k
\mK(x_{i-1}, x_i, y_i; u_i - u_{i-1}),
\eeq
with the convention that $u_0=x_0=0$. Note that the unpinned partition function is related to the pinned one via the relation:
\beq
\mZ_m = \sumtwo{0\le k \le m}{a\in\{0,1\}} \mZ_k^{0a}\;\mZ_{m-k}^{a} (\bar \gt_1 > m-k).
\eeq

Next, we define for $a,b\in\{0,1\}$ and $f\ge 0$ (recall~\eqref{eq:mgf}), 
\beq
\hat \mK(a, b, f) = \hat \mK_{\ins}(a, b, f) +\hat \mK_{\out}(a, f) \ind_{\{a=b\}}.
\eeq
%%%
From \eqref{eq:lienZq}  and Lemma~\ref{lem:estim-qtn}, for all $f<g(\tone)$ and $a,b\in \{0,1\}$, $\hat \mK_{\ins}(a,b,f)<+\8$. Moreover, from \cite[Proposition 3.5]{PoiSim19} and Proposition~\ref{pr:homo}, for all $a\in \{0,1\}$ and $f< \phi(\ttwo)$,
\beq
\hat \mK_{\out}(a,f) \leq \sum_{\ell\ge 1} \bP(\gs(\ttwo\Z)>\ell)e^{f\ell}<+\8.
\eeq

%%%
We define the $\{0,1\}^2$-matrix
\beq
\label{eq:matrice}
\ba
\hat \mK(\cdot,\cdot,f)&:=\begin{pmatrix}
\hat \mK_{\ins}(0,0,f) +\hat \mK_{\out}(0,f)  &\hat \mK_{\ins}(0,1,f) \\
\hat \mK_{\ins}(1,0,f)  &\hat \mK_{\ins}(1,1,f) +\hat \mK_{\out}(1,f)
\end{pmatrix}
\\
&=\begin{pmatrix}
\hat \mK_{\ins}(0,0,f) +\hat \mK_{\out}(0,f)  &\hat \mK_{\ins}(0,1,f) \\
\hat \mK_{\ins}(0,1,f)  &\hat \mK_{\ins}(0,0,f) +\hat \mK_{\out}(1,f)
\end{pmatrix}
\ea
\eeq 
where we have used that $\hat \mK_{\ins}(1,1,f)=\hat \mK_{\ins}(0,0,f)$ and $\hat \mK_{\ins}(0,1,f)=\hat \mK_{\ins}(1,0,f)$.
Since the matrix $\hat \mK(\cdot,\cdot,f)$ has positive entries, we may define $\gL(f)$ as its Perron-Frobenius eigenvalue. We recall that (see~\cite[Proof of Theorem 1.1]{S06})
\beq
\gL(f) = \sup_{v \in \R^2\setminus \{0\} } \min_{i} \frac{(\hat \mK(\cdot,\cdot,f)v)_i}{v_i}
\eeq
so that $f\to \gL(f)$ is continuous and increasing on $[0,g(\tone))$. 
\par For all $\ell\in\bbN$, let
\beq
\mK_{\ins}(\ell) := \mK_{\ins}(0,0,\ell) + \mK_{\ins}(0,1,\ell)
= \mK_{\ins}(1,0,\ell) + \mK_{\ins}(1,1,\ell).
\eeq
We shall later use the following routine inequality:
\begin{lemma} 
\label{lem:sandwPF}
For all $f\ge0$,
\beq
\min(\hat \mK_{\out}(0,f),\hat \mK_{\out}(1,f))
\le
\gL(f) - \hat \mK_{\ins}(f)
\le
\max(\hat \mK_{\out}(0,f),\hat \mK_{\out}(1,f)).
\eeq
\end{lemma}
\begin{proof}[Proof of Lemma~\ref{lem:sandwPF}] Use Corollary~1 in~\cite[Section 1.1]{S06}.
\end{proof}
This implies that $\gL(0)<1$ and since
\beq
\hat \mK(0,0,f) \ge \hat \mK_{\ins}(0,0,f) = \frac{e^{-\beta}}{2} \hat q_{\tone}^0(f) \to \infty, 
\eeq
as $f\uparrow g(\tone)$, we conclude that the equation
\beq
\label{eq:defel}
 \gL(f)  = 1
\eeq
admits a unique solution in $(0,g(\tone))$ that we denote by $\phi = \phi(\tau^-,\tone, \tau^+) $ and call \textit{free energy}.
Thus, there exists a map $h\colon \{0,1\} \mapsto (0,\infty)$ such that
\beq
\label{eq:decompo}
\sum_{b \in \{0,1\}} \hat \mK(a, b, \phi) \frac{h(b)}{h(a)} = 1, 
\qquad 
\forall a\in \{0,1\}.
\eeq
The relation in~\eqref{eq:decompo} allows us to define the law of a Markov renewal process $\cP$ (which actually depends on $\gb$ and on the environment of obstacles) such that for all $i\in \bbN$, $b\in\{0,1\}$, $y\in\{\ins,\out\}$,
\beq
\label{def:cP}
\cP(\bar \gt_i -  \bar \gt_{i-1}= \ell, X_i = b, Y_i = y | \bar \gt_j, X_j, Y_j,\ j<i)
= \mK(X_{i-1}, b, y; \ell) e^{\phi\ell} \frac{h(b)}{h(X_{i-1})},
\eeq
with $\bar\gt_0= 0$ and $X_0 = 0$ as initial state. In the rest of the paper, we shall use subscripts to specify the starting point of the modulating chain (which is zero by default):
\beq
\cP_a(\cdot) = \cP(\cdot | X_0 = a).
\eeq
We use both subscripts and superscripts to specify the starting and ending points, namely, for $m\in\bbN_0$,
\beq
\cP_a^b(m\in \bar\gt, A) = \sum_{k\ge 0} \cP(\bar \gt_k = m, X_k = b, A | X_0 = a).
\eeq

Let us now explain the link between the Markov renewal process $\cP$ defined in~\eqref{def:cP} and the polymer measure. First, let $\cF = \gs(\bar \gt, X)$ and $\cF_i = \gs(\bar \gt_j, X_j, j\le i)$ for $i\in \bbN$. Define for $m\in \bbN$ the following $\gs$-algebra:
\beq
\hat \cF_m = \{A \in \cF \colon \forall i\ge 1, A \cap \{\bar \gt_i = m\} \in \cF_i \}.
\eeq
From \eqref{eq:decomposeZn} and \eqref{def:cP}, one can deduce
\begin{proposition}
\label{pr:ren-int}
For all $m\ge 1$ and $A\in \hat \cF_m$,
\beq
\mZ_m^\pin(A) e^{\phi m} = \cP_0^0(A, m\in\bar\gt) + \frac{h(0)}{h(1)} \cP_0^1(A, m\in\bar\gt).
\eeq
\end{proposition}
\subsection{Free energy estimates}
\label{subsec:fee}

Let us bound the free energy $\phi = \phi(\tau^-,\tone, \tau^+)$ defined in Section~\ref{sec:defcP} by simpler quantities corresponding to the free energy for less and more favorable environments of obstacles.
Let $\phi(t)$ be the free energy corresponding to the periodic environment $t\bbZ$ (see Proposition~\ref{pr:homo}) and recall that $g(t)$ is given in~\eqref{eq:def:gt}. Then, we have the following
\begin{lemma} 
\label{lem:sandw:phit1t2}
The following inequalities hold:
\beq
\label{eq:sandw:phit1t2}
\phi(\tone) \le \phi(\tau^-,\tone,\tau^+) \le \phi(\bbN,\tone,\bbN) < g(\tone).
\eeq
\end{lemma}
\begin{proof}[Proof of Lemma~\ref{lem:sandw:phit1t2}]
For the third inequality, there is nothing to prove as, by construction, $\phi(\tau^-,\tone,\tau^+)$ belongs to $(0,g(\tone))$ in particular when $\tau^-=\tau^+=\bbN$. Let us now deal with the second inequality. By a straightforward pathwise comparison we obtain, with obvious  notation,
\beq
\label{eq:pathwise-comp}
\mK_{\out}(a,\ell)  \ge \mK_{\out}^{\bbN}(a,\ell), \qquad a\in\{0,1\}.
\eeq
Since the Perron-Frobenius eigenvalue of a matrix is non-decreasing coordinatewise (as soon as it exists), we get the result. We now turn to the first inequality, which is slightly more involved. We fix $f< \phi(\tone)$ and prove that $\Lambda(f)<1$.
Let us first recall the distinction between
\beq
\gt = \{m\in \bbN_0\colon S_m \in \mtau\}
\eeq
and
\beq
\bar \gt = \{m\in \bbN_0\colon S_m \in \{\mtau_{\lopt-1}, \mtau_{\lopt-1}+\tone\}\}.
\eeq
By using~\cite[Proposition~3.5]{PoiSim19} we could compare the free partition functions (i.e. without any constraint on the final point) for the environments $(\tau^-,\tone, \tau^+)$ and $\tone\bbZ$ but it does not seem so easy to compare any of the pinned versions (that is with the constraint $\ell \in \gt$ or $S_\ell = 0$). However we may write
\beq
\label{eq:compareZf}
\sum_{\ell\ge 0} \mZ_\ell(\ell \in \gt) e^{f\ell}
\le \sum_{\ell\ge 0} Z^{\tone\bbZ}_\ell(\ell \in \gt) e^{f\ell},
\eeq
where the superscript on top of the partition function $\mZ_\ell$ indicates a change of environment.
Indeed, 
\beq
\mZ_\ell(\ell \in \gt) = \sum_{i\ge 0} \mZ_\ell(\gt_i = \ell),
\eeq
so
\beq
\sum_{\ell\ge 0} \mZ_\ell(\ell \in \gt) e^{f\ell} = \sum_{i\ge 0} e^{-\gb i} \bE^{\mtau}_{\tau_{\lopt-1}}[e^{f\gt_i}].
\eeq
By~\cite[Proposition~3.5]{PoiSim19},
\beq
\bE^{\mtau}_{\tau_{\lopt-1}}[e^{f\gt_i}] \le \bE^{\tone\bbZ}[e^{f\gt_i}]
\eeq
and we get the desired inequality. Moreover, note that
\beq
\sum_{\ell\ge 0} Z^{\tone\bbZ}_\ell(\ell \in \gt) e^{f\ell} = \frac{1}{1-e^{-\gb} \hat q_{\tone}(f)} < +\infty,
\eeq
so that 
\beq
\label{eq:sumZffini}
\sum_{\ell\ge 0} \mZ_\ell(\ell \in \gt) e^{f\ell}<+\8.
\eeq 
\par We denote by $h_f$ an eigenvector of $\hat \mK(\cdot,\cdot,f)$ associated to $\Lambda(f)$ and analogously to \eqref{def:cP}, define the law $\cP^f$ of a Markov renewal process such that for all $i\in \bbN$, $a\in\{0,1\}$, $y\in\{\ins,\out\}$,
\beq
\label{def:cPf}
\cP^f(\bar \gt_i -  \bar \gt_{i-1}= \ell, X_i = a, Y_i = y | \bar \gt_j, X_j, Y_j,\ j<i)
= \mK(X_{i-1}, a, y; \ell) \frac{h_f(a)}{h_f(X_{i-1})}\frac{e^{f \ell}}{\Lambda(f)}.
\eeq
Similarly to Proposition \ref{pr:ren-int}, we obtain for all $\ell\geq 0$, 
\beq
\mZ_\ell(\ell \in \gt)\ge \mZ_\ell^{00}(\ell \in \gt)=e^{-f \ell} \sum_{k\geq 1}\Lambda(f)^k \cP^f(\bar \gt_k=\ell, X_k=0),
\eeq
so that
\beq
\sum_{\ell\ge 0} \mZ_\ell(\ell \in \gt)e^{f \ell} \ge\sum_{k\geq 1}\Lambda(f)^k \cP^f(X_k=0).
\eeq
Together with \eqref{eq:compareZf} and \eqref{eq:sumZffini}, we deduce that 
\beq
\label{eq:laststep}
\sum_{k\geq 1}\Lambda(f)^k \cP^f(X_k=0)<+\8.
\eeq
The reader may check that the Markov chain $(X_k)_{k\ge 0}$ has non-degenerate transition probabilities. Moreover, it is irreducible and aperiodic, so that $(\cP^f(X_k=0))_{k\geq 0}$ converges when $k$ goes to infinity to some real number in $(0,1)$ that is the probability of $0$ under the invariant probability of $X$. Together with \eqref{eq:laststep} this implies that $\Lambda(f)<1$ and that concludes the proof.
\end{proof}
Lemma \ref{lem:sandw:phit1t2} provides the first order term in the expansion of the free energy as $\tone$ is large. We now estimate the second order term. To that purpose, let us define $\gep(\tau^-,\tone, \tau^+)$ by
\beq
\label{eq:assphi}
\phi(\tau^-,\tone, \tau^+) = \frac{\pi^2}{2\tone^2}\Big[1- \gep(\tau^-,\tone, \tau^+)\Big].
\eeq
Then, we have:
\begin{lemma}
\label{lem:est-scd-order-FE}
As $n\to\infty$ (enforcing on $\good_n$ that $\tone\to \infty$),
\beq
\frac{4}{e^\gb(1+\sqrt{1-e^{-2\gb}})-1} \le  \liminf \tone \gep(\tau^-,\tone, \tau^+) \le \limsup \tone \gep(\tau^-,\tone, \tau^+) \le \frac{4}{e^{\gb}-1}.
\eeq
\end{lemma}

\begin{proof}[Proof of Lemma~\ref{lem:est-scd-order-FE}]
The idea is to use~Lemma~\ref{lem:sandw:phit1t2} and estimate $\phi(\tone)$ and $\phi(\bbN,\tone,\bbN)$. From~\cite{CP09b} we get
\beq
\label{eq:asympt:phit}
\phi(t) = \frac{\pi^2}{2t^2}[1 - \gep(t,t)]
\qquad \textrm{where} \quad \gep(t,t) \sim \frac{4}{(e^{\gb}-1)t},
\eeq
and this gives the rightmost inequality.
Let us now estimate $\phi(\bbN,\tone,\bbN)$, which corresponds to the free energy when there are interfaces everywhere outside of the optimal gap, that is $\tau^- = \tau^+ = \bbN$. 
In this case we get an explicit expression for $\hat \mK_{\out}^{\bbN}(\phi)$, where $\phi=\phi(\bbN,\tone,\bbN)$ now only depends on $\tone$. Indeed, we have
\beq
\mK_{\out}^{\bbN}(\ell) = \frac12 e^{-\gb  \ell} q_\infty^0(\ell),
\eeq
hence, by~\eqref{eq:mgf:qinfty},
\beq
\label{eq:comp:Zout:N}
\hat \mK_{\out}^{\bbN}(\phi) = \frac12 \hat q_\infty^0(\phi-\gb) = \frac12 (1 - \sqrt{1- e^{-2\gb}}) + o(1),
\eeq
in the limit of large $\tone$. We may conclude in this case that
\beq
\label{eq:eps:unfav.env}
\phi(\bbN,\tone,\bbN) = \frac{\pi^2}{2t^2}[1 - \gep(\tone,1)]
\qquad \textrm{where} \quad
\gep(\tone,1) \sim \frac{4}{\tone[e^\gb(1+\sqrt{1-e^{-2\gb}})-1]},
\eeq
in the limit of large $\tone$. By combining Lemma~\ref{lem:sandw:phit1t2}, \eqref{eq:asympt:phit} and~\eqref{eq:eps:unfav.env}, we finally get the result.
\end{proof}

\subsection{Technical estimates}
\label{subsec:te}
The first result of this section is an upper bound on $\mK_{\out}$ that turns out to be very useful to control the length of the excursions outside of the localization interval.
\begin{lemma} 
\label{lem:Zout}
There exists $C>0$ such that for all $J\ge 1$, $\tau \in \good_n(J)$ and $\ell \geq J^{\frac{\gamma+2}{\gamma}}$,
\beq
\mK_{\out}(a,\ell) \leq C\ell^3 e^{-\ell^{\frac{\gamma}{2(\gamma+2)}}},\qquad a\in\{0,1\}.
\eeq
\end{lemma}

\begin{proof}[Proof of Lemma~\ref{lem:Zout}] 
We only prove the case $a=1$ as the other one is similar. 
Let $\ell \ge J^{\frac{\gamma+2}{\gamma}}$. By Proposition~\ref{prop:roughUB} we have for $j\ge 1$,
\beq
\ba
\mK_{\out}(1,\ell) &\le \bP(\bgs(\tau^+)> H_{\tau^+_j})
+ \bP(\bgs(\tau^+) \wedge H_{\tau^+_j} >\ell-1)\\
&\le e^{-\gb j} + C \ell^3 \exp\Big[-\phi\Big(\max_{1\le i \le j} T_{\ell_0+i} \Big)\ell\Big].
\ea
\eeq
By picking $j = \ell^{\frac{\gga}{\gga+2}}$ (so that $j\geq J$) and using that $\tau \in \kBfour_n(J)$ together with~\eqref{eq:phi}, we get that
\beq
\phi\Big(\max_{1\le i \le j} T_{\ell_0+i} \Big)\geq {\rm (cst)} \ell^{-\frac{\gamma+4}{2(\gga+2)}}.
\eeq
This gives
\beq
\mK_{\out}(1,\ell) \le e^{-\gb \ell^{\frac{\gga}{\gga+2}}} + C\ell^3e^{- \ell^{\frac{\gamma}{2(\gamma+2)}}},
\eeq
from which the result follows.
\end{proof}
The second result of this section is a control on the ratio $h(a)/h(1-a)$ ($a\in \{0,1\}$) that appears in the definition of $\cP$, see~\eqref{def:cP}.
\begin{lemma}
\label{eq:control.ratio.h}
There exists $C>0$ (not depending on $\tau\in\good_n$) such that,
\beq
\frac1C \le \frac{h(1)}{h(0)} \le C.
\eeq
\end{lemma}

As a consequence of Lemma~\ref{eq:control.ratio.h} and Proposition~\ref{pr:ren-int}, there exists a constant $C>0$ such that for all $m\in\bbN$ and $A\in \hat \cF_m$,
\beq
\label{eq:comp-Zpin}
\frac1C\ \cP(A, m\in\bar\gt) \le \mZ_m^\pin(A)e^{\phi m} \le C\ \cP(A, m\in\bar\gt).
\eeq
We repeatedly use the latter inequality in the following.

\begin{proof}[Proof of Lemma~\ref{eq:control.ratio.h}]
From \eqref{def:cP}, we have for $a\in\{0,1\}$,
\beq
\cP_a(X_1=1-a)= \hat \mK_{\ins}(a,1-a,\phi) \frac{h(1-a)}{h(a)}.
\eeq
Recall that $\phi \ge \phi(\tone)$, see Lemma~\ref{lem:sandw:phit1t2}. We have, using~\eqref{eq:lienZq} and Lemma~\ref{lem:asympt-q},
\beq
\label{eq:control.ratio.h.aux}
 \hat \mK_{\ins}(a,1-a,\phi) =e^{-\beta}\hat q_{\tone}^1(\phi) 
 \sim \frac{e^{-\gb}}{\tone \gep(\tau^-, \tone, \tau^+)},
 \eeq
and we conclude the proof thanks to Lemma~\ref{lem:est-scd-order-FE} as $\cP_a(X_1=1-a)\leq 1$.
\end{proof}

\subsection{Transition probabilities and excursions lengths}
\label{subsec:tpel}
We first observe that the transition probability of $\cP$ defined in \eqref{def:cP}, does not depend on $Y_{i-1}$.
Also, the laws of the excursion length conditional on the transitions of the modulating chain $(X,Y)$ do not depend on the function $h$. Indeed, we get for $a \in\{0,1\}$ and $\ell \in \bbN$:
\beq
\label{eq:explicitcP1}
\ba
&\cP(\bar \gt_1 = \ell | Y_1=\ins,X_1=a,X_0=a)= \frac{\mK_{\ins}(a,a,\ell)e^{\phi \ell}}{\hat \mK_{\ins}(a,a,\phi)},\\
&\cP(\bar \gt_1 = \ell | X_1=1-a,X_0=a)= \frac{\mK_{\ins}(a,1-a,\ell)e^{\phi \ell}}{\hat \mK_{\ins}(a,1-a,\phi)},\\
&\cP(\bar \gt_1 = \ell | Y_1=\out,X_0=a)= \frac{\mK_{\out}(a,\ell)e^{\phi \ell}}{\hat \mK_{\out}(a,\phi)}.
\ea
\eeq
Note that it is also the case for the probability to transit from one side of the optimal gap to the same side, as for $a\in\{0,1\}$,
\beq
\label{eq:explicitcP2}
\ba
&\cP( Y_1=\ins,X_1=a | X_0=a)= \hat \mK_{\ins}(a,a,\phi) = \hat \mK_{\ins}(1-a,1-a,\phi) ,\\
&\cP( Y_1=\out | X_0=a)= \hat \mK_{\out}(a,\phi),
\ea
\eeq
while the probability to change side (interpreted as the random walk crossing the optimal gap) is, for $a\in\{0,1\}$,
\beq
\label{eq:explicitcP3}
\cP( X_1=1-a | X_0=a)= \hat \mK_{\ins}(a,1-a,\phi) \frac{h(1-a)}{h(a)} = \hat \mK_{\ins}(1-a,a,\phi) \frac{h(1-a)}{h(a)}.
\eeq

\begin{lemma}
\label{lem:binf.proba.out}
There exists $\gd_0\in(0,1)$ (not depending on $\tau\in\good_n$) such that for $a,b\in~\{0,1\}$,
\beq
\ba
&\gd_0 \le \cP_a(Y_1 = \out) \le 1-\gd_0,\\
&\gd_0 \le \cP_a(Y_1 = \ins, X_1=a) \le 1-\gd_0,\\
&\gd_0 \le \cP_a(X_1=b) \le 1-\gd_0. 
\ea
\eeq
\end{lemma}

\begin{proof}[Proof of Lemma~\ref{lem:binf.proba.out}]
\par We only have to prove the lower bounds as for a fixed $a\in \{0,1\}$ the three probabilities sum up to one. We start with the first inequality. From~\eqref{eq:explicitcP2} and~\eqref{eq:pathwise-comp} we have
\beq
\label{eq:transout}
\cP_a(Y_1 = \out) = \hat \mK_{\out}(a, \phi) \ge \hat \mK_{\out}^{\bbN}(\phi).
\eeq
Recall from Lemma~\ref{lem:sandw:phit1t2} that $\phi \ge \phi(\tone)$. Arguing as in~\eqref{eq:comp:Zout:N} and using only that $\phi(\tone)$ converges to zero as $\tone$ tends to infinity, we get
 \beq
 \hat \mK_{\out}^{\bbN}(\phi) \ge  \hat \mK_{\out}^{\bbN}(\phi(\tone)) \to \frac12(1-\sqrt{1-e^{-2\gb}}),
 \eeq
 as $\tone\to \infty$. 
As a side remark, let us notice that the r.h.s goes to $0$ when $\beta$ goes to infinity and $1/2$ when $\beta$ goes to $0$.

\par We turn to the second inequality. Using~\eqref{eq:explicitcP2} and Lemma~\ref{lem:asympt-q},
\beq
\cP_a(Y_1 = \ins, X_1=a) = \hat \mK_{\ins}(0,0,\phi) = \frac{e^{-\gb}}{2} \hat q^0_{\tone}(\phi) = \frac{e^{-\gb}}{2}\Big[1+\frac{2+o(1)}{\tone\gep(\tau^-, \tone, \tau^+)}\Big],
\eeq
and one may conclude with the help of Lemma~\ref{lem:est-scd-order-FE}.

\par We turn to the last probability, that is the one to cross $(0,\tone)$. By~\eqref{eq:explicitcP3},  
\beq
\cP_0( X_1=1)= \hat \mK_{\ins}(0,1,\phi) \frac{h(1)}{h(0)},
\eeq
and one may conclude the proof thanks to Lemma~\ref{eq:control.ratio.h} (see also \eqref{eq:control.ratio.h.aux} in the proof of that lemma).
\end{proof}

The next lemma gives a lower bound on the distribution of the first return time to the endpoints of the localization interval.
\begin{lemma}
\label{lem:cP_tail}
There exists $C>0$ (not depending on $\tau\in\good_n$) such that for all $a\in\{0,1\}$, all integers $u,v\in 2\bbN$ such that $u<v$ and all $\tone$ large enough,
\beq
\cP_a(u \le \bar \gt_1 <v) \ge C \Big( e^{-[g(\tone)-\phi]u} - e^{-[g(\tone)-\phi]v} \Big).
\eeq
\end{lemma}
\begin{remark}
Letting $v$ go to infinity in the above equation, and using Lemma \ref{lem:sandw:phit1t2} and \eqref{eq:def:gt} and \eqref{eq:phi}, one obtains that
\beq
\cP_a(\bar \gt_1 \ge u) \ge C  \exp\left\{-{\rm(cst.)}\frac{u}{\tone^3}\right\},
\eeq
which entails that the expectation of $\bar \gt_1$ under $\cP_a$ is bounded from below by a constant multiple of $\tone^3$. We will provide more precise estimates on the expectation of the excursions in Lemma \ref{lem:mom-exc}.
\end{remark}

\begin{proof}[Proof of Lemma~\ref{lem:cP_tail}] Let $a\in\{0,1\}$ and $u<v$. We first restrict to inward excursions:
\beq
\cP_a(u \le \bar\gt_1 <v) \ge \cP_a(u \le \bar\gt_1 <v, Y_1= \ins).
\eeq
Recall \eqref{eq:explicitcP1}. We obtain
\beq
\cP_a(u \le \bar\gt_1 <v) \ge \sum_{u\le \ell < v} \sum_{b\in\{0,1\}} \mK_{\ins}(a,b,\ell) e^{\phi\ell} \frac{h(b)}{h(a)},
\eeq
and by Lemma~\ref{eq:control.ratio.h},
\beq
\cP_a(u \le \bar\gt_1 <v) \ge C \sum_{u\le \ell < v} \mK_{\ins}(\ell) e^{\phi\ell}.
\eeq
Using~\eqref{eq:lienZq}, we get
\beq
\cP_a(u \le \bar\gt_1 <v) \ge C \sum_{u\le \ell < v} q_{\tone}(\ell) e^{\phi\ell}.
\eeq
From here, the proof follows as in~\cite[Lemma 2.2]{CP09b}: use Lemma~\ref{lem:estim-qtn} and Lemma~\ref{lem:est-scd-order-FE} (see also~\eqref{eq:assphi}).
\end{proof}

 We turn to the control of the first two moments of the interarrival times. Let $(\xi^{(\out)}_{0,i})_{i\ge 1}$ be the length of the excursions from $0$ to $0$ outside the optimal gap, $(\xi^{(\ins,0)}_{0,i})_{i\ge 1}$ be the length of the excursions from $0$ to $0$ inside the optimal gap and finally $(\xi^{(\ins,1)}_{0,i})_{i\ge 1}$ be the length of the excursions from $0$ to $1$.
The laws of $\xi^{(\out)}_{0,i}$, $\xi^{(\ins,0)}_{0,i}$ and $\xi^{(\ins,1)}_{0,i}$ are thus that of $\bar\gt_1$ conditioned on $\{X_0 = X_1=0,Y_1=\out\}$, $\{X_0 = X_1=0,Y_1=\ins\}$ and $\{X_0 = 0, X_1=1\}$ respectively. We analogously define the variables $(\xi^{(\out)}_{1,i})_{i\ge 1}$, $(\xi^{(\ins,1)}_{1,i})_{i\ge 1}$  and $(\xi^{(\ins,0)}_{1,i})_{i\ge 1}$. These are sequences of i.i.d. random variables. Moreover $(\xi^{(\ins,0)}_{0,i})_{i\ge 1}$ and $(\xi^{(\ins,1)}_{1,i})_{i\ge 1}$ have the same law, as do $(\xi^{(\ins,1)}_{0,i})_{i\ge 1}$ and $(\xi^{(\ins,0)}_{1,i})_{i\ge 1}$. We refer to the first variable of each sequence by omitting the time index that is, for example, $\xi^{(\out)}_{0}=\xi^{(\out)}_{0,1}$. The expectation with respect to $\cP$ is denoted by $\cE$.

\begin{lemma}[Control of the first two moments of the conditioned excursion lengths]
\label{lem:mom-exc}
Let $\rho\in(0,\frac12)$ and $J\ge 1$. There exist constants $0<C<\8$ and $C_1(\rho,J)$ such that for $\tau\in \good_n(\rho, J)$ and for all $a,b\in \{0,1\}$,
\beq
\label{lem:mom-exc1}
\tfrac 1 C \tone^3 \leq \cE[\xi^{(\ins,b)}_{a}] \le C \tone^3,
\qquad \cE[\xi^{(\out)}_{a}] \le C_1(\rho,J)
\eeq 
and
\beq
\label{lem:mom-exc2}
\cE[(\xi^{(\ins,b)}_{a})^2]  \le C \tone^6,
\qquad \cE[(\xi^{(\out)}_{a})^2] \le C_1(\rho,J).
\eeq 
\end{lemma}
\begin{remark}
The assumption that $\tau\in\good_n$ is actually not required for the moments of the excursion lengths inside the optimal gap.
\end{remark}

\begin{proof}[Proof of Lemma~\ref{lem:mom-exc}] 
(i) \textit{Proof of \eqref{lem:mom-exc1}.} Both cases being similar we only treat the case $a=0$. We first observe that 
\beq
\label{eq:deriv:observ}
\ba
(a)\quad&\cE[\xi^{(\ins,1)}_{0}]%
=\sum_{\ell\ge 1} \ell \frac{\mK_{\ins}(0,1,\ell)e^{\phi\ell}}{\hat \mK_{\ins}(0,1,\phi)}%
= \frac{(\hat q_{\tone}^1)'(\phi)}{(\hat q_{\tone}^1)(\phi)},\\
(b)\quad&\cE[\xi^{(\ins,0)}_{0}]%
=\sum_{\ell\ge 1} \ell \frac{\mK_{\ins}(0,0,\ell)e^{\phi\ell}}{\hat \mK_{\ins}(0,0,\phi)}%
= \frac{(\hat q_{\tone}^0)'(\phi)}{(\hat q_{\tone}^0)(\phi)},\\
(c)\quad&\cE[\xi^{(\out)}_{0}]%
=\sum_{\ell\ge 1} \ell \frac{\mK_{\out}(0,\ell)e^{\phi\ell}}{\hat \mK_{\out}(0,\phi)}.\\%
\ea
\eeq
\par To deal with (a), we use Lemma~\ref{lem:asympt-q} to obtain
\beq
\frac{(\hat q_{\tone}^1)'(\phi)}{(\hat q_{\tone}^1)(\phi)}\stackrel{\tone\to\infty}{\sim} \frac{2\tone^2}{\pi^2 \gep(\tau^-,\tone, \tau^+)}
\eeq 
and we conclude with Lemma \ref{lem:est-scd-order-FE}.

\par To deal with (b), we use the same equations to obtain
\beq
\frac{(\hat q_{\tone}^0)'(\phi)}{(\hat q_{\tone}^0)(\phi)}\stackrel{\tone\to\infty}{\sim}\frac{4\tone}{\pi^2 \gep(\tau^-,\tone, \tau^+)^2}\left(1 + \frac{2}{\tone\gep(\tau^-,\tone, \tau^+)}\right)^{-1},
\eeq
and conclude with Lemma \ref{lem:est-scd-order-FE} again. 

\par We finally deal with (c).
The denominator can be ignored as $\hat \mK_{\out}(1,\phi)\geq \gd_0$ by Lemma \ref{lem:binf.proba.out}. Then, we split the sum in the numerator according to the threshold $\ttwo^{\eta}$, where $\eta$ is any fixed real number such that $2< \eta< \frac{4(\gga+2)}{\gga+4}$, and first manage with the second part, which corresponds to $\ell\ge \ttwo^\eta$. We observe that for $\ell\geq 1$,
\beq
\mK_{\out}(1,\ell) \leq \bP(\bgs(\tau^+)  > \ell-1).
\eeq
By Proposition~\ref{prop:roughUB} we obtain
\beq
\label{eq:UB-Zout}
\mK_{\out}(1,\ell) \leq C \ell^3 \exp(-\phi(\ttwo) \ell),
\eeq
and
\beq
\label{eq:second bout}
\sum_{\ell \ge \ttwo^{\eta}} \ell \mK_{\out}(1,\ell)e^{\phi\ell} 
\le \sum_{\ell \ge \ttwo^{\eta}}  C \ell^4 \exp(-(\phi(\ttwo)-\phi) \ell).
\eeq
On $\kBthree_n(\rho)$, it holds that $\ttwo^2(\phi(\ttwo)-\phi)$ is bounded from below by ${\rm (Cst)\rho^3}$ and, as $\eta>2$, the last term in \eqref{eq:second bout} is bounded uniformly in $\ttwo$.

\par We turn to the first part of the sum in the numerator, which corresponds to $\ell \le \ttwo^\eta$. As $\tau \in \kBfour_n(J)$, one can use Lemma \ref{lem:Zout} so that
\beq
\sum_{J^{\frac{\gamma+2}{\gamma}} \leq \ell \le \ttwo^{\eta}} \ell \;\mK_{\out}(1,\ell)e^{\phi\ell}
\leq \sum_{J^{\frac{\gamma+2}{\gamma}} \leq \ell \le \ttwo^{\eta}} C\ell^4 e^{-\ell^{\frac{\gga}{2(\gga+2)}}}e^{\phi\ell}.
\eeq
On $\kBthree_n(\rho)$, we obtain from~\eqref{eq:def:gt} and Lemma~\ref{lem:sandw:phit1t2} that $\phi\leq c(\rho)/\ttwo^2$. This implies that for all $\ell \leq \ttwo^{\eta}$, $\phi \ell \leq c(\rho) \ell^{1-2/\eta}$ and, as $1-\frac 2 \eta< \frac{\gga}{2(\gga+2)}$,
\beq
\sum_{J^{\frac{\gamma+2}{\gamma}} \leq \ell \le \ttwo^{\eta}} C\ell^4 e^{-\ell^{\frac{\gga}{2(\gga+2)}}}e^{\phi\ell}
\leq \sum_{\ell \ge 1} C\ell^4 e^{-\ell^{\frac{\gga}{2(\gga+2)}}}e^{c(\rho) \ell^{1-2/\eta}} \leq C(\rho)<+\8.
\eeq
For the $J^{\frac{\gamma+2}{\gamma}}$ first terms of the sum, we only use the rough control :
\beq
\sum_{\ell \leq J^{\frac{\gamma+2}{\gamma}}}\cP(\xi_1^{(\out)}=\ell) \ell \leq \sum_{\ell \leq J^{\frac{\gamma+2}{\gamma}}} \ell \leq J^{2(\frac{\gamma+2}{\gamma})}. 
\eeq
We have thus proven that if $\tau\in \kBthree_n(\rho)\cap \kBfour_n(J)$, there exists a constant $C_1(\rho, J)$ so that 
\beq
\cE(\xi_1^{(\out)})\leq C_1(\rho,J).
\eeq 

(ii) \textit{Proof of \eqref{lem:mom-exc2}.} Let us now deal with the second moments. With the same observation as in~\eqref{eq:deriv:observ}, we get
\beq
\label{eq:deriv2:observ}
\ba
(d)\quad&\cE[(\xi^{(\ins,1)}_{0})^2]%
=\sum_{\ell\ge 1} \ell^2 \frac{\mK_{\ins}(0,1,\ell)e^{\phi\ell}}{\hat \mK_{\ins}(0,1,\phi)}%
= \frac{(\hat q_{\tone}^1)''(\phi)}{(\hat q_{\tone}^1)(\phi)},\\
(e)\quad&\cE[(\xi^{(\ins,0)}_{0})^2]%
=\sum_{\ell\ge 1} \ell^2 \frac{\mK_{\ins}(0,0,\ell)e^{\phi\ell}}{\hat \mK_{\ins}(0,0,\phi)}%
= \frac{(\hat q_{\tone}^0)''(\phi)}{(\hat q_{\tone}^0)(\phi)},\\
(f)\quad&\cE[(\xi^{(\out)}_{0})^2]%
=\sum_{\ell\ge 1} \ell^2 \frac{\mK_{\out}(0,\ell)e^{\phi\ell}}{\hat \mK_{\out}(0,\phi)}.%
\ea
\eeq
Items (d) and (e) are taken care of with Lemma~\ref{lem:asympt-q} together with Lemma \ref{lem:est-scd-order-FE}. Item (f) can be treated in exactly the same manner as (c).
\end{proof}

\subsection{Mass renewal function and survival probability}
\label{subsec:mrfsp}
Recall that the mass renewal function $(\cP(n\in \bar \gt))_{n\geq 1}$ is related to the pinned survival probability (Proposition~\ref{pr:ren-int}). In this section we provide bounds both for the pinned (Lemma \ref{lem:mass-ren-fct}) and unpinned (Lemma \ref{lem:sharp-Zn-renorm}) survival probabilites. Due to their technical nature and their similarities, the proofs of these two lemmas are deferred to Appendix~\ref{app:plms}.
\begin{lemma}\label{lem:mass-ren-fct}
Let $\rho\in(0,\frac12)$ and $J\ge 1$. There exists $C=C(\rho)$ such that for $\tau\in\good_n(\rho, J)$, $a\in\{0,1\}$, $\tone$ large enough, and $k\ge \tone^3$ (for the lower bound) or $k\ge 1$ (for the upper bound),
\beq
\frac{\ind_{\{k {\rm\ is\ even\ or\ }\tone {\rm\ is\ odd}\}}}{C \tone^3} \le \cP_a(k\in \bar\gt) \le \frac{C}{k^{3/2}\wedge \tone^3}.
\eeq
\end{lemma}
%%%
Recall that $\phi = \phi(\tau^-,\tone, \tau^+)$ is the free energy defined in Section~\ref{sec:defcP}. The following lemma is the unpinned analog of Lemma~\ref{lem:mass-ren-fct}.
\begin{lemma}\label{lem:sharp-Zn-renorm} Let $\rho\in(0,\frac12)$ and $J\ge 1$. There exists $C(\rho)>0$ such that for $\tau\in\good_n(\rho,J)$, $\tone$ large enough and $k\ge 2\tone^3$ (for the lower bound) or $k\ge 2 \tone^2$ (for the upper bound),
\beq
\frac{1}{C(\rho)\tone} \le \mZ_k e^{\phi k} \le \frac{C(\rho)}{\tone}.
\eeq
\end{lemma}

\begin{remark}\label{rmk:pin-vs-free}
The pinned and free partition functions lead to the same value of the free energy but have different prefactors. Indeed, when $k\ge 2\tone^3$, $\mZ_k e^{\phi k}$ is of order $1/\tone$ by Lemma~\ref{lem:sharp-Zn-renorm} whereas $\mZ_k^\pin e^{\phi k}$ is of order $1/\tone^3$ by Proposition~\ref{pr:ren-int}, Lemma~\ref{eq:control.ratio.h} and Lemma~\ref{lem:mass-ren-fct}. As a consequence, we obtain that the right endpoint of the free polymer (started at the boundary of the optimal gap) is pinned at the boundary of the optimal gap with probability roughly $1/\tone^2$, which differs from the $1/\tone^3$ found for the Markov renewal process $\cP$ in Lemma~\ref{lem:mass-ren-fct}. This comes from the fact that, unlike the pinned measure, the free measure is not comparable to the Markov renewal process $\cP$ close to the right endpoint of the polymer (the latter point shall be made more precise during the proof of Proposition~\ref{pr:locsup1}, see in particular \eqref{eq:def-Theta}). Moreover, we notice that the pinning probability under the free measure is of the same order ($1/\tone^2$) as the free energy $\phi$. Remarkably, a very similar situation occurs for the homogeneous pinning model (with a single interface) in the localized regime, see~\cite[Theorem 2.2]{Gia07}. One may naturally wonder how general this observation can be made.
\end{remark}

%%%%%%%%%%%%%%%%%%%%%%%%%%%%%%%%%%%%%%%%%%%
%%%%%%%%%%%%%%%%%%%%%%%%%%%%%%%%%%%%%%%%%%%
\section{Localization starting from the optimal gap : case $\gga<1$}
\label{sec:locinf1}
In this section we assume that $\gga<1$. In this case, the localization interval is so large that a stronger form of localization occurs, as it is stated in the next proposition.
\begin{proposition} 
\label{pr:locinf1}
Let $\petit \in (0,1)$. For $C>0$ large enough, for all $\rho\in (0,1/2), \gep_0\in (0,1), \gep\in (0,\gb/2)$ and $J\geq 1$, for $\tau\in\good_n(\gep_0,\gep,\rho,J)$ and $n$ large enough,
\beq
\label{eq:locinf1}
\bP_{\tau_{\lopt-1}}\Big(\forall k\ge C,\  S_k \in I_\loc \Big| \bgs(\mtau) > n\Big) \ge 1-\petit.
\eeq
\end{proposition}

\begin{proof}[Proof of Proposition~\ref{pr:locinf1}] 
\par The proof mainly follows ~\cite[Theorem 1.1 (3)]{CP09b} with adaptions due to the fact that our environment is not periodic. We are going to prove that for all $\petit\in (0,1)$, one can choose $C$ and $n$ large enough so that for all $\rho\in (0,1/2), \gep_0\in (0,1), \gep\in (0,\gb/2)$ and $J\geq 1$, on $\good_n(\gep_0,\gep,\rho,J)$,
\beq
\frac
{\mZ_{n}(S_{[C,n]}\cap I_\loc^c\neq\emptyset)}
{\mZ_{n}(S_{[C,n]}\cap I_\loc^c=\emptyset)}
\leq {\petit},
\eeq
which readily implies~\eqref{eq:locinf1}. We first control the denominator using \eqref{eq:estim-qtn2} in Lemma \ref{lem:estim-qtn}: 
\beq
\ba
D:=\mZ_{n}(S_{[C,n]}\cap I_\loc^c=\emptyset) &\geq \mZ_{n}(\bar \gt_1 > n, Y_1=in)\\
& \geq \frac{c_1}{2\tone}e^{-g(\tone)n}.
\ea
\eeq
We now turn to the numerator. We recall the definition of  $\cN(n)$ in \eqref{eq:defcN} and write
\beq
{\mZ_{n}(S_{[C,n]}\cap I_\loc^c\neq\emptyset)}\leq 
E_1+E_2+E_3,
\eeq
where
\beq
\ba
E_1&={\mZ_{n}(\bar \gt_{\cN(n)} \leq  C, S_{[C,n]}\subset I_\loc^c)},\\
E_2&={\mZ_{n}(C < \bar \gt_{\cN(n)} \leq n -2 \tone^2)},\\
E_3&={\mZ_{n}(\bar \gt_{\cN(n)} > n -2 \tone^2)}.
\ea
\eeq
%%%%
We consider each of these three terms separately. From Proposition \ref{prop:roughUB},
\beq
E_1\leq (\textrm{cst.})\ n^3 e^{-\phi(\ttwo)(n-C)}
\eeq
so that on $\kBone_n(\gep_0)\cap \kBthree_n(\rho)$ and for $n$ large enough we obtain that $E_1/D < \petit$.
%%%%
\par Decomposing over all possible values for $\bar \gt_{\cN(n)}$ in $E_2$ we obtain using~\eqref{eq:comp-Zpin},
\beq
\ba
E_2&\le {\rm (cst.)}\sum_{k=C+1}^{ n -2 \tone^2}\cP(k\in \bar \gt)e^{-\phi k}\mZ_{n-k}(\bar \gt_1>n-k)\\
&\leq {\rm (cst.)}\sum_{k=C+1}^{ n -2 \tone^2}\cP(k\in \bar \gt)e^{-\phi k}\ \frac{C(\rho)}{\tone}e^{-\phi (n-k)}\\
&\leq \frac{{\rm (cst.)}}{\tone}e^{-\phi(\tone) n} \left(\sum_{k=C+1}^{ n -2 \tone^2}\cP(k\in \bar \gt) \right),
\ea
\eeq
where we have used Lemma \ref{lem:sharp-Zn-renorm} to go from the first to the second line and Lemma \ref{lem:sandw:phit1t2} to go from the second one to the third one.
Using Lemma \ref{lem:mass-ren-fct}, we get 
\beq
\sum_{k=C+1}^{ n -2 \tone^2}\cP(k\in \bar \gt) \leq \sum_{k\geq C+1} \frac{\textrm{(\textrm{cst.})}}{k^{3/2}} +\textrm{(cst)}  \frac{n}{\tone^3}.
\eeq
so that, on $\kBone_n(\gep_0)$, using \eqref{eq:phi} and \eqref{eq:def:gt}, we obtain
\beq
\frac{E_2}{D}\leq {\rm (cst.)}\Big( \sum_{k\geq C+1} \frac{\textrm{1}}{k^{3/2}} +\frac{n}{\tone^3}\Big)e^{\frac{{\rm (cst.)}n}{\tone^3}}.
\eeq 
 One can then fix $C$ large enough so that for $n$ large enough $E_2/ D\leq \petit$. 
%%%%
\par Using the same decomposition for $E_3$ we obtain
 \beq
E_3 \leq {\rm (cst.)}\sum_{k=n -2 \tone^2}^{n }\cP(k\in \bar \gt)e^{-\phi k}\mZ_{n-k}(\bar \gt_1>n-k).
\eeq
We first observe that
\beq
\mZ_{n-k}(\bar \gt_1>n-k)\leq \bP_0(H_0 >n-k) \leq \frac{\textrm{(cst.)}}{\sqrt{n-k+1}}.
\eeq
Then, using Lemma \ref{lem:sandw:phit1t2} and \eqref{eq:encadr_gphi}, one notices that $2 \phi \tone^2 \leq \textrm{(cst)}$ and from Lemma \ref{lem:mass-ren-fct} (note that on $\kBone_n$, $\tone^3\leq (n-2\tone^2)^{3/2}$ for $n$ large enough) we obtain
 \beq
E_3 \leq \textrm{(cst.)} e^{-\phi(\tone) n} \frac{1}{\tone ^3} \sum_{k=n -2 \tone^2}^{n} \frac{1}{\sqrt{n-k+1}} \leq \textrm{(cst.)}e^{-\phi(\tone) n} \frac{1}{\tone ^2}.
\eeq
From this estimate we easily prove that on $\kBone_n(\gep_0)$ and for $n$ large enough $E_3/ D\leq \petit$.
%%%%
\end{proof}

%%%%%%%%%%%%%%%%%%%%%%%%%%%%%%%%%%%%%%%%%%%
%%%%%%%%%%%%%%%%%%%%%%%%%%%%%%%%%%%%%%%%%%%
\section{Localization starting from the optimal gap : case $\gga>1$}
\label{sec:locsup1}
In this section we turn to the case $\gga>1$ where the gaps are integrable and the localization interval is thus of a smaller order than in the previous section. There is however still localization in this single interval even if not in such a strong sense as in Proposition \ref{pr:locinf1}.
\begin{proposition}
\label{pr:locsup1}
For all $\petit, \gep_0, \in(0,1)$, $\rho \in (0,\frac12)$ and $J\ge 1$, there exists $C=C(\petit,\rho,J)>0$ such that for all $a,b\in\{0,1\}$, $\tau\in\good_n(\gep_0,\rho,J)$ and for $n$ large enough,
\beq
\label{eq:locsup1}
\bP_{\tau_{\lopt-1}}\Big(\frac{n}{CT_{\lopt}^3} \le \cN^{a,b}_{\ins}(n), \cN^a_{\out}(n), \cT_{\out}(n) \le C \frac{n}{T_{\lopt}^3}\Big| \bgs(\mtau) > n\Big) \ge 1-\petit.
\eeq
\end{proposition}
The proof of this result consists in two main steps. First, we consider the probability in~\eqref{eq:locsup1} with respect to $\cP$, see Section~\ref{sec:lumrn}. Secondly, in Section~\ref{sec:markovtorw}, we transfer this result to the probability $\bP_{\tau_{\lopt-1}}(\cdot | \bgs(\mtau) >n)$ by controlling the Radon-Nikodym derivative of the latter with respect to the former. 

\subsection{Localization under the Markov renewal process}
\label{sec:lumrn}
We first control the number of excursions under $\cP$.
\begin{lemma}
\label{lem:nbr-exc}
Let $\petit\in(0,1)$, $\rho \in(0,\frac12)$ and $J\ge 1$. There exists $C=C(\petit,\rho,J)>0$ such that for $\tone$ large enough, $\tau\in\good_n(\rho,J)$ and all $k\geq \tone^3$,
\beq
\label{eq:lem:nbr-exc}
\cP(\tfrac 1C \tfrac{k}{\tone^3} \le \cN(k) \le C\tfrac{k}{\tone^3}) \ge 1-\petit.
\eeq
\end{lemma}

\begin{proof}[Proof of Lemma~\ref{lem:nbr-exc}]
Let us start with the upper bound. Note that for all $M\geq 0$
\beq
\cP(\cN(k) < M) = \cP(\bar \gt_M > k).
\eeq
The basic idea is to use a second moment bound. However, as the increments $(\gD\gt_i)$ are not independent (as we deal with a {\it Markov} renewal process), the evaluation of the second moment of $\bar \gt_M$ might be cumbersome. To simplify, we proceed as follows (recall the definition of the $\xi$'s before Lemma \ref{lem:mom-exc}):
\beq
\ba
&\cP(\bar \gt_M > k) \le \cP\Big(\sum_{1\le i \le M} (\xi^{(\out)}_{0,i}+\xi^{(\ins,0)}_{0,i}+\xi^{(\ins,1)}_{0,i}+\xi^{(\out)}_{1,i}+\xi^{(\ins,1)}_{1,i}+\xi^{(\ins,0)}_{1,i}) > k\Big)\\
&\le \sum_{a\in\{0,1\}} \cP\Big(\sum_{1\le i \le M}\xi^{(\ins,0)}_{a,i} > k/6\Big)+\cP\Big(\sum_{1\le i \le M}\xi^{(\ins,1)}_{a,i} > k/6\Big)+\cP\Big(\sum_{1\le i \le M} \xi^{(\out)}_{a,i} > k/6\Big).
\ea
\eeq
We may now use the Markov inequality for sums of i.i.d. random variables for each of the six terms in the previous sum. As the six ones can be treated with the same technique we only focus on the first one, in the case $a=0$: 
%\beq
%\cP\Big(\sum_{1\le i \le M} \xi_{0,i}^{(\ins,0)} > k/6\Big)
%\le \frac{36 M\cE[(\xi_{0}^{(\ins,0)})^2]} {k^2\Big(1 - \frac{6M}{k}\cE[\xi_{0}^{(\ins,0)}]\Big)^2},
%\quad \textrm{provided that\ } \frac{6M}{k}\cE[\xi^{(\ins,0)}_{0}] < 1.
%\eeq

\beq
\cP\Big(\sum_{1\le i \le M} \xi_{0,i}^{(\ins,0)} > k/6\Big)
\le \frac{6 M\cE(\xi_{0}^{(\ins,0)})} {k}.
\eeq

We may now conclude by using Lemma~\ref{lem:mom-exc} and choosing $M = \lfloor \frac{k}{C_1 \tone^3} \rfloor$, with $C_1$ large enough.

\par We now turn to the lower bound in \eqref{eq:lem:nbr-exc}. Let $M\in\bbN$ and
\beq
\mathfrak{n}_{0,M}^{(\ins,0)} := \#\{1\le i \le M\colon X_{i-1}=X_i=0, Y_i = \ins\}
\eeq
be the number of inward excursions from the lower boundary of the optimal gap to itself, among the $M$ first ones. Note that for all $c_1>0$ and with the slight abuse of notation $c_1M = \lfloor c_1 M \rfloor$,
\beq
\label{eq:proof-lem:nbr-exc}
\cP(\cN(k) \ge M) = \cP(\bar \gt_M \le k)
\le \cP\Big(\sum_{i=1}^{ c_1M } \xi_{0,i}^{(\ins,0)} \le k\Big) + \cP(\mathfrak{n}_{0,M}^{(\ins,0)} \le c_1M).
\eeq
Using standard Markov chain theory and  Lemma~\ref{lem:binf.proba.out}, one can choose $c_1$ small enough, so that the second term goes to $0$ with $M$. We thus consider $M=C_1\frac{k}{\tone^3}$ with $C_1$ large enough so that 
the second term is smaller than $\petit$ uniformly in $k\geq \tone^3$. Let us now deal with the first term. By Lemma~\ref{lem:mom-exc}, $\cE(\xi_{0}^{(\ins,0)}) \ge C_0^{-1}\tone^3$ for some $C_0>0$. Provided that the constant $C_1$ chosen above is actually larger than $C_0/c_1$ so that $c_1M \times C_0^{-1} \tone^3> k$, we obtain, uniformly in $k\geq \tone^3$,
\beq
\ba
\cP\Big(\sum_{i=1}^{ c_1M } \xi_{0,i}^{(\ins,0)} \le k\Big)& \le
\cP\Big(\sum_{i=1}^{ c_1M } [\xi_{0,i}^{(\ins,0)} - \cE(\xi_{0}^{(\ins,0)})] \le k -  c_1M  C_0^{-1} \tone^3\Big)\\
&\le \frac{ c_1M }{(k -  c_1 M  C_0^{-1} \tone^3)^2} 
\var_\cP(\xi_{0}^{(\ins,0)}).
\ea
\eeq
Using Lemma~\ref{lem:mom-exc} one can prove that uniformly in $k\geq \tone^3$, this last term decays asymptotically (as $C_1$ is large) like $C_0^{2} / (c_1C)$ so that, choosing initially $C_1$ even larger in the definition of $M$, one can conclude that it is bounded by $\petit$.
\end{proof}

We now refine Lemma~\ref{lem:nbr-exc} by considering the starting points of the excursions. For all $a\in\{0,1\}$ and $k\in\bbN$, let 
\beq
\cN_{a}(k)= \cN_{\out}^{a}(k)+\cN_\ins^{a,0}(k)+\cN_\ins^{a,1}(k).
\eeq

\begin{lemma}
\label{lem:nbr-exc-a}
Let $\petit\in(0,1)$, $\rho \in(0,\frac12)$ and $J\ge 1$. There exists $C=C(\petit,\rho,J)>0$ such that for $\tone$ large enough, $a\in \{0,1\}$, $\tau\in\good_n(\rho,J)$ and all $k\geq \tone^3$,
\beq
\cP\Big(\frac{k}{C \tone^3}\le \cN_{a}(k) \le \frac{C k}{\tone^3}\Big) \ge 1-\petit
\eeq
\end{lemma}

\begin{proof}[Proof of Lemma \ref{lem:nbr-exc-a}]
We only focus on the lower bound in the event above as the upper bound is trivial using Lemma \ref{lem:nbr-exc} and the fact that $\cN_{a}(k)\leq \cN(k)$.
Recall the definition of $C_1=C_1(\petit,\rho, J)$ given by Lemma \ref{lem:nbr-exc} and the one of $\gd_0$ given in Lemma \ref{lem:binf.proba.out}. Pick $C_2>C_1/\gd_0$. For all $a\in\{0,1\}$, we have
\beq
\ba
\cP\left(\cN_{a}(k) \ge \frac{k}{C_2 \tone^3} \right) &\ge 
\cP\Big(\sum_{i \le \frac{k}{C_1 \tone^3}} \ind_{\{X_i = a\}} \ge \frac{k}{C_2 \tone^3} , \cN(k) \ge \frac{k}{C_1 \tone^3} \Big)\\
&\ge 
\cP\Big(\sum_{i \le   \frac{k}{C_1 \tone^3}} \ind_{\{X_i = a\}} \ge  \frac{k}{C_2 \tone^3} \Big) - \cP\Big(\cN(k) < \frac{k}{C_1\tone^3}\Big).
\ea
\eeq
Thanks to Lemma~\ref{lem:nbr-exc}, the second probability is smaller than $\petit$ (note that this is still the case if we choose $C_1$ even larger). We then notice that, thanks to Lemma~\ref{lem:binf.proba.out},
\beq
\cP\Big(\sum_{i \le \frac{k}{C_1 \tone^3}} \ind_{\{X_i = a\}} \ge \frac{k}{C_2 \tone^3}  \Big) \ge 
\cP\Big(\bin\Big( \frac{k}{C_1 \tone^3}, \gd_0\Big)\ge \frac{k}{C_2 \tone^3}\Big).
\eeq
The latter is greater than $1-\petit$, uniformly in $k\geq \tone^3$, if $C_2$ is large enough.
\end{proof}

We now distinguish between inward and outward excursions.
\begin{lemma}
\label{lem:nbr-exc-out}
For all $\petit\in(0,1)$, $\rho\in (0,\frac12)$ and $J\ge 1$, there exists $C = C(\petit,\rho,J)>0$ such that for $\tone$ large enough, $a,b\in \{0,1\}$, $\tau\in\good_n(\rho,J)$ and $k\geq \tone^3$,
\beq
\cP\Big(\frac{k}{C \tone^3}\le \cN^{a,b}_{\ins}(k), \cN^a_{\out}(k) \le \frac{C k}{\tone^3}\Big) \ge 1-\petit.
\eeq
\end{lemma}

\begin{proof}[Proof of Lemma~\ref{lem:nbr-exc-out}]
Since $\cN_{\out}^{a}(k)$ and the $\cN^{a,b}_{\ins}(k)$'s are smaller that $\cN(k)$ it is easy to manage with the upper bound part using Lemma \ref{lem:nbr-exc}.
We thus focus on the lower bound. Consider $C_2=C_2(\petit,\rho,J)$ as given by Lemma \ref{lem:nbr-exc-a}. Pick $C_3>C_2/\gd_0$ with $\gd_0$ as in Lemma \ref{lem:binf.proba.out}. For all $a\in\{0,1\}$, we have
\beq
\ba
\cP\Big(\cN_{\out}^{a}(k) \ge \frac{k}{C_3 \tone^3}\Big) \ge 
\cP\Big(\cN_{\out}^{a}(k) \ge \frac{k}{C_3 \tone^3} \Big| \cN_a(k) \ge \frac{k}{C_2 \tone^3}\Big)\cP\Big(\cN_a(k) \ge \frac{k}{C_2 \tone^3}\Big).
\ea
\eeq
Thanks to Lemma~\ref{lem:nbr-exc-a}, the second probability in the product is greater than $1-\petit$. By using the Markov property at the return times to $a$ and thanks to Lemma~\ref{lem:binf.proba.out}, we then notice that
\beq
\cP\Big(\cN_{\out}^{a}(k) \ge \frac{k}{C_3 \tone^3} \Big| \cN_a(k) \ge \frac{k}{C_2 \tone^3}\Big) \ge 
\cP\Big(\bin\Big(\tfrac{k}{C_2 \tone^3}, \gd_0\Big)\ge\frac{k}{C_3 \tone^3} \Big).
\eeq
The latter is greater than $1-\petit$ uniformly in $k\geq \tone^3$ if we choose $C_3$ large enough.
\end{proof}
We now focus on the cumulated length of the outward excursions. This quantity will later correspond to the time spent outside of the localization interval.
\begin{lemma}
\label{lem : tps exterieur}
For all $\petit\in(0,1)$, $\rho\in(0,\frac12)$ and $J\ge 1$, there exists $C = C(\petit,\rho,J)>0$ such that for $\tone$ large enough, $\tau\in\good_n(\rho,J)$ and $k\geq \tone^3$,

\beq
\cP\Big( \frac{k}{C\tone^3} \leq \cT_\out(k) \le C\frac{k}{\tone^3}\ \Big) \ge 1-\petit.
\eeq

\end{lemma}

\begin{proof}[Proof of Lemma~\ref{lem : tps exterieur}]
\par For the first line we simply observe that 
\beq
\cT_\out(k) \geq \cN_{\out}^{0}(k)+\cN_{\out}^{1}(k),
\eeq
and conclude using Lemma \ref{lem:nbr-exc-out}. For the second one we use that $\cE(\xi_a^{(\out)})\leq K$ for some constant $K$ by Lemma \ref{lem:mom-exc}. As both cases are similar we may suppose $a=1$. For all $C_4>0$
\beq
\cP\Big( \sum_{i\leq \cN_{\out}^{1}(k)}\xi^{(\out)}_{1,i}  \ge C_4\frac{k}{\tone^3} \Big)\leq 
\cP\Big( \cN_{\out}^{1}(k)\geq C_3(\petit,\rho,J) \frac{k}{\tone^3} \Big)+\cP\Big(  \sum_{i\leq C_3(\petit,\rho,J)\frac{k}{\tone^3}}\xi^{(\out)}_{1,i}  \ge C_4\frac{k}{\tone^3} \Big)
\eeq
where $C_3(\petit,\rho,J)$ is the constant provided by Lemma \ref{lem:nbr-exc-out}. The first term is directly controlled by this lemma. For the second one we simply use Markov inequality and then choose $C_4(\petit,\rho,J)$ large enough so that $C_3 K/C_4\leq \petit$. 
\end{proof}

\subsection{From the Markov renewal process to the random walk conditioned to survive}
\label{sec:markovtorw}
In this section we prove Proposition~\ref{pr:locsup1} by transfering the results from Lemma~\ref{lem:nbr-exc-out} and Lemma~\ref{lem : tps exterieur} to the polymer measure $\mZ_n(\cdot)/\mZ_n$, since \eqref{eq:locsup1} can be rewritten as
\beq
\mZ_n\Big(\frac{n}{C\tone^3} \le \cN^{a,b}_{\ins}(n), \cN^a_{\out}(n),\cT_{\out}(n) \le C\frac{n}{\tone^3}\Big) \ge (1-\petit) \mZ_n.
\eeq
We have already established the link between the measure $\cP$ and the {\it pinned} polymer measure $\mZ_n^\pin(\cdot)/\mZ_n^\pin$ in Proposition~\ref{pr:ren-int}, see also~\eqref{eq:comp-Zpin}. However, the pinned and unpinned polymer measures are not directly comparable close to the endpoint $n$. Therefore, we need to chop off a small amount of the random walk trajectory close to its right-boundary. This small amount shall be denoted by 
\beq
\chop:=2\tone^2.
\eeq
The idea is then to decompose the probability of the event under consideration (w.r.t. the polymer measure starting from the lower boudary of the optimal gap) according to the last point of $\bar \gt$ before $n-\chop$. Recall the definition of the free energy $\phi$ given in Section~\ref{sec:markov-ren-int} (see~\eqref{eq:defel}). For $\chop\le k \le n$, set:
\beq
\label{eq:def-Theta}
\Theta(n,k) :=
\frac{e^{-\phi(n-k)}\mZ_k(\bar \gt_1 > k-\chop)}%
{\cP(\bar \gt_1> k-\chop) \mZ_n},%
\eeq
which will (roughly speaking) play the role of the Radon-Nikodym derivative of the polymer measure w.r.t. the measure $\cP$. The key tool is the following lemma, the proof of which is deferred to Appendix \ref{sec:radon}.
\begin{lemma} Let $\gep_0\in(0,1)$, $\rho\in(0,\frac12)$ and $J\ge 1$. There exists $C(\rho)>0$ such that for all $\tau\in\good_n(\gep_0,\rho,J)$, and $\tone$ large enough,
\label{lem:ctrl_RN-deriv}
\beq
\sup_{\chop\le k \le n} \Theta(n,k) \le C(\rho).
\eeq
\end{lemma}
We may finally prove the main result of this section:
\begin{proof}[Proof of Proposition~\ref{pr:locsup1}] (i) Let us first prove that
\beq
\mZ_n\Big(\cN(n) \le \frac{n}{C\tone^3}\Big) \le \petit \mZ_n,
\eeq
which readily gives us the upper bound for the $\cN^{a,b}_\ins(n)$'s and $\cN^a_\out(n)$'s.
First, we decompose the partition function according to the last point in $\bar \gt$ before $n-\chop$. As $\cN(n)$ is non-decreasing in $n$, it follows that
\beq
\mZ_n\Big(\cN(n) \le \frac{n}{C\tone^3}\Big) \le
\sum_{\chop\le k\le n} \mZ_{n-k}\Big(\cN(n-k)\le \frac{n}{C\tone^3}, n-k\in\bar\gt\Big) \mZ_k(\bar\gt_1 >k-\chop).
\eeq
Recall~\eqref{eq:comp-Zpin} and~\eqref{eq:def-Theta}. We get
\beq
\mZ_n\Big(\cN(n) \le \frac{n}{C\tone^3}\Big) \le {\rm (cst.)} \mZ_n \sum_{\chop\le k\le n} \cP\Big(\cN(n-k)\le \frac{n}{C\tone^3}, n-k\in\bar\gt\Big) \cP(\bar\gt_1>k-\chop) \Theta(n,k).
\eeq
By Lemma~\ref{lem:ctrl_RN-deriv}, there exists $C(\rho)>0$ such that
\beq
\ba
\mZ_n\Big(\cN(n) \le \frac{n}{C\tone^3}\Big) &\le C(\rho) \mZ_n \sum_{\chop\le k\le n} \cP\Big(\cN(n-k)\le \frac{n}{C\tone^3}, n-k\in\bar\gt\Big) \cP(\bar\gt_1>k-\chop)\\
&\le C(\rho) \mZ_n \cP\Big( \cN(n-\chop) \le \frac{n}{C\tone^3}\Big).
\ea
\eeq
Lemma~\ref{lem:nbr-exc} allows us to conclude, since $\chop$ is negligible in front of $n$ on $\kBone_n$.
\par (ii) Let us now deal with the lower bound. Let $a,b\in\{0,1\}$. Let us write, with obvious notation,
\beq
\cN^{a,b}_\ins(n) = \cN^{a,b}_\ins(n-\chop) +\cN^{a,b}_\ins(n-\chop,n).
\eeq
Then,
\beq
\label{eq:from-cP-to-Zn-ii}
\mZ_n\Big(\cN^{a,b}_\ins(n) \ge C\frac{n}{\tone^3}\Big) \le \mZ_n\Big(\cN^{a,b}_\ins(n-\chop) \ge \frac{C(n-\chop)}{2\tone^3}\Big)
+ \mZ_n\Big(\cN^{a,b}_\ins(n-\chop,n) \ge \frac{Cn}{2\tone^3}\Big).
\eeq
The first term in the sum can be handled with the same arguments as in (i) using  Lemma~\ref{lem:ctrl_RN-deriv} again and Lemma~\ref{lem:nbr-exc-out} instead of Lemma~\ref{lem:nbr-exc}. From now on, we focus on the second term. By decomposing on the last and first renewal points of $\bar\gt$ before and after $n-\chop$, we get
\beq
\mZ_n\Big(\cN^{a,b}_\ins(n-\chop,n) \ge \frac{Cn}{2\tone^3}\Big)= \sum_{u,v\in\{0,1\}}
\sumtwo{\chop\le k\le n}{0\le \ell\le \chop} \mZ_{n-k}^{0u}(n-k\in\bar\gt)\mZ_{k-\ell}^{uv}(\bar\gt_1 = k-\ell) 
\mZ_\ell^v\Big(\cN^{a,b}_\ins(\ell)\ge \frac{Cn}{2\tone^3}\Big).
\eeq
Clearly,
\beq
\mZ_\ell^v\Big(\cN^{a,b}_\ins(\ell)\ge \frac{Cn}{2\tone^3}\Big) \le \exp\Big(-\gb \frac{Cn}{2\tone^3}\Big).
\eeq
Moreover, since $\ell \le \chop = 2\tone^2$, there exists $c_0>0$ such that (recall Lemma~\ref{lem:estim-qtn})
\beq
\label{eq:LBZellc0}
\mZ_{\ell}^v \ge \tfrac12 q_{\tone}(\ell) \ge \frac{c_0}{\tone^3}.
\eeq
As a consequence,
\beq
\ba
&\mZ_{n}\Big(\cN^{a,b}_\ins(n-\chop,n) \ge \frac{Cn}{2\tone^3}\Big)\\ 
&\qquad \le\sum_{u,v\in\{0,1\}}
\sumtwo{\chop\le k\le n}{0\le \ell\le \chop} \mZ_{n-k}^{0u}(n-k\in\bar\gt)\mZ_{k-\ell}^{uv}(\bar\gt_1 = k-\ell) 
\mZ_\ell^v\frac{\tone^3}{c_0} \exp\Big(-\gb \frac{Cn}{2\tone^3}\Big)\\
&\qquad \le  \frac{\tone^3}{c_0}\exp\Big(-\gb \frac{Cn}{2\tone^3}\Big) \mZ_n,
\ea
\eeq
which concludes this step, since $n\ge (\gep_0 \tone)^{\gga+2}$ on $\good_n(\gep_0)$ and $\gga>1$. The proof works exactly the same way for $\cN^a_\out(n)$.\\
\par (iii) Let us finally deal with $\cT_\out(n)$. The upper bound follows by writing 
\beq
\mZ_n\Big(\cT_\out(n) \le C\frac n {\tone^3}\Big) \le 
\mZ_n\Big(\cT_\out(n-\chop) \le C\frac n {\tone^3}\Big)
\eeq
and using the same proof strategy as in (i) using Lemma~\ref{lem:ctrl_RN-deriv} again and Lemma~\ref{lem : tps exterieur} instead of Lemma~\ref{lem:nbr-exc}. Let us now focus on the lower bound. Analogously to~\eqref{eq:from-cP-to-Zn-ii} and with obvious notation, we write
\beq
\mZ_n\Big(\cT_\out(n) \ge \frac{n}{C\tone^3}\Big) \le 
\mZ_n\Big(\cT_\out(n-\chop) \ge \frac{n-\chop}{2C\tone^3}\Big) + 
\mZ_n\Big(\cT_\out(n-\chop,n) \ge \frac{n}{2C\tone^3}\Big).
\eeq
We only need to focus on the last term, which we decompose into two parts:
\beq
{\rm (I)} = \sum_{u,v\in\{0,1\}}
\sumtwo{\chop\le k\le n}{0\le \ell\le \chop} \mZ_{n-k}^{0u}(n-k\in\bar\gt)\mZ_{k-\ell}^{uv}(\bar\gt_1 = k-\ell, Y_1=\ins)\mZ_{\ell}^v\Big(\cT_\out(\ell) \ge \frac{n}{2C\tone^3}\Big)
\eeq
and
\beq
\label{eq:UPZell(II)}
{\rm (II)} = \sum_{u,v\in\{0,1\}}
\sumtwo{\chop\le k\le n}{0\le \ell\le \chop} \mZ_{n-k}^{0u}(n-k\in\bar\gt)\mZ_{k-\ell}^{uv}(\bar\gt_1 = k-\ell, Y_1=\out)\mZ_{\ell}^v\Big(\cT_\out(\ell) \ge \frac{n}{2C\tone^3}-(\chop-\ell)\Big).
\eeq
Let us start with (I). By considering the events $\{\cN(\ell)>\ga\}$ and $\{\cN(\ell)\le\ga\}$ ($\ga$ is a positive integer to be determined later), we obtain
\beq
\mZ_{\ell}^v(\cT_\out(\ell) \ge \frac{n}{2C\tone^3}) \le e^{-\gb \ga}
+ \mZ_{\ell}^v\Big(\cT_\out(\ell) \ge \frac{n}{2C\tone^3}, \cN(\ell)\le \ga\Big).
\eeq
In turn, a rough union bound yields
\beq
\mZ_{\ell}^v\Big(\cT_\out(\ell) \ge \frac{n}{2C\tone^3}, \cN(\ell)\le \ga\Big) \le 
\sum_{1\le k \le \ga} \sum_{1\le i \le k} \mZ_\ell^v\Big(\bar\gt_i-\bar\gt_{i-1} \ge \frac{n}{2\ga C\tone^3}, Y_i = \out, \cN(\ell)=k\Big).
\eeq
We may now decompose the partition function inside the sum into a product, similarly to~\eqref{eq:decomposeZn}, and replace the long outward excursion (corresponding to index $i$) with an inward excursion, using~\eqref{eq:LBZellc0}. The reader may check that we get in this way:
\beq
\mZ_{\ell}^v\Big(\cT_\out(\ell) \ge \frac{n}{2C\tone^3}, \cN(\ell)\le \ga\Big) \le \frac{\ga^2 \tone^3}{c_0}  \left(\sum_{\frac{n}{2\ga C\tone^3} \le j\le \chop} \mK_{\out}(v,j)\right) \mZ_\ell^v.
\eeq
Since $\gga >1$, we may set $\ga = \tone^\gep$, with $0<\gep < \gga-1$.
Recall that $n\ge (\gep_0 \tone)^{\gga+2}$ on $\good_n(\gep_0)$, so that $\frac{n}{2\ga C\tone^3}\ge J^{\frac{\gga+2}{\gga}}$, for $\tone$ large enough. We can then use Lemma~\ref{lem:Zout} and get
\beq
\sum_{\frac{n}{2\ga C\tone^3} \le j\le \chop} \mK_{\out}(v,j) \le 
\tone^5 \exp\Big(- \Big[\frac{n}{2\ga C\tone^3}\Big]^{\frac{\gga}{2(\gga+2)}}\Big).
\eeq
As $\frac{n}{\ga \tone^3}$ is of order $\tone^{\gga-1-\gep}$, the bound on (I) readily follows. 
\par Let us now deal with (II). If $\chop-\ell \le \frac{n}{4C\tone^3}$ then we proceed as for (I), since then $\cT_\out(\ell) \ge \frac{n}{4C\tone^3}$ in the last term of~\eqref{eq:UPZell(II)}. Else, $k-\ell\ge \chop-\ell \ge \frac{n}{4C\tone^3}$ and we argue that for $n$ large enough,
\beq
\mZ_{k-\ell}^{uv}(\bar\gt_1 = k-\ell,Y_1=\out)\le \petit \mZ_{k-\ell}^{uv}(\bar\gt_1 = k-\ell, Y_1=\ins) 
\le \petit \mZ_{k-\ell}^{uv}(\bar\gt_1 = k-\ell).
\eeq
Indeed,
\beq
\mZ_{k-\ell}^{uv}(\bar\gt_1 = k-\ell,Y_1=\out) \le {\rm (cst)} (k-\ell)^3 \exp(-\phi(\ttwo)(k-\ell)),
\eeq
while, from Lemma~\ref{lem:estim-qtn},
\beq
\mZ_{k-\ell}^{uv}(\bar\gt_1 = k-\ell,Y_1=\ins) \ge {\rm (cst)} (k-\ell)^{-3/2} \exp(-g(\tone)(k-\ell)).
\eeq
Use~\eqref{eq:def:gt} and~\eqref{eq:phi} with the fact that $\tau\in \kBthree_n(\rho)$ (i.e. $\rho\tone < \ttwo < (1-\rho) \tone$) to conclude.
\end{proof}
%%%%%%%%%%%%%%%%%%%%%%%%%%%%%%%%%%%%%%%%%%%
%%%%%%%%%%%%%%%%%%%%%%%%%%%%%%%%%%%%%%%%%%%
\section{Proof of Theorems~\ref{thm:ggasup1} and~\ref{thm:ggainf1}}
\label{sec:proofmain}
We finally reassemble all our results to prove Theorems~\ref{thm:ggasup1} and~\ref{thm:ggainf1}. We mainly focus on the proof of the former and only specify what has to be adapted for the latter.

\begin{proof}[Proof of Theorem~\ref{thm:ggasup1}] 

\par \textit{Step 1.} We first prove that Propositions \ref{pr:weakloc} and \ref{pr:hitting} are still true if one replaces $\bgs$ (that is $\bgs(\tau)$) by $\bgs(\mtau)$ in the conclusion. We observe that by a straightforward pathwise comparison $\bP(\bgs>n)\geq \bP(\bgs(\mtau)>n)$ so that
\beq
\bP(H^*_{\kopt}\leq n <H^*_{\kopt+1} | \bgs(\mtau) >n) \geq \frac{\bP(H^*_{\kopt}\leq n <H^*_{\kopt+1},\bgs(\mtau)>n)}{\bP(\bgs>n)}.
\eeq
Note that on the event $\{n <H^*_{\kopt+1}\}$, one can switch $\tau$ and $\mtau$ so that
\beq
\bP(H^*_{\kopt}\leq n <H^*_{\kopt+1},\bgs(\mtau)>n)=\bP(H^*_{\kopt}\leq n <H^*_{\kopt+1},\bgs>n)
\eeq
and finally
\beq
\label{eq:otau}
\bP(H^*_{\kopt}\leq n <H^*_{\kopt+1} | \bgs(\mtau) >n)\geq \bP(H^*_{\kopt}\leq n <H^*_{\kopt+1} | \bgs >n).
\eeq

\par Let us now deal with Proposition \ref{pr:hitting} and first observe that from Proposition~\ref{pr:weakloc}
\beq
\label{eq:compareotau}
\bP(\bgs>n) \leq \frac{1}{1-\petit}\bP(H^*_{\kopt+1}\wedge \bgs>n)=\frac{1}{1-\petit}\bP(H^*_{\kopt+1}\wedge \bgs(\mtau)>n)\leq \frac{1}{1-\petit}\bP(\bgs(\mtau)>n)
\eeq
so that \eqref{eq:decompohit} rewrites 
\beq
\ba
&\bP(H_{\tau_{\lopt-1}} > \kappa n|\bgs(\mtau) > n)\\
&\quad\leq 
\frac{1}{1-\petit}\bP(\kappa n < H^*_{k_0}\leq n < H^*_{k_0+1}  |\bgs   > n)+\bP(n \notin [H^*_{\kopt},H^*_{\kopt+1}) |\bgs(\mtau)> n),
\ea
\eeq
and we conclude easily using \eqref{eq:otau} and the end of the proof of Proposition \ref{pr:hitting}.

\par  \textit{Step 2.} Let us now prove Theorem~\ref{thm:ggasup1} for $\bP(\cdot|\bgs(\mtau)>n)$ instead of $\bP(\cdot|\bgs>n)$.
We start by picking the parameters in the following order:
\begin{enumerate}
\item $\rho\in(0,\frac12)$ small enough and $J\ge 1$ large enough to satisfy the assumptions of Proposition~\ref{pr:goodenv};
\item $\gd\in(0,1)$ and $\eta\in(0,1)$ small enough to satisfy the assumptions of Propositions~\ref{pr:goodenv} and~\ref{pr:hitting};
\item $\gep_0\in(0,1)$ small enough (depending on $\gd$) to satisfy the assumptions of Propositions~\ref{pr:goodenv} and~\ref{pr:hitting};
\item $\gep\in(0,\frac\gb 2)$ small enough (depending on $\gep_0$ and $\eta$) to satisfy the assumptions of Proposition~\ref{pr:hitting};
\item $\mathsf{C}\ge 1$ large enough (depending on $\gep_0$ and $\gep$) to satisfy the assumptions of Proposition~\ref{pr:goodenv}.
\end{enumerate}
Let us now assume that $\tau \in \good_n(\gd, \gep_0, \gep, \eta, \rho, J, \mathsf{C})$, which holds with $\bbP$-probability larger than $1-\petit$, so that $\tau$ satisfies the assumptions of Propositions~\ref{pr:hitting} and~\ref{pr:locsup1}. Let us abbreviate
\beq
{\rm LOC}(n,C) = \Big\{\frac{n}{CT_{\lopt}^3} \le \cN^{a,b}_{\ins}(n), \cN^a_{\out}(n), \cT_{\out}(n) \le C \frac{n}{T_{\lopt}^3}\Big\}.
\eeq
Then for all $C,\kappa>0$  ,
\beq
\ba
&\bP(H_{\tau_{\lopt-1}}\le \gk n, {\rm LOC}(n,C)|\bgs(\mtau)>n)\\
&\qquad= \sum_{1\le k \le \gk n}
\frac{
\bP(H_{\tau_{\lopt-1}} = k < \bgs(\mtau))
\bP_{\tau_{\lopt-1}}({\rm LOC}(n-k,C), \bgs(\mtau) > n-k)
}
{\bP(\bgs(\mtau) >n)}.
\ea
\eeq
Since $n-k \ge (1-\gk)n$, Proposition~\ref{pr:locsup1} guarantees that there exists $C>0$ such that for $n$ large enough,
\beq
\ba
&\bP(H_{\tau_{\lopt-1}}\le \gk n, {\rm LOC}(n,C)|\bgs(\mtau)>n)\\  &\qquad\ge 
(1-\petit)\sum_{1\le k \le \gk n}
\frac{
\bP(H_{\tau_{\lopt-1}} = k < \bgs(\mtau))
\bP_{\tau_{\lopt-1}}(\bgs(\mtau) > n-k)
}
{\bP(\bgs(\mtau) >n)}\\
&\qquad = (1-\petit) \bP(H_{\tau_{\lopt-1}}\le \gk n | \bgs(\mtau)>n).
\ea
\eeq
By Proposition~\ref{pr:hitting}, we finally obtain that
\beq
\bP(H_{\tau_{\lopt-1}}\le \gk n, {\rm LOC}(n,C)|\bgs(\mtau)>n) \ge (1-\petit)^2.
\eeq

\par \textit{Step 3.} We finally come back to $\bP(\cdot | \bgs>n)$.
Let us abbreviate 
\beq
A:=\{H_{\tau_{\lopt-1}}\le \gk n\} \cap {\rm LOC}(n,C).
\eeq 
From~\eqref{eq:compareotau},
\beq
\ba
\bP(A|\bgs>n)\geq \frac{\bP(A,H^*_{\kopt+1}>n,\bgs>n)}{\bP(\bgs>n)}\geq (1-\petit) \frac{\bP(A,H^*_{\kopt+1}>n,\bgs(\mtau)>n)}{\bP(\bgs(\mtau)>n)}.
\ea
\eeq
Moreover
\beq
\bP(A,H^*_{\kopt+1}>n|\bgs(\mtau)>n)\geq \bP(A|\bgs(\mtau)>n)-\bP(H^*_{\kopt+1}\leq n|\bgs(\mtau)>n)
\eeq
and one concludes using Steps $1$ and $2$.
\end{proof}

\begin{proof}[Proof of Theorem~\ref{thm:ggainf1}] The proof strategy is the same as for Theorem~\ref{thm:ggasup1} except that we use Proposition~\ref{pr:locinf1} (that brings no additional constraint on the choice of parameters) instead of Proposition~\ref{pr:locsup1}.
\end{proof}

%%%%%%%%%%%%%%%%%%%%%%%%%%%%%%%%%%%%%%%%%%%
%%%%%%%%%%%%%%%%%%%%%%%%%%%%%%%%%%%%%%%%%%%

%%%%%%%%%%%%%%%%%%%%%%%%%%%%%%%%%%%%%%%%%%%%%%
%% Single Appendix:                         %%
%%%%%%%%%%%%%%%%%%%%%%%%%%%%%%%%%%%%%%%%%%%%%%
%\begin{appendix}
%\section*{???}%% if no title is needed, leave empty \section*{}.
%\end{appendix}
%%%%%%%%%%%%%%%%%%%%%%%%%%%%%%%%%%%%%%%%%%%%%%
%% Multiple Appendixes:                     %%
%%%%%%%%%%%%%%%%%%%%%%%%%%%%%%%%%%%%%%%%%%%%%%
\begin{appendix}

\section{Results about ruin probabilities}
\label{app:ruin}
\subsection{Proof of Lemma~\ref{lem:estim-qtn}}
\begin{proof}[Proof of Lemma~\ref{lem:estim-qtn}]
When $n\in2\bbN$, the bounds in~\eqref{eq:estim-qtn1} and~\eqref{eq:estim-qtn2} are proven in~\cite[Lemma 2.1]{CP09b}. It is assumed therein that $t\in 2\bbN$. However, their proof relies on explicit formulas (see below) that are also valid for $t\in2\bbN-1$. We now focus on the case where $n\in2\bbN-1$. First, it suffices to combine \eqref{eq:estim-qtn1} with the trivial bounds
\beq
\sum_{i>n+1}  q_t(i)\le \sum_{i>n} q_t(i) \le \sum_{i>n-1} q_t(i)
\eeq
to obtain~\eqref{eq:estim-qtn2}. Let us now deal with~\eqref{eq:estim-qtn1} in the case where $n,t\in 2\bbN-1$. We remind the reader that
\beq
\ba
q^0_t(n) &=\Big( \frac 2 t \sum_{\nu=1}^{\lfloor(t-1)/2\rfloor} \cos^{n-2}\Big(\frac{\pi \nu}{t}\Big) \sin^{2}\Big(\frac{\pi \nu}{t}\Big) \Big) \ind_{\{n \in 2\bbN\}}, \\
q^1_t(n) &=\Big(\frac 1 t \sum_{\nu=1}^{\lfloor(t-1)/2\rfloor} (-1)^{\nu+1}\cos^{n-2}\Big(\frac{\pi \nu}{t}\Big) \sin^{2}\Big(\frac{\pi \nu}{t}\Big) \Big) \ind_{\{n-t \in 2\bbN_0\}},
\ea
\eeq
see equation (5.8) in~\cite[Chapter XIV]{Feller-vol1} or equation (B.1) in~\cite{CP09b}. To obtain the upper bound, we may replace the indicator $\ind_{\{n \in 2\bbN\}}$ by $1$ (even though $n$ is odd) and get
\beq
q_t(n) \le \frac4 t \sum_{\nu=1}^{\lfloor(t+2)/4\rfloor} \cos^{n-2}\Big(\frac{(2\nu-1)\pi}{t}\Big) \sin^{2}\Big(\frac{(2\nu-1)\pi}{t}\Big),
\eeq
from where the analysis is as in~\cite{CP09b}, see equation (B.2) therein. As for the lower bound, we start by fixing some $\gep\in(0,\frac12)$ and write:
\beq
q_t(n) = 2 q_t^1(n) \ge U_0(n) - U_1(n) - U_2(n),
\eeq
where
\beq
\ba
U_0(n) &= \frac{2}{t} \cos^{n-2}\Big(\frac{\pi}{t}\Big) \sin^{2}\Big(\frac{\pi}{t}\Big),\\
U_1(n) &=\frac 2 t \sum_{\nu=2}^{\lfloor \gep t\rfloor} \cos^{n-2}\Big(\frac{\pi \nu}{t}\Big) \sin^{2}\Big(\frac{\pi \nu}{t}\Big),\\
U_2(n) &=\frac 2 t \sum_{\nu=\lfloor \gep t\rfloor +1}^{\lfloor(t-1)/2\rfloor} \cos^{n-2}\Big(\frac{\pi \nu}{t}\Big) \sin^{2}\Big(\frac{\pi \nu}{t}\Big).
\ea
\eeq
In a similar way as (B.3)--(B.7) in~\cite{CP09b}, one may check that
\beq
U_0(n) = \frac{2\pi^2}{t^3}(1+o(1)) e^{-g(t)n},
\eeq
while
\beq
U_1(n) \le \frac{\rm (cst)}{n^{3/2}}e^{-g(t)n},
\qquad
U_2(n) \le {\rm (cst)} \cos^n(\pi\gep).
\eeq
The desired lower bound follows by choosing $\cT_0$ large enough so that $e^{-g(\cT_0)}> \cos(\pi \gep)$ and $n\ge c_5t^2$, with $c_5$ large enough.
\end{proof}

\subsection{Generating functions of ruin probabilities}
\label{sec:mgf-qt}
Let us now collect a few facts about the moment generating functions of the ruin probabilities defined in~\eqref{eq:defRuinproba2}.
For any $0 <f < g(t)$, set
\beq
\gD = \gD(f) = \arctan \sqrt{e^{2f} - 1}.
\eeq
Then~\cite[Appendix A]{CP09b},
\beq
\ba
\label{eq:mgf:qt}
\hat q_t(f) &= 1 + \tan(\gD) 	\frac{1- \cos(t\gD)}{\sin(t\gD)},\\
\hat q^0_t(f) &= 1 - \frac{\tan (\gD)}{\tan(t \gD)},\\
\hat q^1_t(f) &= \frac12 \frac{\tan(\gD)}{\sin(t\gD)}.
\ea
\eeq
In the limiting case $t=\infty$, the generating function is defined for all $f\leq 0$ and
\beq
\label{eq:mgf:qinfty}
\hat q^0_{\infty}(f) = 1 - \sqrt{1- e^{2f}}.
\eeq
Let us come back to the case $t<+\8$ and denote by $\tilde q_t$ the function such that $\hat q_t(f) = \tilde q_t(\gD(f))$. Then, see~\cite{CP09b},
\beq
\label{eq:mgf:qt.deriv}
\ba
(\tilde q_t)'(\gD) &= (\tilde q^0_t)'(\gD) + 2(\tilde q^1_t)'(\gD) = \frac{1-\cos(t\gD)}{\sin(t\gD)}
\Big[%
\frac{1}{\cos^2\gD} +%
t \frac{\tan \gD}{\sin(t\gD)}
\Big],\\%
\textrm{with}\quad (\tilde q^1_t)'(\gD) &= \frac{1}{2 \sin(t\gD)}
\Big[%
\frac{1}{\cos^2(\gD)} - \frac{t \tan \gD}{\tan(t\gD)}
\Big],\\
(\tilde q^0_t)'(\gD) &= \frac{1}{\sin(t\gD)}
\Big[%
-\frac{\cos(t\gD)}{\cos^2(\gD)} + \frac{t \tan \gD}{\sin(t\gD)}
\Big],\\
\textrm{and}\quad (\tilde q_t)''(\gD) &=\frac{1-\cos(t\gD)}{\sin(t\gD)}
\Big[%
\frac{2\sin\gD}{\cos^3(\gD)} +%
 \frac{2t}{\sin(t\gD) \cos^2\gD} +%
 t^2(1-\cos(t\gD)) \frac{\tan \gD}{\sin^2(t\gD)}
\Big],\\%
&= (\tilde q^0_t)''(\gD)+2(\tilde q^1_t)''(\gD),\\
\textrm{with}\quad (\tilde q^1_t)''(\gD) &= \frac{1}{2 \sin(t\gD)}
\Big[%
-\frac{2t \cos(t\gD)}{\sin(t\gD)\cos^2(\gD)} + \frac{2 \sin(\gD)}{\cos^3(\gD)}
+ \frac{t^2 (1+\cos^2(t\gD)) \tan(\gD)}{\sin^2(t\gD)}
\Big],\\
(\tilde q^0_t)''(\gD) &= \frac{1}{\sin(t\gD)}
\Big[%
\frac{2t}{\sin(t\gD)\cos^2(\gD)} - \frac{2\sin(\gD)}{\cos^3(\gD)}\cos(t\gD)
- \frac{2 t^2\cos(t\gD) \tan(\gD) }{\sin^2(t\gD)} 
\Big],
\ea
\eeq
and
\beq
\label{eq:delta:deriv}
\ba
\gD'(f) &= \frac{1}{\sqrt{e^{2f}- 1}} = \frac{1}{\tan \gD},\\
\gD''(f) &= - \frac{e^{2f}}{(e^{2f}- 1)^{3/2}} = -\frac{1+\tan^2 \gD}{\tan^3 \gD} =-\frac{\cos(\gD)}{\sin^3(\gD)}.
\ea
\eeq
We also provide the following large-$t$ asymptotic results:
\begin{lemma}
\label{lem:asympt-q}
For all {positive} functions $t\to \gep (t)$ converging to $0$ as $t\to +\8$ and such that $1/t^2 = o(\gep(t))$.

\beq
\label{eq:asymp:qhat2}
\ba
\hat q_t\Big(\frac{\pi^2}{2t^2}(1-\gep(t))\Big) - 1 \sim \frac{4}{t \gep(t)},\\
\hat q^0_t\Big(\frac{\pi^2}{2t^2}(1-\gep(t))\Big)-1 \sim \frac{2}{t \gep(t)},\\
\hat q^1_t\Big(\frac{\pi^2}{2t^2}(1-\gep(t))\Big) \sim \frac{1}{t \gep(t)},
\ea
\eeq
\beq
\label{eq:asymp:qhat.deriv}
\ba
(\hat q_t)'\Big(\frac{\pi^2}{2t^2}(1-\gep(t))\Big) &\sim \frac{8t}{\pi^2 \gep(t)^2}\\
(\hat q^0_t)'\Big(\frac{\pi^2}{2t^2}(1-\gep(t))\Big) &\sim \frac{4t}{\pi^2 \gep(t)^2}\\
(\hat q^1_t)'\Big(\frac{\pi^2}{2t^2}(1-\gep(t))\Big) &\sim \frac{2t}{\pi^2 \gep(t)^2},
\ea
\eeq
and
\beq
\label{eq:asymp:qhat.derivsec}
\ba
(\hat q_t)''\Big(\frac{\pi^2}{2t^2}(1-\gep(t))\Big) &\sim \frac{32t^3}{\pi^4 \gep(t)^3},\\
(\hat q^0_t)''\Big(\frac{\pi^2}{2t^2}(1-\gep(t))\Big) &\sim \frac{16t^3}{\pi^4 \gep(t)^3},\\
(\hat q^1_t)''\Big(\frac{\pi^2}{2t^2}(1-\gep(t))\Big) &\sim \frac{8t^3}{\pi^4 \gep(t)^3}.
\ea
\eeq
\end{lemma}

\begin{proof}[Proof of Lemma~\ref{lem:asympt-q}]
Let us pick $f(t)= \frac{\pi^2}{2t^2}(1-\gep(t))$ and abbreviate $\gD = \gD(f(t))$. Then,
\beq
\tan \gD = \frac{\pi}{t}\sqrt{1-\gep(t)}[1+O(1/t^2)],
\eeq
as $t\to\infty$, and, with the assumption that $1/t^2 = o(\gep(t))$, we get
\beq
t\gD = \pi [1 - \tfrac12 \gep(t)(1+o(1))].
\eeq
The asymptotics in~\eqref{eq:asymp:qhat2} are obtained from~\eqref{eq:mgf:qt}. Then, a straightforward computation gives 
\beq
\ba
&(\tilde q_t^1)'(\gD) \sim \frac{2}{\pi \gep(t)^2}, \qquad 
(\tilde q_t^0)'(\gD) \sim \frac{4}{\pi \gep(t)^2},\\
&(\tilde q_t^1)''(\gD) \sim \frac{8t}{\pi^2 \gep(t)^3}, \qquad 
(\tilde q_t^0)''(\gD) \sim \frac{16t}{\pi^2 \gep(t)^3},\\
&\gD'(f(t)) \sim \frac{t}{\pi}, \qquad
\gD''(f(t)) \sim -\frac{t^3}{\pi^3},
\ea
\eeq
and one readily obtains~\eqref{eq:asymp:qhat.deriv} and~\eqref{eq:asymp:qhat.derivsec} using 
\beq
\ba
&(\hat q_t^a)'(f) = \gD'(f) (\tilde q_t^a)'(\gD(f))\\
& (\hat q_t^a)''(f) = \gD''(f) (\tilde q_t^a)'(\gD(f)) + [\gD'(f)]^2 (\tilde q_t^a)''(\gD(f)),
\ea
\eeq
for $a\in\{0,1\}$.
\end{proof}
%%%%%%%%%%%%%%%%%%%%%%%%%%%%%%%%%%%%%%%%%%%
%%%%%%%%%%%%%%%%%%%%%%%%%%%%%%%%%%%%%%%%%%%

\section{Proof of Proposition~\ref{pr:goodenv}}
\label{sec:proof-good-env}

\begin{proof}
[Proof of Proposition~\ref{pr:goodenv}] The proposition follows from Lemma \ref{lem:goodB0} to \ref{lem:goodB78} below, in combination with \cite[Proposition 6.1]{PoiSim19} (an inspection of the proof therein reveals that $\gep_0$ may actually be chosen independently from $\gd$).
\end{proof}

\begin{lemma}
\label{lem:goodB0}
For all $\petit \in(0,1)$, there exists $\gep_0\in(0,1)$ small enough such that
\beq
\label{eq:cclB0}
\liminf_{n\to\infty} \bbP(\kBone_n(\gep_0)) \ge 1-\petit.
\eeq
\end{lemma}

\begin{lemma}
\label{lem:goodB6}
For all $\petit \in(0,1)$, there exists $\eta>0$ small enough such that
\beq
\label{eq:cclB6}
\liminf_{n\to\infty} \bbP(\kBtwo_n(\eta)) \ge 1-\petit.
\eeq
\end{lemma}

\begin{lemma}
\label{lem:goodB1}
For all $\petit \in(0,1)$, there exists $\rho\in(0,\frac12)$ small enough such that
\beq
\label{eq:cclB1}
\liminf_{n\to\infty} \bbP(\kBthree_n(\rho)) \ge 1-\petit.
\eeq
\end{lemma}

 \begin{lemma}
\label{lem:goodB2}
For all $\petit\in(0,1)$, there exists $J\ge 1$ large enough such that
\beq
\liminf_{n\to\infty} \bbP(\kBfour_n(J)) \ge 1-\petit.
\eeq
\end{lemma}

\begin{lemma}
\label{lem:goodB78}
For all $\petit \in(0,1)$, $\gep_0 \in (0,1)$ and all $\gep\in(0,\gb/2)$, there exists $\mathsf C = \mathsf C(\gep_0,\gep)>0$ so that 
\beq
\label{eq:cclB78}
\liminf_{n\to\infty} \bbP(\kBfive_n(\gep_0,\gep,\mathsf{C}))\ge 1-\petit.
\eeq
\end{lemma}

\subsection{Preparatory lemma} 
\label{sec:preplemma}
Recall the definition of $(\Pi_n)$ in \eqref{eq:defpin} that is a sequence of point processes on the quadrant $E:= [0, +\infty) \times (0, +\infty)$. In the following we abbreviate for all $n\ge 1$,
\beq
(X^n_i,Y^n_i) := \Big(\frac{i-1}{n}, \frac{T_i}{n^{1/\gamma}}\Big) \qquad \textrm{for all } i\geq 1.
\eeq
Each element of the sequence $(\Pi_n)$ belongs to $M_p(E)$, that is the space of Radon point measures on $E$. We recall that this sequence converges weakly (with the topology of vague convergence) to a Poisson point process $\Pi$ with intensity measure $p = \dd x \otimes \frac{c_\tau \gamma}{y^{\gamma + 1}} \dd y$, where $c_\tau$ is the constant appearing in~\eqref{eq:deftau} (see \cite[Proposition 2.4]{PoiSim19} for a reference). In \cite{PoiSim19} we defined for any $\lambda>0$,
\beq
\label{eq:psi0}
\begin{array}{ccccc}
 \psi^\lambda & : & E & \to & \R^+ \\
 & &  (x,y) & \mapsto & \lambda x + \frac{\pi^2}{2 y^2}, \\
\end{array}
\eeq
and for any $\mu$ in $M_p(E)$,
\begin{equation}
\label{eq:ystar}
\Psi^\lambda (\mu):=\inf_{(x,y)\in \mu} \psi^\lambda(x,y),
\end{equation}
with the convention $\inf \emptyset =+\8$ and $(x,y)\in \mu$ means that $(x,y)$ is in the support of $\mu$. We might omit the superscript $\gl$ when it does not bring confusion. When the infimum of $\Psi^\gl$ is achieved by a unique point we call it $z^*=(x^*,y^*)$. When the following quantities are uniquely defined, we call $\bar z=(\bar x, \bar y)$ the point in $\mu$ such that
\beq
\label{eq:ybar}
\bar y = \inf_{\substack{(x,y)\in \mu\ \colon \\x> x^* \text{ and } y>y^*}} y,
\eeq
and $\underline{z}=(\ub x, \ub y)$ the point in $\mu\setminus\{z^*\}$ such that
\beq
\label{eq:yub}
\ub y = \sup_{\substack{(x,y)\in \mu\setminus\{z^*\}\ \colon \\x<\bar x}} y.
\eeq
Finally, we call $z^{**}$ the minimizer of $\psi^\lambda$ over $\mu\setminus \{z^*\}$ when it is well-defined and unique, so that
\begin{equation}
\label{eq:ystarstar}
\psi^\lambda(z^{**}) = \inf_{(x,y)\in \mu\setminus \{z^*\}} \psi^\lambda(x,y).
\end{equation}
Provided these points are well-defined when one replaces $\mu$ by $\Pi_n$, we obtain four (random) points denoted by $Z_n^*$, $\bar Z_n$, $\ub Z_n$ and $Z_n^{**}$, respectively. 

\begin{lemma}
\label{lem:phikphi} For all $\petit \in (0,1)$, there exists a compact rectangle $K\subset E$ so that 
\beq
 \liminf_{n} \bbP(\cS(n))\geq 1-\petit,
 \eeq
where
\beq
\cS(n):=\{Z_n^*,\bar Z_n, \ub Z_n \text{ and } Z_n^{**} \text{ are well defined and lie in }K\}.
\eeq

 \end{lemma}
\begin{proof}[Proof of Lemma \ref{lem:phikphi}]
%%%%
We divide the proof in four steps.
\par \textit{Step $1.$} We first prove that for all $\petit>0$ there exists a compact set $K_\petit \subset E$ such that
\beq
\liminf_{n\to\infty} \bbP(\cS_1(n))> 1-\petit,
\eeq
where
\beq
\cS_1(n) = \Big\{\text{the infimum of $\psi$ over $\Pi_n$ is achieved by a unique minimizer that lies in }K_\petit \Big\}.
\eeq
We first choose $\gep>0$ small enough so that $\liminf_{n} \bbP(\max_{i\leq n} T_i\geq \gep n^{1/\gamma}) \geq 1-\petit$. On this event, considering any point $(X_i^n, Y_i^n)$ such that $X_i^n\le 1$ and $Y_i^n \ge \gep$, we easily obtain that $\inf_{\Pi_n} \psi\leq \lambda +\pi^2/(2\gep^2)$. Since the set $\{z\in E \colon \psi(z) \le \gl +\pi^2/(2\gep^2)\}$ is contained in $[0, 1+\frac{\pi^2}{2\lambda \gep^2}] \times (0,+\infty)$, it contains almost surely finitely many points of $\Pi_n$. This implies that the infimum over $\Pi_n$ is almost surely achieved by at least one point in $\Pi_n$ and that all minimizers $(x,y)\in \Pi_n$ satisfy $x \leq b_\gep := 1+\frac{\pi^2}{2\lambda \gep^2}$ and $y\geq c_\gep:=\pi/\sqrt{2\gl b_\gep}$. We now fix $d_\gep>c_\gep$ large enough so that $\liminf_{n} \bbP(\max_{i\leq b_\gep n} T_i\leq d_\gep n^{1/\gamma})\geq 1-\petit$ and, in order to anticipate Step 3, we also fix $0<a_\gep< b_\gep$ small enough so that $\liminf_n \bbP(\max_{i\leq a_\gep n} T_i< c_\gep n^{1/\gamma})\ge 1-\petit$. We have thus proven that on the set 
\beq
\{\max_{i\leq n} T_i\geq \gep n^{1/\gamma}\} \cap \{\max_{i\leq b_\gep n} T_i\leq d_\gep n^{1/\gamma}\} \cap \{\max_{i\leq a_\gep n} T_i<  c_\gep n^{1/\gamma}\},
\eeq
the infimum of $\psi$ over $\Pi_n$ is achieved and all minimizers lie in the compact set $K_\petit:=[a_\gep,b_\gep]\times [c_\gep, d_\gep]$. We now prove that there is a unique minimizer with high probability. We consider the event 
\beq
A=\Big\{\mu \in M_p(E),\ \inf_{z\in \mu\cap K_\petit}\psi(z) \text{ has at least two minimizers}\Big\}.
\eeq
Using \cite[Proposition 3.13]{Res}, one can check that all $\mu\in \partial A$ belong to $A$ or  satisfy $\mu(\partial K_\petit)\geq 1$. Using \cite[Proposition 3.14]{Res}, we obtain that $\Pi(\partial A)=0$ and thus, \beq
\lim_{n\to +\8}\Pi_n(A)=\Pi(A)=0.
\eeq
This concludes the proof of the first step.
%%%%
\par \textit{Step $2.$} Let us first note that almost surely on $\cS_1(n)$, $Z^*_n$ and $\bar Z_n$ are well defined. We now prove that for all $\petit>0$ one can enlarge $K_\petit$ to another compact set $\bar K_\petit \subset E$ such that
\beq
\liminf_{n} \bbP(\cS_2(n))\geq 1-\petit,
\eeq
where
\beq
\cS_2(n) := \cS_1(n) \cap \Big\{\bar Z_n\in \bar K_\petit \Big\}.
\eeq
We observe that for all $\bar b_\gep > b_\gep$,
\beq
\bbP(\cS_1(n), \bar X_n > \bar b_\gep  )\leq \bbP(\argmax_{i\leq \bar b_\gep n}T_i\leq b_\gep n).
\eeq
We choose $\bar b_\gep > b_\gep$ large enough so that $\liminf_n \bbP(\argmax_{i\leq \bar b_\gep n}T_i\leq b_\gep n)\leq \petit$ and $\bar d_\gep > d_\gep$ large enough so that $\liminf_{n} \bbP(\max_{i\leq \bar b_\gep n} T_i\leq \bar d_\gep n^{1/\gamma})\geq 1-\petit$. Setting $\bar K_\petit:=[a_\gep,\bar b_\gep]\times [c_\gep, \bar d_\gep]$ concludes the proof of the second step.
%%%%
\par \textit{Step $3.$} 
We finally prove that for all $\petit>0$ one can enlarge $\bar K_\petit$ to a compact  set $\ub K_\petit \subset E$ such that
\beq
\liminf_{n\to\infty} \bbP(\cS_2(n), \ub Z_n \text{ is well defined, unique, and lies in }\ub K_\petit)\geq 1-\petit.
\eeq
The supremum in \eqref{eq:yub} is achieved (with $\Pi_n$ instead of $\mu$) and any maximizer $(\ub X_n,\ub Y_n)$ satisfies $\ub X_n \leq \bar X_n$ and $\ub Y_n\leq Y^*_n$, by definition. We now choose $\ub c_\gep < c_\gep$ small enough so that $\limsup_n \bbP(\max_{i\le a_\gep n} T_i < \ub c_\gep n^{1/\gga}) \le \petit$ and we set $\ub K_\petit = [0, \bar b_\gep]\times [\ub c_\gep, \bar d_\gep]$ so that, on $\{\max_{i\le a_\gep n} T_i \ge \ub c_\gep n^{1/\gga}\}\cap \cS_2(n)$, any maximizer $\ub Z_n$ belongs to $\ub K_\petit$. We finally prove that for $n$ large enough, $\ub Z_n$ is uniquely defined with probability larger that $1-\petit$ using the same argument that was used at the end of Step 1.\\
%%%%
\par \textit{Step $4.$}
Using the same arguments as above one can prove that the infimum in \eqref{eq:ystarstar} is well defined and achieved (again with $\Pi_n$ instead of $\mu$).  On $\cS(n)$, as $\bar Z_n\in K$, one may deduce that any minimizer $Z^{**}_n$ in \eqref{eq:ystarstar} satisfies $X^{**}_n\leq b^{**}_\gep:=\bar b_\gep+\frac{\pi^2}{2\gl\ub c^2_\gep}$ and $Y^{**}_n\geq c^{**}_\gep := \frac{\pi}{\sqrt{2\gl b^{**}_\gep}}$. One may finally choose $d^{**}_\gep \geq \bar d_\gep$ large enough so that
 \beq
 \limsup_{n\to\infty} \bbP(\max_{i\leq  b^{**}_\gep n}T_i> d^{**}_\gep n^{1/\gamma} )\leq \petit.
 \eeq
 On $\cS(n)\cap \{\max_{i\leq  b^{**}_\gep n}T_i\leq d^{**}_\gep n^{1/\gamma}\}$, one obtain that $Z^{**}_n$ lies in the compact set $K^{**}=[0,b^{**}_\gep]\times[c^{**}_\gep, d^{**}_\gep]$. Once we know that all minimizers lie in $K^{**}$ we prove that there is actually only one by using the same argument as in Step 1.\\
%%%%%
\end{proof}

\subsection{Technical proofs for Section~\ref{sec:goodenv}}

\begin{proof}[Proof of Lemma~\ref{lem:goodB0} and Lemma~\ref{lem:goodB6}]
{\it Step 1.} Using the compact rectangle from Step 1 of Lemma \ref{lem:phikphi} (that does not intersect the $y$-axis) and noting that
\beq
Z^*_N = \Big( \frac{\tilde\lopt(n)-1}{N}, \frac{T_{\tilde\lopt(n)}}{N^{1/\gga}} \Big),
\eeq 
we already know that for all $\petit\in(0,1)$, there exists $\gep_0\in(0,1)$ such that $\tilde \lopt$ is uniquely defined, with $\gep_0 N \le \tilde\lopt \le \gep_0^{-1} N$ and $\gep_0 N^{\frac 1 \gga} \le T_{\tilde\lopt} \le \gep_0^{-1} N^{\frac 1 \gga}$. 
\par {\it Step 2.} Let us now prove that all minimizers of  $G_n$ lie in the interval $[\gep_0N, \gep_0^{-1}N]$ with probability larger than $1-\petit$, provided $\gep_0$ is chosen small enough. We replicate the argument used in the proof of Lemma~\ref{lem:phikphi} (Step 1). First, we choose $\gep>0$ small enough so that $\liminf_{n} \bbP(\max_{i\leq N} T_i\geq \gep N^{1/\gamma}) \geq 1-\petit$. On this event, considering any point $i\leq N$ such that $T_i\geq \gep N^{1/\gamma}$, we easily obtain using \eqref{eq:encadr_gphi},
\beq
\inf_{\ell}G_n(\ell)\leq \gb+\bbE[\log T_1] +\frac{C_1}{\gep^2}.
\eeq
This implies that the infimum is almost surely achieved and that all minimizers satisfy $\ell\leq b_\gep N$ where $b_\gep:=1+\gb^{-1}({\bbE[\log T_1]} +\frac{C_1}{\gep^2})$ and $T_\ell \geq c_\gep N^{1/\gamma}$ with $c_\gep :=C_1^{-1}(\gb+\bbE[\log T_1] +\frac{C_1}{\gep^2})^{-1}$. We fix $0<a_\gep< b_\gep$ small enough so that $\liminf_n \bbP(\max_{i\leq a_\gep N} T_i< c_\gep N^{1/\gamma})\ge 1-\petit$. On the set 
\beq
\label{eq:B26}
\Big\{\max_{i\leq N} T_i\geq \gep N^{1/\gamma}\Big\} \cap \Big\{\max_{i\leq N} T_i\leq \gep N^{1/\gamma}\Big\}  \cap \Big\{\max_{i\leq a_\gep N} T_i<  c_\gep N^{1/\gamma}\Big\},
\eeq
any minimizer $\ell$ of $G_n$ indeed lies in $[a_\gep N, b_\gep N]$ and satisfies $T_{\ell} \ge c_\gep N^{1/\gga}$.
\par {\it Step 3.} Let us now prove that $\lopt$ is uniquely defined and actually coincides with $\tilde \lopt$ for $n$ large enough, with probability larger than $1-\petit$. To this end, let us first notice that for every $\eta\in(0,1)$, every $\ell$ that minimizes $G_n$, on the intersection of~\eqref{eq:B26} and $A^{(10)}_n(\gep_0,\eta)$, and for $n$ large enough (so that, in particular, $A^{(10)}_n(\gep_0,\eta)$ has probability at least $1-\petit$),
\beq
\label{eq:diffGGtilde}
|G_n(\ell)-\tilde G_n(\ell)| \leq \frac{3\eta}{4}.
\eeq
The next step is proving that with probability larger than $1-\petit$,
\beq
\label{eq:secondmin-tildeG}
\min_{\ell \neq \tilde \lopt} \tilde G_n(\ell) - \min \tilde G_n(\ell) \ge 2\eta.
\eeq
Using the notation of Section~\ref{sec:preplemma}, we set for any $\mu\in M_p(E)$,
\beq
\Phi(\mu)=\psi^\gl(z^*)-\psi^\gl(z^{**}),
\eeq
if $z^*$ and $z^{**}$ are both uniquely defined (otherwise set the function to zero) and $\Phi_K(\mu)=\Phi(\mu(\cdot \cap K))$, where $K$ stands for a compact subset of $[0,+\infty)\times (0,+\infty)$. Adapting the proof of Lemma \ref{lem:phikphi}, one can show that $\Pi$ is almost surely a continuity point of $\Phi_K$ and that $\Phi_K(\Pi)$ has no atom in $\R$. This implies that for all $\eta > 0$
\beq
\label{eq:convdeux}
\lim_{n\to +\8}\bbP(\Phi_K(\Pi_n)<2\eta)=\bbP(\Phi_K(\Pi)<2\eta).
\eeq
 On $\cS^{**}_n$, $\min_{\ell\neq \lopt}  G_n^{\beta}(\ell)- \min_{\ell}  G_n^{\beta}(\ell)=\Phi_K(\Pi_n)$ so that for all $\eta>0$,
\beq
\label{eq:proofB2}
\bbP(\min_{\ell\neq \lopt}  G_n^{\beta}(\ell)- \min_{\ell}  G_n^{\beta}(\ell)< 2\eta)\leq \bbP(\Phi_K(\Pi_n)<2\eta)+\bbP\left((\cS^{**}_n)^c\right).
\eeq
We thus choose $\eta>0$ small enough so that $\bbP(\Phi_K(\Pi)<2\eta)\leq \petit$ and conclude the proof using \eqref{eq:convdeux} and Lemma \ref{lem:phikphi}.
Combining~\eqref{eq:diffGGtilde} and~\eqref{eq:secondmin-tildeG}, we get the desired result. As a byproduct, we use the inequality in~\eqref{eq:proofB2} to prove Lemma~\ref{lem:goodB6}.
\end{proof}

\begin{proof}[Proof of Lemma~\ref{lem:goodB1}] By Lemma~\ref{lem:goodB0}, we may safely assume that $\lopt = \tilde \lopt$ in what follows. The idea is to write the ratio between $\max_{i\neq \lopt, i< i(\kopt+1)} T_i$ and $T_{\lopt+1}$ as a function of the point process $\Pi_n$ and then use the convergence of $\Pi_n$ to $\Pi$. 
In the case where $Z_n^*,\bar Z_n$ and $\ub Z_n$ are well defined we set
\beq
\Phi(\Pi_n)=\frac{\ub Y_n}{Y_n^*},
\eeq
and $\Phi(\Pi_n)=0$ otherwise. 
We consider a compact set $K\subset E$ such that $\bbP(\Pi(\partial K)=0)=1$ and define the function
$\Phi_K$ for all $\mu\in M_p(E)$ by
\beq
\Phi_K(\mu)=\Phi(\mu(\cdot \cap K)).
\eeq
Let us check that $\Pi$ is almost surely a continuity point of $\Phi_K$ (that is however not the case for $\Phi$). Almost surely, $\Pi(\cdot \cap K)$ has only a finite number of atoms and none of them belongs to $\partial K$. Moreover, arguing as in~\cite[Proposition~2.5]{PoiSim19}, for $\mu=\Pi(\cdot \cap K)$, the optimisers in \eqref{eq:ystar} \eqref{eq:ybar} \eqref{eq:yub} are almost surely unique, provided they are well defined (that is if the set of point satisfying the constraints is not empty). Using~\cite[Proposition~$3.13$]{Res} one can check that $\Pi$ is thus almost surely a continuity point of $\Phi_K$ (even in the case where one of the optimisers is not well defined and $\Phi_K=0$). As a consequence, using also~\cite[Proposition $2.4$]{PoiSim19}, $(\Phi_K(\Pi_n))$ converges in law to $\Phi_K(\Pi)$.
\par Recall that for $m\in\N$ and conditional on $\{\Pi(K)=m\}$, the restriction of $\Pi$ to $K$ has the same law as $\sum_{i=1}^m \delta_{Z_i}$, where the $(Z_i)_{1\le i\leq m}$'s are i.i.d. with continuous law $p_K$ that is the intensity measure of $\Pi$ restricted to $K$ and renormalized to a probability measure. From this we deduce that $\Phi_K(\Pi)$ has no atom in $\bbR$. This implies that for all $\rho \in (0,1)$,
\beq
\label{eq:convcompact}
\lim_{n\to +\8}\bbP(\Phi_K (\Pi_n) \geq 1-\rho) = \bbP(\Phi_K (\Pi) \geq 1-\rho).
\eeq
For $\petit>0$, we consider the compact set $K=[0, \bar b_\gep]\times [\ub c_\gep, \bar d_\gep]\subset E$ given by Lemma \ref{lem:phikphi} and write for all $\rho\in (0,1)$
\beq
\bbP\Big( \frac{ \max_{i\neq \lopt, i< i(\kopt+1)} T_i }{T_{\lopt}} \geq 1-\rho \Big) \leq \bbP\Big(\Phi_K (\Pi_n) \geq 1-\rho\Big)+\bbP\Big(\cS(n)^c\Big).
\eeq
We choose $\rho$ small enough such that $\bbP(\Phi_K (\Pi) \geq 1-\rho)\leq \petit$ so that using \eqref{eq:convcompact} and Lemma \ref{lem:phikphi}, one can easily deduce
\beq
\limsup_{n\to +\8} \bbP\Big( \frac{ \max_{i\neq \lopt, i< i(\kopt+1)} T_i }{T_{\lopt}} \geq 1-\rho \Big)\leq \petit.
\eeq
If moreover $\rho<\ub c_\gep /\bar d_\gep$ then
\beq
\bbP\Big( \frac{ \max_{i\neq \lopt, i< i(\kopt+1)} T_i }{T_{\lopt}} \leq \rho \Big) \leq \bbP\Big(\cS(n)^c\Big).
\eeq
This concludes the proof of Lemma~\ref{lem:goodB1}.
\end{proof}

%%%%%%%%%%%%%%%%%%%%%%%%%%%
\begin{proof}[Proof of Lemma~\ref{lem:goodB2}]
By Lemma~\ref{lem:goodB0}, we may safely assume that $\lopt = \tilde \lopt$ in what follows.
The event $\kBfour_n(J)^c$ splits in two parts and we only control the first one as both are similar.  We observe that 
\beq
\ba
\label{eq:decompo2termes}
&\bbP(\exists \ell\geq J,\ \max\{T_{\ell_0+i} \colon i\le \ell\} \geq \ell^{\frac{4+\gamma}{4\gamma}})\\
&\qquad \leq \bbP(\max\{T_{\ell_0+i} \colon i\le J\} \geq J^{\frac{4+\gamma}{4\gamma}})+\bbP(\exists \ell\geq J,\ T_{\ell_0+\ell} \geq \ell^{\frac{4+\gamma}{4\gamma}}).
\ea
\eeq
{\it Step 1.} Let us first manage with the second term. 
Note that the family $(T_{\tilde \ell_0+\ell})$ is not i.i.d. anymore as $\tilde \ell_0$ is random. Our first task is thus to bring us back to this i.i.d. case, see \eqref{eq:back2iid} below. To this end, we decompose according to the different possible values of $\tilde \ell_0$ and $T_{\tilde \ell_0}$:
\beq
\label{eq:decompoB}
\ba
\bbP(\exists \ell\geq J,\ T_{\tilde\ell_0+\ell} \geq \ell^{\frac{4+\gamma}{4\gamma}})
&=\sum_{k,t} \bbP(\exists \ell\ge J, T_{k+\ell} \ge \ell^{\frac{4+\gga}{4\gga}}, \tilde \ell_0 = k, T_k = t)\\
\leq 
&\sum_{k,t}
\bbP(\exists \ell\geq J,\ T_{k+\ell} \geq \ell^{\frac{4+\gamma}{4\gamma}}, \forall \ell> -k,\ \tilde G_n^\gb(k+\ell)\geq \tilde G_n^\gb(k),T_k=t).
\ea
\eeq
For all $\ell> -k$, using~\eqref{eq:tildeG} one can rewrite the event $\{\tilde G_n^\gb(k+\ell)\geq \tilde G_n^\gb(k),T_k=t\}$ as $\{T_{k+\ell}\leq U(\ell,t),T_k=t\}$, where 
\beq
U(\ell,t)= \Big[\Big(\frac{1}{t^2} - \frac{2\gl \ell}{\pi^2 n} \Big) \vee 0\Big]^{-1/2}.
\eeq
As the r.v. $T_\ell$'s are i.i.d. this leads to 
\beq
\ba
{\rm r.h.s.}\ \eqref{eq:decompoB}&\leq \sum_{k,t}
\bbP(\exists \ell\geq J,\ T_{k+\ell} \geq \ell^{\frac{4+\gamma}{4\gamma}}, \forall \ell> -k,\ T_{k+\ell}\leq U(\ell,t),T_k=t)\\
&\leq \sum_{k,t}
\bbP(\forall \ell\in(-k,0),\ T_{k+\ell}\leq U(\ell,t),T_k=t)\\
&\qquad \qquad \times \bbP(\exists \ell\geq J,\ T_{k+\ell} \geq \ell^{\frac{4+\gamma}{4\gamma}}, \forall \ell\geq 1,\ T_{k+\ell}\leq U(\ell,t)).
\ea
\eeq
As the event $\{\exists \ell\geq J,\ T_{k+\ell} \geq \ell^{\frac{4+\gamma}{4\gamma}}\}$ is non-decreasing (coordinatewise) with $(T_{k+\ell})_{\ell\geq 1}$ while $\{\forall \ell\geq 1,\ T_{k+\ell}\leq U(\ell,t)\}$ is non-increasing (coordinatewise) with respect to the same random sequence, the FKG inequality gives for all $k$ and $t$,
\beq
\ba
&\bbP(\exists \ell\geq J,\ T_{k+\ell} \geq \ell^{\frac{4+\gamma}{4\gamma}}, \forall \ell\geq 1,\ T_{k+\ell}\leq U(\ell,t))\\
&\qquad \leq \bbP(\exists \ell\geq J,\ T_{k+\ell} \geq \ell^{\frac{4+\gamma}{4\gamma}})
\bbP( \forall \ell\geq 1,\ T_{k+\ell}\leq U(\ell,t)).
\ea
\eeq
As the $T_{\ell}$'s are i.i.d., we have for all $k
\geq 0$
\beq
\bbP(\exists \ell\geq J,\ T_{k+\ell} \geq \ell^{\frac{4+\gamma}{4\gamma}})=\bbP(\exists \ell\geq J,\ T_{\ell} \geq \ell^{\frac{4+\gamma}{4\gamma}}).
\eeq
Finally, we obtain
\beq
\label{eq:back2iid}
\ba
\bbP(\exists \ell\geq J,\ T_{\tilde\ell_0+\ell} \geq \ell^{\frac{4+\gamma}{4\gamma}})
&\leq 
\bbP(\exists \ell\geq J,\ T_{\ell} \geq \ell^{\frac{4+\gamma}{4\gamma}})\sum_{k,t}
\bbP(\tilde\ell_0=k,T_k=t)\\
&\leq 
\bbP(\exists \ell\geq J,\ T_{\ell} \geq \ell^{\frac{4+\gamma}{4\gamma}}).
\ea
\eeq
Let us now recall that $i_k$ is the index of the $k-$th record of the sequence $(T_\ell)_{\ell\geq 0}$. We thus obtain that for all $K\geq 1$,
\beq
\bbP(\exists \ell\geq J,\ T_{\ell} \geq \ell^{\frac{4+\gamma}{4\gamma}})\leq \bbP(\exists k\geq K,\  T_{i_k}> i_k^{\frac{4+\gamma}{4\gamma}})+\bbP(i_{K} >  J).
\eeq
We first observe that
\beq
\ba
\sum_{k\geq 1}\bbP(T_{i_k} > i_k^{\frac{4+\gamma}{4\gamma}})&=\sum_{j\geq 1}\bbP(T_{j} > j^{\frac{4+\gamma}{4\gamma}}, T_i < T_j \ \text{for all }i<j)\\
				&\leq \sum_{j\geq 1}\frac{1}{j}\bbP\Big(\max_{i\leq j}T_{i} > j^{\frac{4+\gamma}{4\gamma}}\Big)\\
				&\leq \sum_{j\geq 1}\frac{1}{j}\Big[1- \Big\{1-\bbP(T_1\geq j^{\frac{4+\gamma}{4\gamma}})\Big\}^j\Big].\\
\ea
\eeq
Note that we lose the equality at the second line as two or more variables could achieve simultaneously the maximum. The term in the last sum is equivalent to $j^{-(1+\frac\gga 4)}$ and, as $\gamma>0$, the sum converges. We can thus fix $K$ large enough so that 
\beq
\bbP\Big(\exists k\geq K,\  T_{i_k}> i_k^{\frac{4+\gamma}{4\gamma}}\Big)\leq \petit.
\eeq
Finally, as $i_{K}<+\8$ almost surely, we choose $J$ so that $\bbP(i_{K} >  J)\leq \petit$.\\

{\it Step 2.} We use the same ideas to manage with the first term in \eqref{eq:decompo2termes} and obtain
\beq
\bbP(\max\{T_{\tilde \ell_0 +i} \colon i\le J\} \geq J^{\frac{4+\gamma}{4\gamma}})
\le \bbP(\max\{T_{i} \colon i\le J\} \geq J^{\frac{4+\gamma}{4\gamma}}).
\eeq
We conclude easily as this does not depend on $n$ and goes to $0$ with $J$ going to $+\8$.\\
\end{proof}

\begin{proof}[Proof of Lemma~\ref{lem:goodB78}]
The random variable $\card\{k\in [\gep_0 N, \gep_0^{-1}N] \colon T_k \ge \frac{\alpha(\gep)\gep_0}{4}  N^{\frac 1 \gga}\}$ is distributed as a binomial random variable with parameters $(\gep_0^{-1}- \gep_0)N$ and $\bbP(T_1 \ge \frac{\alpha(\gep)\gep_0}{4}  N^{\frac 1 \gga}) \sim (\frac{\alpha(\gep)\gep_0}{4})^{-\gga}N^{-1}$. As such, it converges in law to a Poisson random variable as $n\to\infty$, which proves our result.
\end{proof}

%%%%%%%%%%%%%%%%%%%%%%%%%%%%%%%%%%%%%%%%%%%
\section{Proof of Lemma~\ref{lem:mass-ren-fct} and Lemma~\ref{lem:sharp-Zn-renorm}}
\label{app:plms}
\begin{proof}[Proof of Lemma~\ref{lem:mass-ren-fct}]
(i) Let us start with the lower bound. Let 
\beq
\ba
\Inw(k) &= \{\text{the last excursion before $k$ is inward}\},\\
\Out(k) &= \{\text{the last excursion before $k$ is outward}\},
\ea
\eeq
that is $\Inw(k) = \{Y_\iota = \ins\}$ and $\Out(k) = \{Y_\iota = \out\}$, where $\iota = \sup\{i\ge 0 \colon \bar\gt_i < k\}$.
We write (the constant $c_5$ below being the same as in Lemma~\ref{lem:estim-qtn})
\beq
\cP_0(k\in \bar\gt) \ge \cP_0(k\in\bar\gt,\Inw(k-c_5\tone^2), \bar\gt \cap [k-\tone^3,k-c_5\tone^2)\neq \emptyset).
\eeq
By decomposing on the value of
\beq
\gga_{k-\tone^3} = \inf\{i\ge k-\tone^3\colon i\in\bar\gt\},
\eeq
we get
\beq
\cP_0(k\in \bar\gt) \ge \sumtwo{0\le i < \tone^3-c_5\tone^2}{a\in\{0,1\}}
\cP_0^a(\gga_{k-\tone^3} = k-\tone^3+ i) \cP_a(\tone^3 - i \in \bar\gt, \Inw(\tone^3-c_5\tone^2 - i)).
\eeq
Note that
\beq
\ba
&\cP_a(\tone^3 - i \in \bar\gt, \Inw(\tone^3-c_5\tone^2 - i))\\
&\qquad \ge \cP_a(\Inw(\tone^3-c_5\tone^2 - i)) \cP_a(\bar\gt_1 = \tone^3 - i | \Inw(\tone^3-c_5\tone^2 - i)).
\ea
\eeq
By Lemma~\ref{lem:estim-qtn}, since $\tone^3 -i \ge c_5\tone^2$,
\beq
\cP_a(\bar\gt_1 = \tone^3 - i | \Inw(\tone^3-c_5\tone^2 - i)) \ge \frac{\rm{(cst)}}{\tone^3} \exp[(\phi-g(\tone))(\tone^3 - i)] \ge \frac{\rm{(cst)}}{\tone^3}.
\eeq
We have used above the fact that, by Lemma~\ref{lem:sandw:phit1t2},
\beq
\phi-g(\tone) \ge \phi(\tone) - g(\tone) \ge - \frac{\rm{(cst)}}{\tone^3}.
\eeq
In the case where $\tone$ and $k$ are both even, we note that $k- \gamma_{k-\tone^3}$ is necessarily even. Therefore, we obtain
\beq
\cP_0(k\in \bar\gt) \ge \frac{\rm{(cst)}}{\tone^3} \cP_0(\Inw(k-c_5\tone^2), \bar\gt \cap [k-\tone^3,k-c_5\tone^2)\neq \emptyset).
\eeq
We will now prove that
\beq
\label{eq:LB-inw0}
\cP_0(\Inw(k-c_5\tone^2), \bar\gt \cap [k-\tone^3,k-c_5\tone^2)\neq \emptyset) \ge {\rm (cst)} 
\cP_0(\Inw(k-c_5\tone^2)),
\eeq
by showing that
\beq
\ba
&\cP_0(\Inw(k-c_5\tone^2), \bar\gt \cap [k-\tone^3,k-c_5\tone^2)\neq \emptyset)\\
&\qquad \ge {\rm (cst)}
\cP_0(\Inw(k-c_5\tone^2), \bar\gt \cap [k-\tone^3,k-c_5\tone^2)= \emptyset).
\ea
\eeq
To this end, we decompose the left-hand side according to the rightmost point in $\bar\gt$ before $k-\tone^3$ and the leftmost point strictly after $k-\tone^3$:
\beq
\ba
&\cP_0(\Inw(k-c_5\tone^2), \bar\gt \cap [k-\tone^3,k-c_5\tone^2)\neq \emptyset)\\
&\ge 
\sumtwo{0\le \ell < k-\tone^3 \le m < k -c_5\tone^2}{a,b\in\{0,1\}}  \cP_0^a(\ell \in \bar\gt) 
\cP_a^b(\bar \gt_1 = m-\ell, Y_1 = \ins) \cP_b(\Inw(k-m-c_5\tone^2)).
\ea
\eeq
By \eqref{eq:LB-Pinw} (see below), we obtain:
\beq
\ba
&\cP_0(\Inw(k-c_5\tone^2), \bar\gt \cap [k-\tone^3,k-c_5\tone^2)\neq \emptyset)\\
&\qquad \ge  {\rm (cst)}
\sumtwo{0\le \ell < k-\tone^3}{a\in\{0,1\}} \cP_0^a(\ell \in \bar\gt) 
\cP_a(k-\tone^3 - \ell \le \bar\gt_1 < k -c_5\tone^2- \ell, Y_1 = \ins).
\ea
\eeq
By Lemma~\ref{lem:cP_tail}, for all $u,v\in 2\bbN$ such that $u<v$,
\beq
\label{eq:LB-cP-in}
\cP_a(u \le \bar\gt_1 < v, Y_1 = \ins) \ge {\rm (cst)} 
(e^{-(g(\tone)-\phi)u} - e^{-(g(\tone)-\phi)v}),
\eeq
and we obtain 
\beq
\ba
\label{eq:LB-inw}
&\cP_0(\Inw(k-c_5\tone^2), \bar\gt \cap [k-\tone^3,k-c_5\tone^2)\neq \emptyset)\\
&\qquad \ge
{\rm (cst)} \sum_{0\le \ell < k-\tone^3} \cP_0(\ell \in \bar\gt)(e^{-(g(\tone)-\phi)(k-\tone^3-\ell)} - e^{-(g(\tone)-\phi)(k-c_5\tone^2-\ell)}).
\ea
\eeq
(Actually, $k-\tone^3-\ell$ should be replaced by the smallest even integer larger than $k-\tone^3-\ell$ and $k-c_5\tone^2-\ell$ by the largest even integer smaller than $k-c_5\tone^2-\ell$.)
On the other hand, we have, with the same decomposition,
\beq
\ba
&\cP_0(\Inw(k-c_5\tone^2), \bar\gt \cap [k-\tone^3,k-c_5\tone^2)= \emptyset)\\
&\qquad = \sumtwo{0\le \ell < k-\tone^3}{a\in\{0,1\}}\cP_0^a(\ell \in \bar\gt) \cP_a(\bar\gt_1 \ge k-c_5\tone^2-\ell, Y_1 = \ins).
\ea
\eeq
Using the same upper bound as (2.17) in~\cite[Lemma 2.2]{CP09b}, we get
\beq
\ba
\label{eq:UB-inw}
&\cP_0(\Inw(k-c_5\tone^2), \bar\gt \cap [k-\tone^3,k-c_5\tone^2)= \emptyset)\\
&\qquad \le {\rm (cst)} \sum_{0\le \ell < k-\tone^3} \cP_0(\ell \in \bar\gt) e^{-(g(\tone)-\phi)(k-c_5\tone^2-\ell)}.
\ea
\eeq
(Actually, $k-c_5\tone^2-\ell$ should be replaced by the smallest even integer larger than $k-c_5\tone^2-\ell$.) Recall that from Lemma~\ref{lem:sandw:phit1t2},
\beq
\label{eq:sdwch-g-minus-phi}
\frac{1}{{\rm (cst)}\tone^3} \le g(\tone)-\phi(\bbN,\tone,\bbN) \le g(\tone)-\phi \le g(\tone)-\phi(\tone) \le \frac{\rm (cst)}{\tone^3}.
\eeq
Therefore, combining~\eqref{eq:sdwch-g-minus-phi}, \eqref{eq:LB-inw} and~\eqref{eq:UB-inw}, we get the desired inequality in~\eqref{eq:LB-inw0}. To conclude the lower bound part, it remains to prove that uniformly in $k$ and $b\in\{0,1\}$,
\beq
\label{eq:LB-Pinw}
\cP_b(\Inw(k)) \ge {\rm (cst)} >0.
\eeq
Without any loss in generality, we may assume that $b=0$. Again, it is enough to show that
\beq
\cP_0(\Inw(k)) \ge {\rm (cst)} \cP_0(\Out(k)).
\eeq
By decomposing on the last contact point before $k$, we get
\beq
\cP_0(\Inw(k)) = \sumtwo{0\le i < k}{a\in\{0,1\}} \cP_0^a(i\in\bar\gt)
\cP_a(Y_1 = \ins, \bar \gt_1 \ge k-i).
\eeq
Therefore, it is enough to prove that for all $\ell \in \bbN$,
\beq
\cP_a(Y_1 = \ins, \bar\gt_1 \ge \ell) \ge {\rm (cst)} \cP_a(Y_1 = \out, \bar\gt_1 \ge \ell).
\eeq
We distinguish between $\ell \le 5\tone^3$ and $\ell > 5\tone^3$. If $\ell \le 5\tone^3$ then (recall~\eqref{eq:LB-cP-in})
\beq
\label{eq:LBPain}
\ba
\cP_a(Y_1 = \ins, \bar \gt_1 \ge \ell) &\ge {\rm (cst)} e^{-(g(\tone)-\phi)\ell}\\
&\ge {\rm (cst)}\\
&\ge {\rm (cst)} \cP_a(Y_1 = \out, \bar \gt_1 \ge \ell).
\ea
\eeq
Now suppose that $\ell > 5\tone^3$. From Proposition~\ref{prop:roughUB}, we get
\beq
\cP_a(Y_1 = \out, \bar\gt_1 \ge \ell) \le {\rm (cst)}\sum_{k\ge\ell} k^3 e^{(\phi-\phi(\ttwo))k}.
\eeq
Since $\ttwo \le (1-\rho) \tone$, we get $\phi - \phi(\ttwo) \le -\rho g(\tone)$, thus
\beq
\cP_a(Y_1 = \out, \bar\gt_1 \ge \ell) \le \sum_{k\ge\ell}k^3 e^{-\frac12\rho g(\tone)k} \times e^{-\frac12\rho g(\tone)k}.
\eeq
Since $i \mapsto i^3 e^{-\frac12\rho g(\tone)i}$ is decreasing when $i \ge 5\tone^3 \ge 5 \tone^2 / \rho$, we get
\beq
\ba
\cP_a(Y_1 = \out, \bar\gt_1 \ge \ell) &\le \tone^9 e^{-\frac12\rho \tone} \sum_{i\ge\ell} e^{-\frac12\rho g(\tone)i}\\
&\le C(\rho) \exp(- \tfrac12\rho g(\tone) \ell).
\ea
\eeq
Therefore, using~\eqref{eq:sdwch-g-minus-phi} and the first line of~\eqref{eq:LBPain}, we obtain
\beq
\frac{\cP_a(Y_1 = \ins, \bar\gt_1 \ge \ell)}{\cP_a(Y_1 = \out, \bar\gt_1 \ge \ell)}
\ge \frac{1}{C(\rho)}\exp\Big[\Big(\tfrac12\rho g(\tone) - \frac{\rm cst}{\tone^3}\Big)\ell\Big] \ge \frac{1}{C(\rho)}.
\eeq
which concludes this part of the proof.

\par (ii) We continue with the upper bound. First, we note that there exists $\ga < 1$ such that for all $\tone$ large enough and $a\in \{0,1\}$,
\beq
\sum_{u\le \tone^3}\cP_a(\bar\gt_1= u) \le \ga,
\eeq
following the proof strategy of ~\cite[Proposition 2.3]{CP09b}. The inequality above can be shown by using Lemma~\ref{lem:cP_tail}. Then, we show that there exists a constant $C\in(0,\infty)$ such that
\beq
\cP_a(\bar\gt_1=k) \le \frac{C}{k^{3/2}\wedge \tone^3},
\qquad (k\le \tone^3).
\eeq
This is a consequence of
\beq
\label{eq:UB-small-n}
\cP_a(\bar\gt_1=k, Y_1=\ins) \le \frac{C}{k^{3/2}\wedge \tone^3},
\qquad (k\le \tone^3),
\eeq
which comes from Lemma \ref{lem:estim-qtn}, and 
\beq
\cP_a(\bar\gt_1=k, Y_1=\out) \le
\begin{cases}
C k^{-3/2} & (k<\tone^2)\\
Ck^3 \exp\Big(- k^{\frac{\gga}{2(\gga+2)}}\Big) & (\tone^2 \le k<\tone^3)\\
\end{cases}
\eeq
where the second line comes from Lemma~\ref{lem:Zout} while the first one is just a basic control on the return time at $0$ of a simple random walk. The rest of the proof of the upper bound for $k\le \tone^3$ proceed by induction as in~\cite[Proposition 2.3]{CP09b}. Let us now prove it for $k\ge \tone^3$. The proof starts in the same way as in~\cite{CP09b}. It is enough to prove the equivalent of (B.13) therein, that is~\eqref{eq:equivalB13} below. To this purpose we may write
\beq
\ba
&\cP_0(\bar\gt \cap [k-\tone^3, k- \tone^2] \neq \emptyset, k\in\bar\gt)\\
&=\sumtwo{0\le m <k-\tone^3}{a\in\{0,1\}} \cP_0^a(m\in\bar\gt) \sumtwo{k-\tone^3 \le \ell \le k-\tone^2}{b\in\{0,1\}}
\cP_a^b(\bar\gt_1 = \ell -m) \cP_b(k-\ell \in \bar\gt).
\ea
\eeq
From the lower bound part, we have $\cP_b(k-\ell \in \bar\gt) \ge 1/(C\tone^3)$ for $b\in\{0,1\}$. Moreover, we have by~\eqref{eq:LB-cP-in} that for $a\in\{0,1\}$
\beq
\sum_{k-\tone^3 \le \ell \le k-\tone^2} \cP_a(\bar\gt_1 = \ell - m) \ge {\rm (cst)} e^{-(\phi - g(\tone))(k-m)}.
\eeq
Hence,
\beq
\label{eq:cP-LB-nonemptyband}
\cP_0(\bar\gt \cap [k-\tone^3, k-\tone^2] \neq \emptyset, k\in \bar\gt) \ge
\frac{{\rm (cst)}}{\tone^3}
\sum_{0\le m< k-\tone^3} \cP_0(m\in\bar\gt) \exp(-(\phi - g(\tone))(k-m)).
\eeq
Next we deal with
\beq
\label{eq:cP-emptyband}
\ba
&\cP_0(\bar\gt \cap [k-\tone^3, k- \tone^2] = \emptyset, k\in\bar\gt)\\
&=\sumtwo{m < k-\tone^3}{a\in\{0,1\}} \cP_0^a(m\in\bar\gt) \sumtwo{k-\tone^2\le \ell \le k}{b\in\{0,1\}}
\cP_a^b(\bar\gt_1 = \ell -m) \cP_b(k-\ell \in \bar\gt).
\ea
\eeq
We shall now prove that 
\beq
\label{eq:UB-cP-theta}
\cP_a(\bar\gt_1 = i) \le \frac{\rm (cst)}{\tone^3} \exp\Big([\phi - g(\tone)]i\Big),
\qquad i\ge \tone^3-\tone^2,
\qquad a\in\{0,1\}.
\eeq
The upper bound holds for $\cP_a(\bar\gt_1 = i, Y_1=\ins)$, by~\eqref{def:cP}, Lemma~\ref{lem:estim-qtn} and Lemma~\ref{eq:control.ratio.h}. As for the excursions outside of the main gap, we have, thanks to the inequality in Proposition~\ref{prop:roughUB},
\beq
\ba
\frac{\cP_a(\bar\gt_1 = i, Y_1=\out)}{\tone^{-3}\exp\Big([\phi - g(\tone)]i\Big)} &\le C(i\tone)^{3} \exp(-[\phi(\ttwo)- g(\tone)]i)\\
&\le C(i\tone)^{3} \exp\Big(-C\frac{\rho i}{\tone^2}\Big)\\
&\le C \tone^{9} \exp(-C\rho \tone),
\ea
\eeq
The second inequality is a consequence of~\eqref{eq:def:gt},~\eqref{eq:phi} and the fact that $\ttwo\le (1-\rho)\tone$.  The third inequality holds since $i\ge \tone^3 - \tone^2$.
By substituting~\eqref{eq:UB-cP-theta} into~\eqref{eq:cP-emptyband} and using that $0\le k-\ell \le \tone^2$, we obtain
\beq
\ba
&\cP(\bar\gt \cap [k-\tone^3, k- \tone^2] = \emptyset, k\in\bar\gt)\\ 
&\le 
\frac{C}{\tone^3}\sum_{0\le m < k-\tone^3} \cP(m\in\bar\gt) \exp\Big([\phi - g(\tone)](k-m)\Big)
\sum_{k-\tone^2 \le \ell \le k} \cP(k-\ell \in \bar\gt).
\ea
\eeq
By~\eqref{eq:UB-small-n} and since $k-\ell \le \tone^2$, we have $\cP(k-\ell\in\bar\gt) \le C(k-\ell)^{-3/2}$. Therefore,
\beq
\cP(\bar\gt \cap [k-\tone^3, k- \tone^2] = \emptyset, k\in\bar\gt) \le \frac{C}{\tone^3}\sum_{0\le m < k-\tone^3} \cP(m\in\bar\gt) \exp\Big([\phi - g(\tone)](k-m)\Big).
\eeq
Recalling~\eqref{eq:cP-LB-nonemptyband}, we get
\beq
\label{eq:equivalB13}
\cP(\bar\gt \cap [k-\tone^3, k- \tone^2] = \emptyset, k\in\bar\gt) \le {\rm (cst)} 
\cP(\bar\gt \cap [k-\tone^3, k- \tone^2] \neq \emptyset, k\in\bar\gt),
\eeq
which is what we needed.
\end{proof}

\begin{proof}[Proof of Lemma~\ref{lem:sharp-Zn-renorm}] We recall that the random walk is started from the lower boundary of the optimal gap in the definition of $\mZ_k$.
\par {\it Step 1.} Let us start with the lower bound. Suppose $k \geq 2 \tone^3$. We first decompose the partition function according to the last renewal point before $k$:
\beq
\label{eq:decompolast}
\mZ_k=\sum_{a\in\{0,1\}}\sum_{m=0}^k \mZ^{0a}_{k-m}(k-m\in \bar\gt)\mZ^a_m(\bar\gt_1> m).
\eeq
We readily obtain from~\eqref{eq:comp-Zpin}:
\beq
\ba
\mZ_k e^{\phi k}&\ge \frac1C \sum_{m=0}^k \cP(k-m\in \bar\gt)e^{\phi m}\min_{a\in\{0,1\}} \mZ^a_m(\bar\gt_1> m)\\
&\geq \frac1C \sum_{m=\tone^2}^{\tone^3} \cP(k-m\in \bar\gt)e^{\phi m}\min_{a\in\{0,1\}}\mZ^a_m(\bar\gt_1> m,Y_1=\ins).
\ea
\eeq
As $k-m\geq \tone^3$, $\cP(k-m\in \bar\gt)\geq 1/(C\tone^3)$, by~Lemma~\ref{lem:mass-ren-fct}. Using Lemma \ref{lem:estim-qtn}, as $m\geq \tone^2$ and $a\in\{0,1\}$,
\beq
\mZ^a_m(\bar\gt_1> m,Y_1=\ins)\geq \sum_{i>m}q_{\tone}(i)\geq  \sum_{i>m}\frac{c_1}{\tone^3 } e^{-g(\tone)i}\geq \frac{C}{g(\tone)\tone^3} e^{-g(\tone)m}\geq \frac{C}{\tone} e^{-g(\tone)m}.
\eeq
This leads to
\beq
\ba
\mZ_ke^{\phi k}&\geq \sum_{m=\tone^2}^{\tone^3} \frac{1}{C\tone^3} e^{\phi m}\times  \frac{C}{\tone} e^{-g(\tone)m}.
\ea
\eeq
Since $m\leq \tone^3$ and as a consequence of~\eqref{eq:sdwch-g-minus-phi}, $e^{(\phi-g(\tone)) m}\geq C$ and we obtain
\beq
\mZ_ke^{\phi k}\geq \frac{C}{\tone}.
\eeq
\par {\it Step 2.} Suppose $k\ge 2\tone^2$. For the upper bound, we come back to \eqref{eq:decompolast} and observe that
\beq
\mZ_ke^{\phi k} \le C\sum_{m=0}^k \cP(k-m\in \bar\gt)e^{\phi m} \max_{a\in\{0,1\}}\mZ^a_m(\bar\gt_1> m).
\eeq
Let $a\in\{0,1\}$. We decompose the last term as follows:
\beq
\mZ^a_m(\bar\gt_1> m)=\mZ^a_m(\bar\gt_1> m,Y_1=\ins)+\mZ^a_m(\bar\gt_1> m,Y_1=\out).
\eeq
\par {\it Step 2.1.} We first consider the part of the sum relative to $\{Y_1=\ins\}$. (i)
Let us first control the sum for $m\leq \tone^2$. In this case $k-m\geq \tone^2$, so that using Lemma~\ref{lem:mass-ren-fct}, $\cP(k-m\in \bar\gt)\leq C/\tone^3$ while $\mZ_m(\bar\gt_1> m,Y_1=\ins)\leq 1$ and $e^{\phi m}\leq C$ so that 
\beq
\label{eq:in1}
\sum_{m=0}^{\tone^2} \cP(k-m\in \bar\gt)e^{\phi m}\mZ_m(\bar\gt_1> m,Y_1=\ins) \leq C \sum_{m=0}^{\tone^2} \frac{c}{\tone^3}\leq \frac{C}{\tone}.
\eeq
\par (ii) We turn to the case when $\tone^2\leq m\leq k-\tone^2$. Here again, with the same argument, $\cP(k-m\in \bar\gt)\leq c/\tone^3$ while using Lemma \ref{lem:estim-qtn} and as $m\geq \tone^2$,
\beq
\label{eq:controleZm}
\mZ_m(\bar\gt_1> m,Y_1=\ins)\leq \sum_{i>m} q_{\tone}(i)\leq  \sum_{i>m} \frac{c_2}{\tone^3 } e^{-g(\tone)i}\leq \frac{C}{g(\tone)\tone^3} e^{-g(\tone)m}\leq \frac{C}{\tone} e^{-g(\tone)m}.
\eeq
Finally,
\beq
\sum_{m=\tone^2}^{k-\tone^2} \cP(k-m\in \bar\gt)e^{\phi m}\mZ_m(\bar\gt_1> m,Y_1=\ins) \leq C \sum_{m=\tone^2}^{k-\tone^2} \frac{1}{\tone^4} e^{-g(\tone)m}e^{\phi m}.
\eeq
Since $\phi-g(\tone)\leq - \frac{c}{\tone^3}$, we have 
\beq
\sum_{m=\tone^2}^{k-\tone^2} e^{(\phi-g(\tone))m}\leq \sum_{m\ge 0} e^{(\phi-g(\tone))m}\leq c \tone^3,
\eeq
which completes the control of the second part.
\par (iii) We turn to $k-\tone^2\leq m\leq k$: We use the same bounds as the ones in the previous part except that this time we use that $\cP(k-m\in \bar\gt)\leq \frac{C}{1+(k-m)^{3/2}}$ (by Lemma \ref{lem:mass-ren-fct}). This leads to
\beq
\sum_{m=k-\tone^2}^{k} \cP(k-m\in \bar\gt)e^{\phi m}\mZ_m(\bar\gt_1> m,Y_1=\ins) \leq C\sum_{m=k-\tone^2}^{k}\frac{1}{\tone} \frac{e^{(\phi-g(\tone))m}}{1+(k-m)^{3/2}}  \leq \frac{C}{\tone},
\eeq
where we used the fact that $\phi-g(\tone)\leq 0$ for the last inequality.
 \par {\it Step 2.2.} We now consider the part of the sum relative to $\{Y_1=\out\}$ and again split the sum in three parts.
(i) We manage with the part corresponding to $m\leq \tone^2$ exactly in the same way as \eqref{eq:in1}. (ii) We then consider the case $\tone^2\leq m\leq \frac{1}{2C_1} \tone^{\frac{4(\gga+2)}{\gga+4}}$, with $C_1$ as in~\eqref{eq:encadr_gphi}. By Lemma~\ref{lem:Zout}, we obtain
\beq
\mZ_m(\bar\gt_1> m,Y_1=\out)\leq C m^3 \exp(-m^{\frac{\gga}{2(\gga+2)}}).
\eeq
This leads to 
\beq
\ba
&\sum_{\tone^2\leq m\leq \frac{1}{2C_1} \tone^{\frac{4(\gga+2)}{\gga+4}}} \cP(k-m\in \bar\gt)e^{\phi m}\mZ_m(\bar\gt_1> m,Y_1=\out)\\
&\qquad \qquad \leq C\sum_{\tone^2\leq m\leq \frac{1}{2C_1} \tone^{\frac{4(\gga+2)}{\gga+4}}}  m^3 \exp(-m^{\frac{\gga}{2(\gga+2)}}+\phi m)\\
&\qquad \qquad \leq C\sum_{\tone^2\leq m\leq \frac{1}{2C_1} \tone^{\frac{4(\gga+2)}{\gga+4}}}  m^3 \exp(-\tfrac12 m^{\frac{\gga}{2(\gga+2)}}) \le {\rm (cst)}.
\ea
\eeq
(iii) We end the proof with the case $\frac{1}{2C_1} \tone^{\frac{4(\gga+2)}{\gga+4}}\le m \le k$. By Proposition~\ref{prop:roughUB}, we obtain
\beq
\mZ_m(\bar\gt_1> m,Y_1=\out)\leq C m^3 \exp(-\phi(\ttwo)m).
\eeq
This leads to 
\beq
\ba
&\sum_{\frac{1}{2C_1} \tone^{\frac{4(\gga+2)}{\gga+4}}\le m \le k} \cP(k-m\in \bar\gt)e^{\phi m}\mZ_m(\bar\gt_1> m,Y_1=\out)\\
&\qquad \qquad \leq C\sum_{\frac{1}{2C_1} \tone^{\frac{4(\gga+2)}{\gga+4}}\le m \le k}  m^3 \exp[(g(\tone)-\phi(\ttwo))m]\\
&\qquad \qquad \leq C\sum_{\frac{1}{2C_1} \tone^{\frac{4(\gga+2)}{\gga+4}}\le m \le k}  m^3 \exp\Big(-\frac{\rho m}{\tone^2}\Big)\\
&\qquad \qquad \leq C(\rho)\ \tone^{8}\exp(-\tfrac12\rho \tone^{\frac{2\gga}{\gga+4}}).
\ea
\eeq
\end{proof}
%%%%%%%%%%%%%%%%%%%%%%%%%%
\section{Proof of Lemma~\ref{lem:ctrl_RN-deriv}}
\label{sec:radon}
\begin{proof}[Proof of Lemma~\ref{lem:ctrl_RN-deriv}]
By Remark~\ref{rk:comparegapn}, note that $n\ge 2\tone^3$ for $n$ large enough, which allows us to use Lemma~\ref{lem:sharp-Zn-renorm}.
Using Lemma~\ref{lem:cP_tail} as well (note that $\chop$ is of order $\tone^2$ while $g(\tone)-\phi(\tone)$ is of order $1/\tone^3$) we get for $\chop\le k \le n$:
\beq
\label{eq:ctrl_RN-deriv1}
\Theta(n,k) \le C \tone \mZ_k(\bar \gt_1 > k-\chop) e^{g(\tone)k}.
\eeq 
Let us split the rest of the proof in two cases: either $\chop\le k \le \tone^3$ (first case) or $\tone^3 \le k\le n$ (second case).

\par {\it (First case)} Let us start with the case $\chop\le k \le \tone^3$. In this case there exists some constant $C$ such that $e^{g(\tone)k}\leq C e^{\phi(\tone)k}$, by\eqref{eq:def:gt} and~\eqref{eq:phi}. Therefore,
\beq
\Theta(n,k) \leq C \tone \mZ_ke^{\phi(\tone) k},
\eeq
and we directly conclude using Lemma \ref{lem:sharp-Zn-renorm} and the fact that $k\geq 2\tone^2$.
\par {\it (Second case)} We now consider the case $\tone^3 \le k\le n$. Let us split the right-hand side in~\eqref{eq:ctrl_RN-deriv1} as the sum of 
\beq
(a) = C \tone \mZ_k(\bar \gt_1 > k) e^{g(\tone)k}
\qquad
{\rm and}
\qquad
(b) = C \tone \mZ_k(k\ge \bar \gt_1 > k-\chop) e^{g(\tone)k}.
\eeq
We further decompose these two terms as follows:
\beq
\ba
(a1) &= C \tone \mZ_k(\bar \gt_1 > k, Y_1 = \ins) e^{g(\tone)k},\\
(a2) &= C \tone \mZ_k(\bar \gt_1 > k, Y_1 = \out) e^{g(\tone)k}
\ea
\eeq
and
\beq
\ba
(b1) &= C \tone \mZ_k(k\ge \bar \gt_1 > k-\chop, Y_1=\ins) e^{g(\tone)k},\\
(b2) &= C \tone \mZ_k(k\ge \bar \gt_1 > k-\chop, Y_1=\out) e^{g(\tone)k}.
\ea
\eeq
\par (i) Let us start with $(a1)$. From \eqref{eq:controleZm} and using that $k\geq \tone^2$,
\beq
\mZ_k(\bar \gt_1 > k, Y_1 = \ins)\leq \frac{C}{\tone} e^{-g(\tone)k}
\eeq
so that $(a1)\leq C$.

\par (ii) We turn to $(a2)$. We first consider the case $\tone^3 \leq k \leq \frac{1}{2C_1} \tone^{\frac{4(\gga+2)}{\gga+4}}$ (which does not exist if $\gga<4$). Using Lemma~\ref{lem:Zout}, we get
\beq
\mZ_k(\bar \gt_1 > k, Y_1 = \out)\leq {\rm (cst)}k^3e^{-k^{\frac{\gamma}{2(\gamma+2)}}}.
\eeq
so that 
\beq
(a2)\leq C \tone k^3e^{-k^{\frac{\gamma}{2(\gamma+2)}}} e^{g(\tone)k}\leq C \tone k^3 e^{-\frac12 k^{\frac{\gamma}{2(\gamma+2)}}}\le C\tone^{13}e^{-\frac12 \tone^{\frac{\gamma}{\gamma+2}}},
\eeq
Now, consider the case $k > \frac{1}{2C_1} \tone^{\frac{4(\gga+2)}{\gga+4}}$ and use this time
\beq
\mZ_k(\bar \gt_1 > k, Y_1 = \out)\leq {\rm (cst)}k^3e^{-\phi(\ttwo)k}.
\eeq
As $g(\tone)-\phi(\ttwo)\leq -\frac{\rho}{\ttwo^2}$ we obtain for $k \geq \frac{1}{2C_1} \tone^{\frac{4(\gga+2)}{\gga+4}}$ that
\beq
(a2)\leq C \tone k^3 \exp\Big(-\frac{\rho k}{\ttwo^2}\Big) \leq C\tone^{7}\exp\Big({-\rho \tone^{\frac{2\gamma}{\gamma+4}}}\Big).
\eeq

\par (iii) Let us now manage with the case $(b1)$ using Lemma \ref{lem:estim-qtn}. By decomposing on the value of $\bar \gt_1$, we get
\beq
(b1)\leq C \tone \sum_{i=0}^\chop q_{\tone}(k-i)\mZ_i e^{g(\tone)k}\leq C \tone \sum_{i=0}^\chop \frac{1}{\tone^3}e^{g(\tone)i} \mZ_i.
\eeq
We finally use the rough estimate $\mZ_i\leq 1$ and note that for $i\leq \chop$, $e^{g(\tone)i}$ is bounded, in order to get
\beq
(b1)\leq C \tone \times \frac{\chop}{\tone^3}\leq C . 
\eeq

\par (iv) We conclude with the control of $(b2)$, still in the case $k\geq \tone^3$. We first observe
\beq
(b2) = C \tone  \sum_{i=0}^\chop \mK_{\out}(k-i)\mZ_i e^{g(\tone)k}\leq C \tone e^{g(\tone)k} \sum_{i=0}^\chop (k-i)^3 e^{-\phi(\ttwo)(k-i)}.
\eeq
Here again we use that $g(\tone)-\phi(\ttwo)\leq -\frac{c(\rho)}{\tone^2}$ to obtain
\beq
(b2) \leq C \tone k^3 e^{-c(\rho) \frac{k}{\tone^2}} \sum_{i=0}^\chop  e^{\phi(\ttwo)i}.
\eeq
We bound $\sum_{i=0}^\chop  e^{\phi(\ttwo)i}\leq \chop e^{\phi(\ttwo)\chop}\leq C\chop$ as $\chop\leq 2\tone^2$ and $\ttwo\geq \rho \tone$, and finally
\beq
(b2) \leq C \tone^3 k^3 e^{-c(\rho) \frac{k}{\tone^2}} \leq  C \tone^9  e^{-c(\rho) \tone}. 
\eeq
\end{proof}

\end{appendix}

%%%%%% The bibliography %%%%%%%
\bibliographystyle{abbrv}
\bibliography{biblio_JulFra}

\begin{thebibliography}{10}

\bibitem{Antal95}
P.~Antal.
\newblock Enlargement of obstacles for the simple random walk.
\newblock {\em Ann. Probab.}, 23(3):1061--1101, 1995.

\bibitem{AubryAndre1980}
S.~Aubry and G.~Andr\'{e}.
\newblock Analyticity breaking and {A}nderson localization in incommensurate
  lattices.
\newblock In {\em Group theoretical methods in physics ({P}roc. {E}ighth
  {I}nternat. {C}olloq., {K}iryat {A}navim, 1979)}, volume~3 of {\em Ann.
  Israel Phys. Soc.}, pages 133--164. Hilger, Bristol, 1980.

\bibitem{AufLou11}
A.~Auffinger and O.~Louidor.
\newblock Directed polymers in a random environment with heavy tails.
\newblock {\em Comm. Pure Appl. Math.}, 64(2):183--204, 2011.

\bibitem{BAC07}
G.~Ben~Arous and J.~\v{C}ern\'{y}.
\newblock Scaling limit for trap models on {$\Bbb Z^d$}.
\newblock {\em Ann. Probab.}, 35(6):2356--2384, 2007.

\bibitem{BergerTorri19}
Q.~Berger and N.~Torri.
\newblock Directed polymers in heavy-tail random environment.
\newblock {\em Ann. Probab.}, 47(6):4024--4076, 2019.

\bibitem{BisKon01}
M.~Biskup and W.~K\"{o}nig.
\newblock Screening effect due to heavy lower tails in one-dimensional
  parabolic {A}nderson model.
\newblock {\em J. Statist. Phys.}, 102(5-6):1253--1270, 2001.

\bibitem{BKdS18}
M.~Biskup, W.~K\"{o}nig, and R.~S. dos Santos.
\newblock Mass concentration and aging in the parabolic {A}nderson model with
  doubly-exponential tails.
\newblock {\em Probab. Theory Related Fields}, 171(1-2):251--331, 2018.

\bibitem{MR3052403}
F.~Caravenna, P.~Carmona, and N.~P\'{e}tr\'{e}lis.
\newblock The discrete-time parabolic {A}nderson model with heavy-tailed
  potential.
\newblock {\em Ann. Inst. Henri Poincar\'{e} Probab. Stat.}, 48(4):1049--1080,
  2012.

\bibitem{CP09b}
F.~Caravenna and N.~P\'etr\'elis.
\newblock Depinning of a polymer in a multi-interface medium.
\newblock {\em Electron. J. Probab.}, 14:no. 70, 2038--2067, 2009.

\bibitem{CP09}
F.~Caravenna and N.~P\'etr\'elis.
\newblock A polymer in a multi-interface medium.
\newblock {\em Ann. Appl. Probab.}, 19(5):1803--1839, 2009.

\bibitem{CNPT18}
P.~Carmona, G.~B. Nguyen, N.~P\'{e}tr\'{e}lis, and N.~Torri.
\newblock Interacting partially directed self-avoiding walk: a probabilistic
  perspective.
\newblock {\em J. Phys. A}, 51(15):153001, 23, 2018.

\bibitem{MR3551159}
V.~Chulaevsky.
\newblock Non-perturbative {A}nderson localization in heavy-tailed potentials
  via large deviations moment analysis.
\newblock {\em J. Math. Phys.}, 57(9):093506, 18, 2016.

\bibitem{Comets-book}
F.~Comets.
\newblock {\em Directed polymers in random environments}, volume 2175 of {\em
  Lecture Notes in Mathematics}.
\newblock Springer, Cham, 2017.
\newblock Lecture notes from the 46th Probability Summer School held in
  Saint-Flour, 2016.

\bibitem{Croy2011}
A.~Croy, P.~Cain, and M.~Schreiber.
\newblock Anderson localization in 1{D} systems with correlated disorder.
\newblock {\em Eur. Phys. J. B}, 82(2):107--112, 2011.

\bibitem{CroyMuir16}
D.~Croydon and S.~Muirhead.
\newblock Quenched localisation in the {B}ouchaud trap model with regularly
  varying traps.
\newblock In {\em Stochastic analysis on large scale interacting systems}, RIMS
  K\^{o}ky\^{u}roku Bessatsu, B59, pages 305--320. Res. Inst. Math. Sci.
  (RIMS), Kyoto, 2016.

\bibitem{MR3572330}
P.~S. Dey and N.~Zygouras.
\newblock High temperature limits for {$(1+1)$}-dimensional directed polymer
  with heavy-tailed disorder.
\newblock {\em Ann. Probab.}, 44(6):4006--4048, 2016.

\bibitem{ding2019distribution}
J.~Ding, R.~Fukushima, R.~Sun, and C.~Xu.
\newblock Distribution of the random walk conditioned on survival among
  quenched {B}ernoulli obstacles.
\newblock {\em Ann. Probab.}, 49(1):206--243, 2021.

\bibitem{DIngXu2019}
J.~Ding and C.~Xu.
\newblock Poly-logarithmic localization for random walks among random
  obstacles.
\newblock {\em Ann. Probab.}, 47(4):2011--2048, 2019.

\bibitem{DingXu2019LOC}
J.~Ding and C.~Xu.
\newblock Localization for random walks among random obstacles in a single
  {E}uclidean ball.
\newblock {\em Comm. Math. Phys.}, 375(2):949--1001, 2020.

\bibitem{Feller-vol1}
W.~Feller.
\newblock {\em An introduction to probability theory and its applications.
  {V}ol. {I}}.
\newblock Third edition. John Wiley \& Sons, Inc., New York-London-Sydney,
  1968.

\bibitem{MR3835474}
F.~Flegel.
\newblock Localization of the principal {D}irichlet eigenvector in the
  heavy-tailed random conductance model.
\newblock {\em Electron. J. Probab.}, 23:Paper No. 68, 43, 2018.

\bibitem{Fuk2009}
R.~Fukushima.
\newblock From the {L}ifshitz tail to the quenched survival asymptotics in the
  trapping problem.
\newblock {\em Electron. Commun. Probab.}, 14:435--446, 2009.

\bibitem{Gia07}
G.~Giacomin.
\newblock {\em Random polymer models}.
\newblock Imperial College Press, London, 2007.

\bibitem{Gia11}
G.~Giacomin.
\newblock {\em Disorder and critical phenomena through basic probability
  models}, volume 2025 of {\em Lecture Notes in Mathematics}.
\newblock Springer, Heidelberg, 2011.
\newblock Lecture notes from the 40th Probability Summer School held in
  Saint-Flour, 2010, \'Ecole d'\'Et\'e de Probabilit\'es de Saint-Flour.
  [Saint-Flour Probability Summer School].

\bibitem{MR3491341}
T.~Gueudre, P.~Le~Doussal, J.-P. Bouchaud, and A.~Rosso.
\newblock Ground-state statistics of directed polymers with heavy-tailed
  disorder.
\newblock {\em Phys. Rev. E (3)}, 91(6):062110, 10, 2015.

\bibitem{MR3736656}
J.~Huang, K.~L\^{e}, and D.~Nualart.
\newblock Large time asymptotics for the parabolic {A}nderson model driven by
  space and time correlated noise.
\newblock {\em Stoch. Partial Differ. Equ. Anal. Comput.}, 5(4):614--651, 2017.

\bibitem{MR3689969}
J.~Huang, K.~L\^{e}, and D.~Nualart.
\newblock Large time asymptotics for the parabolic {A}nderson model driven by
  spatially correlated noise.
\newblock {\em Ann. Inst. Henri Poincar\'{e} Probab. Stat.}, 53(3):1305--1340,
  2017.

\bibitem{Ko16}
W.~K\"onig.
\newblock {\em The parabolic {A}nderson model}.
\newblock Pathways in Mathematics. Birkh\"auser/Springer, 2016.
\newblock Random walk in random potential.

\bibitem{KLMS09}
W.~K\"{o}nig, H.~Lacoin, P.~M\"{o}rters, and N.~Sidorova.
\newblock A two cities theorem for the parabolic {A}nderson model.
\newblock {\em Ann. Probab.}, 37(1):347--392, 2009.

\bibitem{konig2023weakly}
W.~K{\"o}nig, N.~P{\'e}tr{\'e}lis, R.~S. dos Santos, and W.~van Zuijlen.
\newblock Weakly self-avoiding walk in a pareto-distributed random potential,
  2023.

\bibitem{Lacoin11}
H.~Lacoin.
\newblock Influence of spatial correlation for directed polymers.
\newblock {\em Ann. Probab.}, 39(1):139--175, 2011.

\bibitem{Lacoin2012-I}
H.~Lacoin.
\newblock Superdiffusivity for {B}rownian motion in a {P}oissonian potential
  with long range correlation: {I}: {L}ower bound on the volume exponent.
\newblock {\em Ann. Inst. Henri Poincar\'{e} Probab. Stat.}, 48(4):1010--1028,
  2012.

\bibitem{Lacoin2012-II}
H.~Lacoin.
\newblock Superdiffusivity for {B}rownian motion in a {P}oissonian potential
  with long range correlation {II}: {U}pper bound on the volume exponent.
\newblock {\em Ann. Inst. Henri Poincar\'{e} Probab. Stat.}, 48(4):1029--1048,
  2012.

\bibitem{Lyu2020}
Y.~Lyu.
\newblock Precise high moment asymptotics for parabolic {A}nderson model with
  log-correlated {G}aussian field.
\newblock {\em Statist. Probab. Lett.}, 158:108662, 12, 2020.

\bibitem{MR2883854}
P.~M\"{o}rters.
\newblock The parabolic {A}nderson model with heavy-tailed potential.
\newblock In {\em Surveys in stochastic processes}, EMS Ser. Congr. Rep., pages
  67--85. Eur. Math. Soc., Z\"{u}rich, 2011.

\bibitem{Poisat2012}
J.~Poisat.
\newblock Random pinning model with finite range correlations: disorder
  relevant regime.
\newblock {\em Stochastic Process. Appl.}, 122(10):3560--3579, 2012.

\bibitem{Poisat2013b}
J.~Poisat.
\newblock On quenched and annealed critical curves of random pinning model with
  finite range correlations.
\newblock {\em Ann. Inst. Henri Poincar\'{e} Probab. Stat.}, 49(2):456--482,
  2013.

\bibitem{PoiSim19}
J.~Poisat and F.~Simenhaus.
\newblock A limit theorem for the survival probability of a simple random walk
  among power-law renewal obstacles.
\newblock {\em Ann. Appl. Probab.}, 30(5):2030--2068, 2020.

\bibitem{Rang2020}
G.~Rang.
\newblock From directed polymers in spatial-correlated environment to
  stochastic heat equations driven by fractional noise in {$1 + 1$} dimensions.
\newblock {\em Stochastic Process. Appl.}, 130(6):3408--3444, 2020.

\bibitem{Res}
S.~I. Resnick.
\newblock {\em Extreme values, regular variation and point processes}.
\newblock Springer Series in Operations Research and Financial Engineering.
  Springer, New York, 2008.
\newblock Reprint of the 1987 original.

\bibitem{S06}
E.~Seneta.
\newblock {\em Non-negative matrices and {M}arkov chains}.
\newblock Springer Series in Statistics. Springer, New York, 2006.
\newblock Revised reprint of the second (1981) edition [Springer-Verlag, New
  York; MR0719544].

\bibitem{Sohier_2013}
J.~Sohier.
\newblock The scaling limits of a heavy tailed {M}arkov renewal process.
\newblock {\em Ann. Inst. Henri Poincar\'{e} Probab. Stat.}, 49(2):483--505,
  2013.

\bibitem{Sznitman93-ptrf}
A.-S. Sznitman.
\newblock Brownian asymptotics in a {P}oissonian environment.
\newblock {\em Probab. Theory Related Fields}, 95(2):155--174, 1993.

\bibitem{Sz98}
A.-S. Sznitman.
\newblock {\em Brownian motion, obstacles and random media}.
\newblock Springer Monographs in Mathematics. Springer-Verlag, Berlin, 1998.

\bibitem{Sznit01}
A.-S. Sznitman.
\newblock Milieux al\'{e}atoires et petites valeurs propres.
\newblock In {\em Milieux al\'{e}atoires}, volume~12 of {\em Panor.
  Synth{\`e}ses}, pages 13--36. Soc. Math. France, Paris, 2001.

\bibitem{tessieri2015}
L.~Tessieri, I.~F. Herrera-Gonz\'{a}lez, and F.~M. Izrailev.
\newblock The band-centre anomaly in the 1{D} {A}nderson model with correlated
  disorder.
\newblock {\em J. Phys. A}, 48(35):355001, 30, 2015.

\bibitem{MR2474543}
R.~van~der Hofstad, P.~M\"{o}rters, and N.~Sidorova.
\newblock Weak and almost sure limits for the parabolic {A}nderson model with
  heavy tailed potentials.
\newblock {\em Ann. Appl. Probab.}, 18(6):2450--2494, 2008.

\bibitem{Viveros21}
R.~Viveros.
\newblock Directed polymer in {$\gamma$}-stable random environments.
\newblock {\em Ann. Inst. Henri Poincar\'{e} Probab. Stat.}, 57(2):1081--1102,
  2021.

\end{thebibliography}

\end{document}